%% file: LPCPstep2.tex
\documentclass[monochrome,12pt]{article}
\usepackage{amssymb, latexsym, amsmath, amsthm, amsfonts, amssymb, paralist, extpfeil, mathrsfs}
\usepackage{cases}
\allowdisplaybreaks
\usepackage[title,titletoc]{appendix}
\usepackage{yfonts}
\usepackage[dvipsnames,svgnames,x11names,table]{xcolor}
\usepackage{color}
\usepackage{oplotsymbl}
\usepackage{graphicx,overpic,booktabs,makecell}
\usepackage{hyperref}
\hypersetup{
	CJKbookmarks=true,
	bookmarksopen=true,
	pdfstartview=FitH,
	colorlinks=true,
	linkcolor=blue,
	filecolor=blue,
	urlcolor=blue,
	citecolor=blue,
	bookmarksnumbered=true,
}

\usepackage{geometry}
\usepackage{enumerate}

\usepackage{xifthen}
\usepackage[numbers,sort&compress]{natbib}
\renewcommand{\cite}[2][]{\ifthenelse{\isempty{#1}}{\citep{#2}}{\citetext{\citenum{#2}, #1}}}

\input{MyOperators.tex}

\newtheorem{theorem}{Theorem}[section]
\newtheorem{proposition}[theorem]{Proposition}
\newtheorem{lemma}[theorem]{Lemma}

\theoremstyle{definition}

\newtheorem{remark}[theorem]{Remark}
\newtheorem{example}[theorem]{Example}

\numberwithin{equation}{section}

\title{\bf Lattice point counting problems on step-two nilpotent Lie groups}
\author{Sheng-Chen Mao\footnote{Corresponding author. \texttt{maoshengchen@lzu.edu.cn}.}  \date{}
}

\begin{document}
\renewcommand{\theequation}{\thesection.\arabic{equation}}
\setcounter{equation}{0} \maketitle
\vspace{-1.0cm}
\bigskip
	
{\noindent\bf Abstract.}
We develop the  theory of lattice point counting on connected and simply connected  nilpotent Lie groups of step-two, endowed with the parabolic type dilation and a family of  homogeneous norms $
\mathcal{N}_{\alpha,M}(x, t)=\left(|M_1x|^\alpha + |M_2t|^{\alpha / 2}\right)^{1 / \alpha}$ adapted to the dilation structure, where $\alpha>0$ and $M_1,M_2$ are invertible matrices. With appropriate notions of lattices, the domains to be counted   are balls associated to these norms, and explicit counting discrepancy estimates are deduced for all possible dimensions and all $\alpha>0$. The bounds are sharp  when the  group center is unidimensional and $\alpha=2$, in certain rational sense. Our study also generalizes and even quantitatively improves previous results on Heisenberg groups  obtained  by Garg--Nevo--Taylor \cite[\textit{Ann. Inst. Fourier}, 2015]{GNT15}: 
(i) In  dimension $5$, the exponent of logarithmic factor  is lowered from $2/3$ to ${1}/{3}$ if $\alpha \in(3,4] $ or $\alpha=1$; and the  factor $\log ^{2/3} R $ is dropped if $\alpha\in(2,3]$. (ii) In  dimension $3$, the estimation is upgraded from $O_\epsilon(R^{ 5/2+\epsilon})$  to $O(R^{2}\log^{ 1/2} R)$ for $\alpha=1$, and  to $O(R^{{19}/{8}})$ for  $\alpha\in (1,2)$; and  the  factor $\log R$ is removed for $\alpha>4$. Moreover, as a byproduct, we extend the lattice counting near Heisenberg spheres, recently considered by Campolongo--Taylor \cite[\textit{Matematica}, 2023]{CT23} and Srivastava--Taylor \cite[\textit{J. Fourier Anal. Appl.}, 2026]{ST26}, to the above step-two group setting with arbitrary dimensional group center, where some quantitative improvements are also attained.  Our method relies upon  Poisson's summation formulas, oscillatory integral estimates and asymptotic properties as well as recursion formulas of Bessel functions.  

	
	
\medskip
\tableofcontents
\medskip
	
\section{Introduction} \label{intro}
The problem of lattice point counting  is an important branch of analytic number theory, with a rich history of over 200 years; see  \cite[]{IKKN06,BKZ18}.  Given an Euclidean ball $\cB_n(r)$ of radius $r>0$ in $\rr^n$, one hopes to approximate the discrete quantity --- the number of integer points contained in the ball, by the continuous quantity --- its volume, for large  $r$, where the difficulty lies in determining the correct order of the error term (called the \textit{discrepancy}).
More specifically, we aim to establish the following asymptotic 
\begin{equation} \label{ee01} 
    \#\left(\zz^n \cap  r\, \cB_n(1)\right) -r^n \vol(\cB_n(1)) =O_\tz(r^\theta), \qquad \mathrm{as}\,\,\, r\to\infty,
\end{equation}
with $\tz<n$ as small as possible, 
where $\cB_n(1)$ denotes the standard unit closed ball and the infimum of all such $\tz$ is called the \textit{correct order}. This goal has been achieved for dimension  $n\ge4$, with the correct order equal to $n-2$. For the notorious cases of dimensions 3 and 2,  also known as the sphere
problem and Gauss circle problem, the solutions remain wide open; and the  correct orders are conjectured to be $1$ and $1/2$, respectively.  The best available exponents are $21/16$ for $n=3$ by D.R. Heath-Brown  \cite[]{Hea99}, and $0.628966\cdots$ for $n=2$ by Li--Yang  \cite[]{LY23a}.

It is natural to study  similar problems for general domains $\Oz\subset\rr^n$ in \eqref{ee01} besides balls $\cB_n(1) $, under suitable conditions. Extensive researches  have been done in this direction, and it continues to be an active area to this day. The reader is referred to  \cite[]{Kra00,Kra88,Fri82,Hux96,IKKN06,Wal57,Got04} for the developments. We mention that  when $\Oz $ is compact convex  with smooth boundary and nonzero Gaussian
curvature everywhere,  the problem is   well understood relatively; see  \cite[]{Guo12,Hla50,IKKN06}. However, it is not yet solved for any $n\ge4$, where the conjectured correct order coincides with the  ball case, i.e., $n-2$. Regarding the case that the curvature may vanish, which includes $l^{p}$-norm ball ($p\in(2,\infty)$)  as a canonical example, see e.g. \cite[]{IKKN06, Ran66a,BCG+20}. For non-convex domain such as $l^{p}$-norm ball ($p\in(0,1)$), see e.g. the work of W.-G. Nowak \cite[]{Now80,Now81,Now83}. There are also many other extensions of \eqref{ee01} with deep insights on the underlying groups and their lattices as well as the domains to be counted. See e.g. \cite[]{GN10,GN12,BL24} and the  references therein for lattice counting on topological groups. However, in such framework   the correct order has not been  determined even for any simple Lie group yet,  to our knowledge. On the other hand, the research on nilpotent case seems rather limited, being restricted exclusively to the Heisenberg group \cite[]{GNT15,Gat20,CT23,ST26}, etc, which is usually recognized as the simplest example of non-Abelian and non-compact Lie groups.

Nilpotent Lie groups constitute a basic subject of contemporary mathematical research, with connections to   differential geometry (including sub-Riemannian geometry), complex analysis, non-commutative harmonic analysis and PDEs; see \cite[]{Le25,Goo76,BLU07,FS82,FR16} and the references therein. This context, particularly  the step-two case, has received much attention very recently; see e.g.  \cite[]{Sri24,LY23,RS24,MZ25,DNGR25}. In addition, nilpotent Lie groups and their discrete subgroups are significant in the investigation of semisimple Lie groups and arithmetic subgroups of linear semisimple Lie groups, cf.  \cite[]{Ji08},   which also can be traced back to many number-theoretic problems, especially in the modular function theory and the reduction theory of quadratic forms; cf. e.g.  \cite[]{Mar91,Rag72,PRR23}.

In this paper, we endeavor  to develop the lattice counting theory in nilpotent Lie groups,   focusing mainly on the step-two class  equipped with a  homogeneous structure of parabolic type. As shall be seen, this setting affords a clear analogue of the classical lattice point counting problem, and makes it possible to achieve a complete understanding of the distribution of lattice points within a non-Euclidean frame, in the sense of obtaining the  optimal discrepancy exponent, as in some known cases on Heisenberg groups.

Recall the Heisenberg group $\hh^d \cong \rr^{2d} \times \rr$ is defined by the group law
\begin{equation} \label{Hd} 
(x,t) \circ (x',t') = \left(x + x', t + t' + \langle  \jj x, x^{\prime} \rangle \right)\quad  \mbox{with}\quad \jj:= 2  \left(\begin{array}{cc}
0 & \ii_d \\
-\ii_d & 0
\end{array}\right).
\end{equation}
In 2015, by combining the Poisson summation formula, van der Corput lemma and the asymptotics of Bessel functions, Garg, Nevo and Taylor \cite[]{GNT15} examined an analogue of  \eqref{ee01} on $\hh^d$ and arrived at that, for all  $R\ge10$,
\begin{equation} \label{ee02} 
    \left|\#\left(\mathbb{Z}^{2 d+1} \cap \mathscr{B}_R^{\az, A}\right)-\vol\left(\mathscr{B}_1^{\az, A}\right) R^{2d+2}\right| \lesssim 
    \begin{cases}
    R^{2 d}, &\mbox{as}\ d\ge1,\,\az=2; \\
    R^{2 } \log R, &\mbox{as}\ d=1,\,2<\az\le4; \\
    R^{2 + \frac{2\az-8}{3\az-4}}\log R, &\mbox{as}\ d=1,\,\az>4, \\
    \end{cases} 
\end{equation}
where $\mathscr{B}_R^{\az, A}:=\{(x,t)\in\hh^d: |x|^\az +A|t|^{\az/2} \le R^\az \}$ with $\az,A>0$. Moreover, from a slicing argument  and the Euler-MacLaurin  formula, they also obtained that, for any $\az>0$,
\begin{equation} \label{ee03} 
    \left|\#\left(\mathbb{Z}^{2 d+1} \cap \mathscr{B}_R^{\az, A}\right)-\vol\left(\mathscr{B}_1^{\az, A}\right) R^{2d+2}\right| \lesssim 
    \begin{cases}
    R^{2 d}, &\mbox{as}\ d\ge3; \\
    R^{4 } \log^{\frac23} R, &\mbox{as}\ d=2; \\
    R^{\frac52+\ez} , &\mbox{as}\ d=1, \\
    \end{cases} 
\end{equation}
where the last estimate  used  implicitly the conjectured exponent for Gauss's circle problem. Note that when $d=1$,  the second estimate  \eqref{ee03} is better than  \eqref{ee02} as $\az>12$; while  it is inferior to \eqref{ee02} as $\az\in(4,12]$. Fix $A=1$ now. The bound in  \eqref{ee02} is shown by   \cite[]{GNT15}   to be optimal as $\az=2$ for any $d\ge1$. In the Cygan--Kor\'anyi ball case  (i.e. as $\az=4$) the task of finding the correct order turns out to be more  intricate, and  has been conducted by Y.A. Gath alone. In 2020, Gath \cite[]{Gat20} proved that the order $2$  in  \eqref{ee02} is the correct one for $d=1$ by means of the Mellin transforms and Riesz mean estimates of  \cite[]{CN61}; later in 2022, for $d\ge3$ Gath  \cite[]{Gat22} demonstrated that the correct order belongs to $[{2d-1},{2d-2/3}]$,  enhancing the exponent $2d$ in  \eqref{ee03}  substantially, and from a sharp  second moment estimate he conjectured that $2d-1$ should be the one. Recently, the author and Yang \cite[]{MY26a} further proved that the correct order cannot  exceed $2d-512/753$ for $d\ge4$. Gath's method \cite[]{Gat22}  is based on  a delicate slicing argument, reductions to  weighted integer lattice counting estimates in Euclidean domains, and technical tools  in  estimating exponential sums from analytic number theory. 
As to $\hh^2$, there is no further progress  beyond the initial result of \cite[]{GNT15}. Notwithstanding, in this work we shall make the first improvement by removing the logarithmic factor. Moreover, refinements for some other cases are also acquired and  listed in Remark  \ref{rem0}. 

It is also of independent interest  to study the distribution of lattices near the hypersurface in nilpotent groups. Recently, Campolongo and Taylor \cite[]{Cam22,CT23} have inspected this topic for integer points near Heisenberg spheres, then Srivastava and Taylor   \cite[]{ST26} considered an average version along with the group translation. For the Euclidean counterparts, see  e.g. \cite[]{BGHM22,IT11,Let10,HS25}. As a byproduct of our results, the above regime of  \cite[]{Cam22,CT23,ST26}   can be extended to the step-two group setting; see Theorem  \ref{lpcpS}. The detailed comparison between their results and ours are provided in Remark  \ref{rem2} below. For other pertinent results concerning lattice point counting  on Heisenberg groups, the reader can consult e.g. \cite[]{MY26, BMM26,Gat22a, Gat24a}.

\subsection{Notation}  \label{notat}

We denote the set of natural numbers by $\nn:=\{0,1,2,\ldots\}$, and put $\nn^*=\nn\setminus\{0\}$. The symbol $\pm$ appearing in the summation serves  to avoid repetition; e.g., $    \sum_\pm f_\pm h_\pm:=f_+h_+ + f_- h_-.$ Given a set $\Oz$,  we use $\#(\Oz)$ to denote its cardinality and $\chi_\Oz$ for the associated characteristic function. The notation $\one\{E\}$   stands for  the  indicator function of an event $E$, which equals $1$ if $E$ happens and $0$ otherwise. We utilize $\ii_n$  to signify the $n$-th order identity matrix. The transpose of a matrix  is written as $\cdot\trp.$ 

This paper employs the usual asymptotic notation. The letter $C$ and its possible variants are adopted to represent implicit positive constants which may change  between lines. Let $w$ be a non-negative function.   We mean $|f| \leq C w$ by $f = O(w)$, and $f \lesssim w$ (resp. $f \gtrsim w$) by $f \leq C w$ (resp. $f \ge C\, w$) when $f$ is also non-negative. If both $f\lsim g$ and $f\gsim g$  hold, we will write $f\sim g$ for short.  A subscript on which parameters these constants depend shall be specified (if needed) like $O_{\gz},\lesssim_{\gz},\sim_{\gz}$. 

\subsection{Step-two nilpotent Lie groups}  \label{S2G}
Let us first recall some basic facts of step-two  nilpotent Lie groups; for more details  the readers are referred to e.g. \cite[]{BLU07,FS82,FR16}. Suppose that $q,m\in\nn^*$, and that $\bbg(q,m)$ is a connected and simply connected Lie group whose real  Lie algebra  $\mathfrak{g}$ with Lie bracket $[\cdot,\cdot]$ obeys that
\begin{align*}
\mathfrak{g} = \mathfrak{g}_1 \oplus \mathfrak{g}_2, \qquad \{0\}\neq[\fg_1,\fg_1]\subset \mathfrak{g}_2\subset \mathfrak{z}(\mathfrak{g}),
\end{align*}
where  
$\mathfrak{z}(\mathfrak{g})$ is the center of $\mathfrak{g} $ and $\dim \fg_1=q,\, \dim\fg_2=m$. We  endow   $ \mathfrak{g}$ with a natural parabolic dilation structure $\{\dz_r\}_{r>0}$ as follows:
\begin{equation} \label{dial} 
    \dz_r(v_1,v_2):=(r \, v_1,r^2\,v_2), \quad \forall\, v_1\in\mathfrak{g}_1, v_2\in\mathfrak{g}_2, r>0,
\end{equation}
such that each $\dz_r$ is a morphism of the Lie algebra $\fg$.
This makes $\bbg(q,m)$ a homogeneous group of step 2. One then can identify $\bbg(q,m)$ and $\mathfrak{g}$ via the Lie group exponential map. In this way, $\bbg(q,m)$ can be viewed as $\rr^q \times \rr^m$ with the following group law
\begin{align}\label{gst}
(x , t) \circ  (x^{\prime}, t^{\prime}) :=
\left(x + x^{\prime}, t + t^{\prime} + \frac12 \langle  \uu x, x^{\prime} \rangle \right), \quad (x, t),(x',t') \in \rr^q \times \rr^m,
\end{align}
where
$
\langle\uu x,  x^{\prime} \rangle := (\langle U^{(1)} x, x^{\prime} \rangle, \ldots, \langle   U^{(m)} x,x' \rangle) \in \rr^m,
$
with $m$ real $q$-th order matrices $U^{(j)}$, and $\langle\cdot,\cdot\rangle$ is the standard Euclidean inner product. In this coordinate, the identity coincides with the origin, and the inverse of $(x,t)\in\bbg(q,m)$ equals $(x,t)^{-1}=(-x,-t+\frac12 \langle  \uu x, x \rangle)$.   The above dilation  \eqref{dial}  on $\mathfrak{g}$ induces the corresponding automorphism $\{\bdz_r\}_{r>0}$ on $\bbg(q,m)$:
$$\bdz_r(x,t):=(rx,r^2t), \qquad \forall\,(x,t)\in\bbg(q,m),r>0.$$
Under \eqref{gst},  a  basis of $\fg$ (called the \textit{Jacobian basis}) is given by (see e.g. \cite[\S 3.2]{BLU07})
\begin{equation} \label{leftV} 
\begin{gathered}
\mathrm{X}_i :=\frac{\partial}{\partial x_i}+\frac{1}{2} \sum_{k=1}^m\left(\sum_{j=1}^q U_{i, j}^{(k)} x_j\right) \frac{\partial}{\partial t_k},  \quad
\mathrm{~T}_{l} :=\frac{\partial}{\partial t_{l}}, \qquad \mbox{for}\ i\le q \ \mbox{and}\ l \leq m .
\end{gathered}
\end{equation}
Now we fix the (bi-invariant) Haar measure $\vol$ on $\bbg(q,m)$ lifted by the exponential map. Here are some examples of step-two groups.
\begin{example} \label{exam} 
\begin{enumerate}
\item \textit{Heisenberg groups}. Taking $q=2d,\, m=1$ and $\uu:=2\,\jj$ in \eqref{gst} reverts the well-known Heisenberg group $\hh^d$ defined by \eqref{Hd}. We also recall the polarized Heisenberg group $\hh^d_\pol$ with $\uu$ specified as
$\jj_\pol:=2\left(\begin{array}{cc}
0 & 0 \\
\ii_d & 0
\end{array}\right)$. The isomorphism between $\hh^d$ and  $\hh^d_\pol$ is given by
\begin{equation*}
    (x,t)\mapsto \left((x'',x'),\ \frac{t}{4} +\frac{1}{2}\langle  x', x'' \rangle\right),\ \mbox{where}\ x:=(x',x'')\in\rr^d\times\rr^d.
\end{equation*}

\item \textit{Heisenberg-type groups $\hh(2n,m)$}. Let $q=2n$ and $\uu$ satisfy the following condition:
\begin{itemize}
\item $U^{(j)}$ is skew-symmetric and orthogonal for any $1\le j \leq m$;
\item $U^{(i)} U^{(j)}=-U^{(j)} U^{(i)}$ for each $i, j \in\{1, \ldots, m\}$ with $i \neq j$.
\end{itemize}
Then we obtain the  Heisenberg-type groups $\hh(2n,m)$, which contain the class of Heisenberg groups up to a natural isomorphism. 
More discussions on $\hh(2n,m)$ can be found in  \cite[\S 3.6 and \S 18]{BLU07}.

\item \textit{Free step-two  Carnot groups $\bbf_{q,2}$}. Suppose that $i, j\in \{1, \ldots,q\}$ is fixed with $i > j$, and  $S(i,j)$ is the $q$-th order skew-symmetric matrix whose entries are $-1$ in the $(i,j)$ position, $+1$ in the $(j, i)$ position and 0 elsewhere. We set $m:=q(q-1)/2$, which equals exactly the number of such $S(i,j)$. Then the free step-two  Carnot group $\bbf_{q,2}$ is formed by taking $U^{(k)} (1\le k\le m)$ in  \eqref{gst} to be these
$m$ matrices. The Lie algebra of $\bbf_{q,2}$ is  recognized as the most non-Abelian as possible, from which the attribute  ``free" comes. See  \cite[\S 3.3 and \S 14]{BLU07}  for more details. 
\end{enumerate}
\end{example}

Throughout this paper, the number $Q:=q+2m$ will be designated as  the \textit{homogeneous dimension} of $\bbg(q,m)$, whose  topological dimension readily equals $q+m$. We mention that the nilpotency of  $\bbg(q,m)$ forces  $q\ge2$, as can be seen from  \eqref{gst}.

Suppose $\az>0$ and $M:=\diag(M_1,M_2)$, where $M_1$ and $M_2$ are $q$-th and $m$-th order invertible matrices respectively. The homogeneous norm $ \cN_{\alpha, M}$ on $\bbg(q,m)$ is defined by
\begin{equation}  \label{norm} 
\mathcal{N}_{\alpha,M}(x, t)=\left(|M_1x|^\alpha + |M_2t|^{\alpha / 2}\right)^{1 / \alpha}, \quad \forall\,(x, t) \in \mathbb{G}(q,m) ,
\end{equation}
which is $\bdz_r$-homogeneous of degree 1, that is,
\begin{equation*}
    \cN_{\az, M}\left(\bdz_r(x, t)\right)=r \,\cN_{\az, M}(x, t), \qquad \forall(x, t) \in \mathbb{G}(q,m),\ r>0 .
\end{equation*}
It is  a quasi-norm which fits the  inequality (cf. e.g.  \cite[Proposition 3.1.38]{FR16})
\begin{equation} \label{tri} 
    \mathcal{N}_{\alpha,M}((x , t) \circ  (x^{\prime}, t^{\prime})) \leq C_0\, (\mathcal{N}_{\alpha,M}(x , t)+  \mathcal{N}_{\alpha,M}(x' , t'))
\end{equation} 
for some  constant $C_0>0$. Nevertheless, any homogeneous Lie group can admit a quasi-norm 
which  satisfies the triangle inequality; see   \cite[Theorem 3.1.39]{FR16} or  \cite[THEOREM 1]{HS90}. In the case of Heisenberg-type group $\hh(2n,m)$, the celebrated  Cygan--Kor\'anyi norm, i.e., when $M_1=\ii_{2n},\, M_2=4\,\ii_m$ and $\az=4$ in  \eqref{norm}, serves as a typical example; cf. e.g.  \cite[p. 705]{BLU07}. We denote the \textit{$\mathcal{N}_{\alpha,M}$-ball} of radius $R>0$ by 
$$B_R^{\az,M}:=\{(x,t)\in\bbg(q,m): \cN_{\az,M}(x, t)\le R\}.$$
Note that
\begin{equation} \label{ballA} 
    \vol\left(B_R^{\az,M}\right) = \vol\left(B_1^{\az,M}\right)R^Q, \quad \fall R>0.
\end{equation} 

Let $\bbg$ be a simply connected nilpotent Lie group. A discrete subgroup $\Gz$ of $\bbg$ is called cocompact (resp. cofinite) if $\Gz\bsl\bbg$ is compact (resp. of finite invariant measure). In fact, the two definitions are equivalent (see  \cite[Theorem 2.1]{Rag72}) in this case.  Subsequently, a cofinite  lattice  will be simply referred to as the lattice. According to Mal'cev's criterion  (see \cite[]{Mal49}, and  also  \cite[Theorem 2.12]{Rag72}), $\bbg$ admits a lattice if and only if its Lie algebra admits a  
basis carrying rational structure constants, that is, $[w_i,w_j]=\sum_{k=1}^{\dim\bbg} c^k_{ij}w_k$ with rational $ c^k_{ij}$, for any $ w_i,w_j$ belong to this basis. Notice that not every nilpotent group can admit a lattice; for a counterexample, see  \cite[Remark 2.14]{Rag72}. On the other hand, as to our group $\bbg(q,m)$, from  \eqref{leftV} and a simple evaluation it follows that
\begin{equation*}
    \left[\X_i, \X_j\right]=  \sum_{k=1}^m \frac{1}{2} \left(U_{j, i}^{(k)}-U_{i, j}^{(k)}\right) \T_k, \quad [\X,\T_l]=0, \quad\fall i,j\le q,\, \X\in\fg,\, \mbox{and}\ l\le m.
\end{equation*}
This implies that $\bbg(q,m)$ admits a lattice whenever each $U^{(j)}$ has only rational entries,  by means of Mal'cev's criterion.  As a result, Heisenberg groups and free step-two  Carnot groups in Example  \ref{exam}  admit lattices. The existence of lattices on Heisenberg-type groups is ensured by  \cite[Theorem 1.1]{CD02}, which states that the structure constants for $\fg$ of $\hh(2n,m)$ can  take values solely in $\{0,1,-1\}$ by choosing suitable bases for $\fg_1$ and $\fg_2$.

A natural problem is to establish a complete classification for all kinds of lattices when $\bbg$ admits ones, about which, however, little is known; see  \cite[]{FMV15} for related development.  In the case of  Heisenberg groups, it is solved. As a matter of fact,  the following subgroups 
\begin{equation} \label{latHd} 
    \Gz_b:=\{((x',x''),t)\in\hh^d_\pol: x'\in b\,\zz^d, x''\in\zz^d, t\in\zz\},\quad b\in(\nn^*)^d,\, b_j|b_{j+1},\, \fall j\le d-1,
\end{equation} 
with $b\,\zz^d :=\{(b_1k_1,\ldots,b_d k_d): (k_1,\ldots,k_d)\in\zz^d \}$, classify all types of lattices of $\hh^d_\pol$  up to automorphism, cf. \cite[\S 2]{GW86}. 

Within the present framework, the lattices of our interest are as follows. Let  $L=\diag(L_1,L_2)$ be  a real matrix
with $L_i (i=1,2)$ the $q,m$-th order invertible matrices respectively, such that the discrete set
\begin{equation} \label{latt} 
    \Gz_L:=L\zz^{q+m}=\{(L_1 k',L_2k''): (k',k'')\in\zz^{q+m}\}
\end{equation}
constitutes a subgroup of $\bbg(q,m)$. This is possible; for instance, assuming each $U^{(j)}$  consists of   rational entries, then we can take $L_2$ to be the   form $L_2=\diag\left(l_{2,1},l_{2,2},\ldots,l_{2,m}\right)$, and  choose an  $L_1$  so that $\frac12\, l_{2,j}^{-1}L_1\trp U^{(j)}L_1$ is integral matrix for any $j=1,2,\ldots,m$. It is routine to validate that the set $L\,[0,1)^{q+m}$ is a fundamental domain for $\Gz_L$, whose covolume thus equals $|\det(L) |\in(0,\infty)$. Hence $\Gz_L$ is indeed a lattice of $\bbg(q,m)$. We point out that, \eqref{latt} has encompassed all forms of lattices given by \eqref{latHd}   in the setting of $\hh^d_\pol$, as easily verified, which are pivotal in the research of Weyl's law on Heisenberg manifolds, cf. e.g.   \cite[]{TZ12,KT05,PT02a}. 

Having laid the groundwork above, one is accordingly led to the following lattice point counting problem: given the ball family $B_R^{\az,M}$, how much do we know about the asymptotic behavior of 
\begin{equation} \label{lpcp} 
    \#\left(\Gz_L \cap B_R^{\az, M}\right), \qquad \mathrm{as}\ R\to\infty?
\end{equation}
In this paper, we contribute a preliminary answer to this question. Going back to the simplest Heisenberg group case considered by  \cite[]{GNT15}, note that here we have incorporated non-radial homogeneous norms and more general type of lattices; and importantly, we have improved upon the results of \cite[]{GNT15}  in many respects (see Remark  \ref{rem0}). We also remark that owing to the global coordinate $\rr^{q+m}$,  the ``lattice" counting problem  \eqref{lpcp} still makes sense, even when  $\Gz_L=L\zz^{q+m}$ is not a subgroup of $\bbg(q,m)$. With this viewpoint,  the crucial conditions for our parameters become: $\az>0$, $q\ge2,m\ge1$ and  the matrices $L,M$ are invertible.

\subsection{Main results}  \label{mainR} 

Given a lattice $\Gz_L$  on $\bbg(q,m)$, we denote the \textit{relative discrepancy} by
\begin{equation*}  \label{disp}
\mathcal{P}_{q,m}^{\az,M}(R) := \frac{|\det(L)|}{\vol\left(B_1^{\az,M}\right)R^Q} \left|\#\left(\Gz_L \cap B_R^{\az,M}\right)- \frac{\vol\left(B_1^{\az,M}\right)R^Q}{|\det(L)|} \right|.
\end{equation*}
And for any $\az>0$ we put 
\begin{equation}  \label{sig} 
    \sigma=\sigma(\az):=\begin{cases} 
    1/3, &\mathrm{if}\,\, 1<\alpha <2 , \\
    1/2, &\mathrm{otherwise}. \\
    \end{cases}
\end{equation} 
Now we formulate our first main result.
\begin{theorem} \label{thm1}
Let $\az\ge1$ and $R\ge10$. Then the following estimates hold:
\begin{enumerate}
\item If $\frac{q}{2}+1\ge\az \ge 3-m$,  then
\begin{equation} \label{GOA1} 
    \mathcal{P}_{q,m}^{\az,M}(R) \lsim R^{- \min \left\{\frac{4 m}{\alpha}+1,\, m+2 \sigma,\,2\right\} } .
\end{equation}

\item If $\az>\frac{q}{2}+1 $ and $q\le \frac{4m}{\az}+1$, then
\begin{equation*}\label{GOA2}
     \mathcal{P}_{q,m}^{\az,M}(R) \lsim R^{-2} (\log R)^{\frac{2}{Q} \one\{ q<\frac{4 m}{\alpha}+1,\, m=1\} +\one\{q=\frac{4 m}{\alpha}+1=2\}} .
\end{equation*}

\item If $\az>\frac{q}{2}+1 $ and $q > \frac{4m}{\az}+1$, then 
\begin{equation*}\label{GOA3}
    \mathcal{P}_{q,m}^{\az,M}(R) \lsim R^{-\min \left\{\frac{2q\az}{q \az+\az-4m},\, 2\right\}} (\log R)^{\frac{2}{Q} \one\{\alpha\le 4 m,\, m=1\}}. 
\end{equation*}

\item If $\az=m=1$, then
\begin{equation*} \label{GOA5} 
    \mathcal{P}_{q,1}^{1,M}(R) \lsim R^{-2} \log^{\frac{2}{q+2}} R .
\end{equation*}

\item If $1<\az<2$ and $m=1$, then
\begin{equation*} \label{GOA4} 
    \mathcal{P}_{q,1}^{\az,M}(R) \lsim R^{-\frac{6q+14}{3q+10}}. 
\end{equation*}
\end{enumerate}
\end{theorem}

\begin{remark} \label{rem0}
(1) Our Theorem  \ref{thm1}  refines several main results  on Heisenberg groups $\hh^d$ (with $q=2d$, $m=1$ in our notation) obtained by  Garg, Nevo and Taylor \cite[]{GNT15}: 
\begin{enumerate}
\item Let $d=2 $. If $\az \in(3,4] $ or $\az=1$, then Theorem  \ref{thm1} (iii) and (iv) respectively implies that the exponent of logarithmic factor  is improved from $\frac23$ (by \cite[Theorem 1.1 (3)]{GNT15} or  \eqref{ee03}) to $\frac{1}{3}$.  Note that when $\az=4$, though many improvements are made for $d=1$ and $d\ge3$,   this is the first improvement for   $d=2 $.
\item For $d=2 $ and $\az\in (2,3]$,  Theorem  \ref{thm1} (i) and (iii) drop the  factor $\log ^\frac23 R $ in the bound obtained by  \cite[Theorem 1.1 (3)]{GNT15} (or  \eqref{ee03}).
\item  When $d=1$, the relative discrepancy is improved from $O_\ez(R^{-\frac32+\ez})$  (by \cite[(4.3)]{GNT15} (or  \eqref{ee03}), which  additionally assumes Hardy's conjecture for Gauss's circle problem) to $O(R^{-2}\log^\frac12 R)$ for $\az=1$, and also to $O(R^{-\frac{13}{8}})$ for any $\az\in (1,2)$, once we invoke Theorem  \ref{thm1} (iv) and (v) separately.
\item For $d=1$ and $\az>4$, Theorem  \ref{thm1} (iii)  removes the logarithmic factor $\log R$ of the result by \cite[Theorem 1.1 (2)]{GNT15} (or  \eqref{ee02}).
\end{enumerate}

(2) Given an $(q+m)$-order real matrix of the form $\mathcal{A}=\diag(\mathcal{A}_1,\mathcal{A}_2)$, we say it is ($\bdz$-)\textit{rational} if  there is a constant $c>0$ such that the dilated matrix
$$\bdz_c(\mathcal{A}):=(c\,\mathcal{A}_1,c^2 \mathcal{A}_2)$$
is an integer matrix, and ($\bdz$-)\textit{irrational} otherwise. For  $ \az=2$, if $m=1$ and $ML$ is {rational}, then Theorem  \ref{thm1} (i) is best possible for any $q\ge2$. The proof is presented  in Subsection  \ref{prRE}.  We conjecture that this optimality extends to  $\fall m\ge2$. 

\end{remark}

\begin{theorem} \label{thm2} 
Let $\az\in(0,1]$, $m\ge2$ and $R\ge10$. Then
\begin{equation*} \label{f01} 
    \mathcal{P}_{q,m}^{\az,M}(R) \lsim R^{-\frac{2Q\az}{Q+2\az-3+\one\{m=2\}}}.
\end{equation*}
\end{theorem}

The remaining case is contained in the following theorem.
\begin{theorem}  \label{thm3} 
Let $\az>0, m=1$ and $R\ge10$. Then 
\begin{equation*}
  \mathcal{P}_{q,1}^{\az,M}(R)  \lesssim \begin{cases}
  R^{-\frac43}, &\mathrm{as}\ q=2,\\
  R^{-\frac{243}{158}},   &\mathrm{as}\  q=3, \\ 
  R^{-2} \log R,   &\mathrm{as}\  q=4, \\ 
  R^{-2},   &\mathrm{as}\  q \geq 5 ,
  \end{cases}
\end{equation*}
\end{theorem}

It should be noted that in Theorems  \ref{thm1},  \ref{thm2} and  \ref{thm3},  the implicit constants depend only  on $q,m,\az,M,L$, and there are some overlaps among their conditions but with possible different results. We do not intend to optimize the statement, since it is evident to find the best. 

\begin{remark} \label{rem1} 
For $q=2,3,4$, if $ML$ has finer property then the error order might be improved based on the corresponding Euclidean results. For instance, when  $M_1L_1$ is the identity matrix, then the estimates listed above can be replaced by $R^{2\tz^*-2+\vez}$, $R^{-\frac{27}{16}+\vez}$, $R^{-2} \log^\frac23 R$, due to   \cite[Theorem 1.2]{LY23a},  \cite[Theorem]{Hea99}  and  \cite[(11)]{Wal60}, respectively. Here $\tz^*=\frac{25 \sqrt{1717}+3292}{13762}\approx
0.31448$, whose opposite number is the unique solution of (see  \cite[Definition 1.1]{LY23a}):       
\begin{equation*}
    -\frac{8}{25} x-\frac{1}{200}\left(\sqrt{2(1-14 x)}-5 \sqrt{-1-8 x}\right)^2+\frac{51}{200}=-x, \quad  \mbox{on}\ [-0.35, -0.3].
\end{equation*} 
\end{remark}

Denote the $\mathcal{N}_{\alpha,M}$-ball    centered at $(x,t)\in\bbg(q,m)$ by
\begin{equation*}
    B_{R}^{\az, M}(x,t):= \{(x',t')\in\bbg(q,m): \cN_{\az,M}((x, t)^{-1}\circ (x',t'))\le R\}.
\end{equation*}
A standard ``ball-to-shell" argument yields the following theorem, concerning the lattice point counting near the sphere $\pt B_R^{\az,M}(x,t)$. This  region is a natural counterpart of the thin Euclidean spherical shell.
\begin{theorem} \label{lpcpS} 
Let $\az>0$, $\dz\in(0,1)$, $R\ge10$ and $(x,t)\in\Gz_L$. Then
\begin{equation} \label{lps1} 
    \#\left(\Gz_L \cap \left(B_{R+\dz}^{\az, M}(x,t)\backslash B_{R-\dz}^{\az, M}(x,t)\right) \right) \lsim \max\{R^{Q-1}\dz, R^{Q-\gz_1} \log^{\gz_2}R\},
\end{equation}
where the implicit constant depends only on  $q,m,\az,M,L$, and the indexes $\gz_j=\gz_j(\az,q,m)$ ($j=1,2$) are the   absolute values of  corresponding bound indexes for $\mathcal{P}_{q,m}^{\az,M}$ established in Theorems  \ref{thm1},  \ref{thm2} and  \ref{thm3}.   For instance, when $\az\ge1$ we have
\begin{align*}
&\gz_1=\begin{cases}
    \min \left\{\frac{4 m}{\alpha}+1,\, m+2 \sigma,\,2\right\}, &\mathrm{as}\ \frac{q}{2}+1\ge\az \ge 3-m,\\
    2, &\mathrm{as}\ \az>\frac{q}{2}+1,  \, q\le \frac{4m}{\az}+1;\ \mathrm{or}\ \az=m=1,\\
    \min \left\{\frac{2q\az}{q \az+\az-4m},\, 2\right\}, &\mathrm{as}\ \az>\frac{q}{2}+1,\, q > \frac{4m}{\az}+1,\\
    \frac{6q+14}{3q+10}, &\mathrm{as}\ 1<\az<2,\, m=1,
    \end{cases} \\[2mm]   
&\gz_2=\begin{cases} 
    1, &\mathrm{as}\ q=\frac{4m}{\az}+1=2, \\
    \frac{2}{Q}, &\mathrm{as}\ \az=m=1; \ \mathrm{or}\ m=1, \, 4\ge\az>\frac{q}2+1,\, q\neq\frac{4}{\az}+1; \\[2mm]
    0, &\mathrm{otherwise}. \\
    \end{cases}
\end{align*}
Furthermore, given any $T\ge1$, denoting the truncated lattice at height $T$ by
\begin{equation*}
    \Gamma_L(T):=\left\{\left(k^{\prime}, k^{\prime \prime}\right) \in \Gamma_L: \left|k^{\prime}\right| \leq T,\left|k^{\prime \prime}\right| \leq T^2\right\},
\end{equation*}
then an average version of  \eqref{lps1} holds:
\begin{align} \label{aveLp} 
   & T^{-Q} \# \left\{((x, t),\left(x^{\prime}, t^{\prime}\right)) \in \Gz_L(T) \times \Gz_L(T): \left|\cN_{\az,M}((x, t)^{-1}\circ (x',t'))-R \right| \le\dz \right\}\nonumber\\[2mm]
   &\qquad \lsim \max\{R^{Q-1}\dz, R^{Q-\gz_1} \log^{\gz_2}R\},
\end{align}
with the implicit constant depending only on  $q,m,\az,M,L$.
\end{theorem}

The indexes in Theorem  \ref{lpcpS}  can also be  further optimized as in  the $\cN_{\az,M}$-ball case mentioned before.  Let us compare it with the previous results.
\begin{remark}  \label{rem2} 
For Heisenberg groups $\hh^d$, Campolongo and Taylor \cite[]{CT23} in 2023 have establish a version of  Theorem  \ref{lpcpS} with  $L=M=\ii_{q+m}$ and $\az\in[2,\infty)\cap\nn$, by combining an energy integral bound with decay estimates for Heisenberg surface measures. We just point out  that, however, the estimate  \eqref{lps1}, which is a direct outcome of the lattice counting in balls as said before, is actually better (and quantitatively improved in some cases). Such kind of improvement was also  mentioned in their paper by using  instead the main result \cite[Theorem 1.1]{GNT15} of Garg, Nevo and Taylor. Later in 2026, Srivastava and Taylor \cite[]{ST26} derived an analogous version of  \eqref{aveLp}, through generalized Radon transforms. The further improvement they attain is that the logarithmic terms can be saved for $d=1,2$ when $\az=4$, compared with the results implied by  \cite[Theorem 1.1]{GNT15}. By Theorem  \ref{thm1}, our estimation is better than  \cite[Theorem 1.2]{ST26} except the range where, say $q\ge10$ and $\az\ge6$; a detailed comparison can be obtained without difficulties. Note also that in \cite[]{ST26} Srivastava and Taylor considered only even integer $\az$.

Note also that our conclusions have involved non-radial homogeneous norms and more general type of lattices even on $\hh^d$. Besides, another novelty of our Theorem  \ref{lpcpS}  lies in that it  extends the relevant results to the step-two nilpotent Lie groups with higher-dimensional group center, which has not been studied in prior works. It is unclear whether the methods of  \cite[]{ST26,CT23}  can still be  applied to this larger setting.

\end{remark}

\subsection{Method and outline of this paper}  \label{outl}
The correct order of a lattice point problem is intimately tied to the geometry of the underlying domain. Unlike the classical Euclidean $l^\az$-norm ball ($\az>0$), the $\cN_{\az,M}$-ball loses the rotational symmetry and instead displays the anisotropy, which is compatible with the dilation structure of the underlying group. When $\az<4$, the spherical surface $\pt B_{1}^{\az, M}$ fails to be of class $C ^2$; and when $\az>4$ it is $C^2$, but the Gaussian curvature  of $\pt B_{1}^{\az, M}$ vanishes everywhere on the equator $\{t=0\}$ as well as the   pole $\{x=0\} $, which also could be of different dimensions and  different degenerate behaviors. When $\az=4$  the curvature vanishes only on the pole $\{x=0\} $. This feature exhibits a greater complexity than the degeneracy  of $l^\az$-hypersphere (the boundary of $l^\az$-norm ball) in Euclidean spaces as $\az$ varies; see  \cite[\S3]{CT23}  for more details on $\hh^d$. So far, the correct order is known only for $\az=2,d\ge1$ and $\az=4,d=1$, in the Heisenberg group $\hh^d$ case. Note  that for $l^2$-norm ball, the famous  Gauss circle problem and  sphere problem are still unsolved. 

There are several existing approaches in the setting of Heisenberg groups. The methods of  \cite[]{Gat20, Gat22, MY26a} are effective  but limited to only $\az=4 $,  which also heavily rely on the arithmetic structure of the Cygan--Kor\'anyi norm  (such as the choice of slicing). So it seems difficult to apply them to general dimensions $q , m$, matrix $M$  and index $\az$.  Garg, Nevo and Taylor in \cite{GNT15} have introduced two  methods for larger range of $\az$. The first uses classical Poisson's summation formulas to transform the problem into the spectral estimates for  Heisenberg balls. This method only analyzes the case $\az\ge2$ , where, in order to overcome the difficulty arising from the vanishing curvature, they utilized the radial property of the homogeneous norm, thereby reducing the high-dimensional oscillatory integral to an one-dimensional problem. Their second method uses a hyperplane slicing argument and results in Euclidean lattice point counting  to turn the problem on Heisenberg groups into an one-dimensional summation estimation, which is then converted into an oscillatory integral via the Euler-MacLaurin formula.

Observe that the first method in \cite{GNT15} lacks of understanding for $\az<2$. In this regime, the $\cN_{\az,M}$-ball is no longer convex; moreover, when $\az\in(0,1) $ the homogeneous norm loses subadditivity; whereas for $\az\in(1,2)$  the reduced one-dimensional oscillatory integral may admit degenerate stationary points --- a situation that does not occur for $\az\ge2$. As for the second method in \cite{GNT15}, the Euler-MacLaurin formula is not applicable to general step-two groups with higher central dimensions. On the other hand, when the central dimension of a step-two group increases, the one-dimensional integral (produced  by radiality) will involve two Bessel functions, and the singularity of the integrand will become much more complicated.   Recall that Bessel functions are a class of special functions with broad applications, whose great abstruseness of  properties are prominent  to mathematicians and physicists in their research; see e.g. \cite{Wat95, GR15}. Furthermore, another new difficulty as well as interesting feature is that, due to the anisotropy of our group structure, the concentrations of lattice points on the first layer $ \rr^q$    and the second layer $\rr^m $   will be distinct, and when $q$  and $m$  are  large enough, different characteristics should appear  (see, for instance, \eqref{ff19}  in Lemma \ref{lem25}).
This point does not seem to be captured by the Heisenberg framework considered in \cite{GNT15} and \cite{Gat20, Gat22}, nor by the methods therein. In the present paper, we barely give a rudimentary study for this aspect; one would like to revisit  this issue in subsequent work to investigate discrepancy estimates uniformly in $q$  and $m$.

Our method builds upon \cite{GNT15} and introduces some original ideas.
The novel ingredients include using a partition of unity to separate singularities of the amplitude, recurrence formulas for Bessel functions, and  refined analysis of the phase functions appearing in the spectral estimates of the $\cN_{\az,M}$-ball in conjunction with its anisotropy (which seems  not fully appreciated in   \cite[]{GNT15}). These tools have a wider scope of applicability adapted to the homogeneous structure, which exploit the cancellation more efficiently  along with the integration by parts, permitting us to extend the relevant results from Heisenberg groups to   our step-two nilpotent Lie groups, and even to make many quantitative improvements as in Remark  \ref{rem0}; see also the beginning of Section \ref{pblm3} for further details.  Moreover, such improvements naturally lead to further advances beyond the estimates for the  lattice point counting near Heisenberg spheres of \cite[]{CT23,ST26}, via a ``ball-to-shell" argument as said in  Remark \ref{rem2}.

This paper is organized as follows. In Section  \ref{auxL}, we list  some preliminary tools which	are of immense importance to  establish the main results. But the proof of Lemmas  \ref{lem24}-\ref{lem25} will be postponed to Sections  \ref{plem12} and  \ref{pblm3} respectively, in view of its lengthy content. In Section  \ref{mainP}  we will first prove Theorem \ref{thm1} via some classical process (including Poisson's formulas and estimating the resulting oscillating integrals) in lattice point problems for $\az\ge1$, where $\cN_{\az,M}$ is subadditive.  And then in Subsection \ref{ss32} we prove Theorem \ref{thm2} by  using a similar  argument and the subadditivity of $\mathcal{N}_{\alpha,M}^\az$ as a replacement when $\az\in(0,1)$, where this moment $\cN_{\az,M}$ fails to retain this critical  property, though. Next, in Subsection \ref{ss33} we will only sketch the proof of  Theorem \ref{thm3}, which follows from a standard slicing argument in the literature. In Subsection \ref{prRE}, we simply prove Remark  \ref{rem0} (2) by  reductio ad absurdum, where the original concept  $\bdz$-rationality  for matrices, adapted to the homogeneous structure of our groups, comes into play. Incidentally, the $\bdz$-irrational case is fundamentally distinct, and has been partially  discussed by the author and S. Yang in \cite[]{MY26}, where we found some  new phenomena on the distribution of lattice points by constructing several examples of nilpotent Lie groups. For canonical results in the Euclidean setting, the readers are referred to  \cite[pp. 16-21]{IKKN06} and  \cite[]{Got04,BGHM22}.

In the last two sections, we derive the spectral estimates for the $\cN_{\az,M}$-ball in Section \ref{plem12}, apart from the asymptotic properties of Bessel functions and some oscillatory  integral estimates, the recursion formula will also be a key ingredient and will be used extensively. In Section \ref{pblm3}, we cope with the most difficult case for $(w,s)$-type frequencies, where partial techniques developed in  Section \ref{plem12} still work. As already mentioned, the new difficulties are singularities (at endpoints $0$ and $1$) of the integrand,  the anisotropy of $|w|$ and $|s|$, and deteriorating  degeneration of the phase functions.  To tackle  them, we apply a resolution of the identity, meanwhile distinguish into two cases,  $\lz_1\gsim \lz_2$ and $\lz_1 \lsim \lz_2$, for which different expressions of $\widehat{\chi_{B_1^{\az,M}}}(w,s)$ are chosen. These procedures, roughly speaking, can reduce matters to the situation we meet in Section \ref{plem12}, and require a detailed analysis of phases originated from the asymptotic expansion of Bessel functions (see Subsection \ref{phAnl}). After these preparations we substantiate the  first estimation \eqref{ff18} of Lemma   \ref{lem25}. Notice that  when $\frac{q}{2}+1 \ge \az \ge 3-m$, the amplitude possesses better regularities, whence  an improvement of \eqref{ff18} is prospective. We shall capture this advantage via  integration by parts with the recursion formula of Bessel function, which constitutes a crucial step in the proof of the second estimation \eqref{ff19}.

\section{Preliminaries} \label{auxL} 
\begin{proposition} \label{prp1} 
Let $\az>0$ and $M$ be invertible. Then the homogeneous norm $\mathcal{N}_{\alpha,M}$ is subadditive with respect to the Euclidean addition, that is,
\begin{equation} \label{subA} 
    \mathcal{N}_{\alpha,M}((x, t)+(x', t')) \le \mathcal{N}_{\alpha,M}(x, t)+ \mathcal{N}_{\alpha,M}(x', t'),\qquad \fall (x, t),(x',t') \in \bbg(q,m),
\end{equation}
if and only if $\az\ge1$. Moreover, the $\mathcal{N}_{\alpha,M}$-ball  
$B_1^{\az,M}$ is Euclidean-convex if and only if $\az\ge2$.
\end{proposition}
\begin{proof}
The proof is simple. When $\az\ge1$, from the elementary inequality $(a+b)^\frac12\le a^\frac12 +b^\frac12, \fall a,b\ge0$ and the triangle inequality in the Banach space $l^\az(\rr^2)$, it follows that
\begin{align}
    \cN_{\alpha, M}\left((x, t)+\left(x^{\prime}, t^{\prime}\right)\right) & =\left(\left|M_1\left(x+x^{\prime}\right)\right|^\alpha+\left|M_2\left(t+t^{\prime}\right)\right|^{\frac{\alpha}{2}}\right)^{\frac{1}{\alpha}} \nonumber\\
    & \leq\left[\left(\left|M_1 x\right|+\left|M_1 x^{\prime}\right|\right)^\alpha+\left(\left|M_2 t\right|^{\frac{1}{2}}+\left|M_2 t^{\prime}\right|^{\frac{1}{2}}\right)^\alpha\right]^{\frac{1}{\alpha}}  \label{ffa99} \\
    & \leq\left(\left|M_1 x\right|^\alpha+\left|M_2 t\right|^{\frac{\alpha}{2}}\right)^{\frac{1}{\alpha}}+\left(\left|M_1 x^{\prime}\right|^\alpha+\left|M_2 t^{\prime}\right|^{\frac{\alpha}{2}}\right)^{\frac{1}{\alpha}}  \nonumber \\
    & =\cN_{\alpha, M}(x, t)+ \cN_{\alpha, M}\left(x^{\prime}, t^{\prime}\right), \nonumber
\end{align}
whence  \eqref{subA} holds. If  this conclusion were true for $\az\in(0,1)$, by  taking
\begin{equation*}
    (x,t)=(M_1^{-1}(\oo_{q-1}\trp,1)\trp, \oo_m), \qquad (x',t')=(\oo_q, M_2^{-1}(\oo_{m-1}\trp,\ez)\trp),
\end{equation*}
where $\ez>0$ and $\oo_\nu$ denotes  the zero vector in $\rr^\nu$, we would have 
\begin{equation*}
    \left(1+\ez^{\frac{\alpha}{2}}\right)^{\frac{1}{\alpha}} \leq 1+\ez^{\frac{1}{2}} ,\quad\fall \ez>0.
\end{equation*}
But from the classical inequality $(1+a)^p\ge 1+p\,a,\fall a\ge0, p\ge1 $, it must yield that 
$$1+\frac1\az \ez^{\frac\az2} \le 1+\ez^\frac12, \quad\fall \ez>0.$$
This leads to a contradiction as $\ez\to0^+$. 

To show the second claim, we  use the convexity of the function $s\mapsto s^p$ on $[0,\infty)$ for $p\ge1$ and conclude that, whenever $\az\ge2$, $(x,t),(x',t')\in B_1^{\az,M}$ and $u\in[0,1]$,
\begin{align*}
    \cN_{\alpha, M}\left(u(x, t)+(1-u)\left(x^{\prime}, t^{\prime}\right)\right) &\le\left[\left(u\left|M_1 x\right|+(1-u)\left|M_1 x^{\prime}\right|\right)^\alpha+ \left(u \left|M_2 t\right| +(1-u)\left|M_2 t^{\prime}\right| \right)^{\frac\alpha2}\right]^{\frac{1}{\alpha}} \\ 
    &\le\left[u\left|M_1 x\right|^\alpha +(1-u)\left|M_1 x^{\prime}\right|^\alpha+ u\left|M_2 t\right|^{\frac\alpha2} +(1-u)\left|M_2 t^{\prime}\right| ^{\frac\alpha2}\right]^{\frac{1}{\alpha}} \\ 
    &=\left[u\,\cN_{\alpha, M}(x, t)^\az + (1-u)\cN_{\alpha, M}(x', t')^\az \right]^{\frac{1}{\alpha}}\\
    &\le 1,
\end{align*}
as required. However, when $\az<2$, this is not the case. Consider two boundary points  $((M_1^{-1}(\oo_{q-1}\trp,1)\trp, \oo_m)$ and $((M_1^{-1}(\oo_{q-1}\trp,s)\trp, M_2^{-1}(\oo_{m-1}\trp,s')\trp) $ with 
\begin{equation*} 
    s^\az+(s')^{\frac\az2}=1,\qquad s,s'\in(0,1).
\end{equation*}
Then their midpoint is outside the ball $B_1^{\az,M}$ when $s$ is near $1$. To see this, it is enough to find an $s_0=s_0(\az)\in(0,1)$ such that
\begin{equation*}
    1<\left(\frac{1+s}{2}\right)^\az +\left(\frac{s'}{2}\right)^\frac\az2 =\frac{(1+s)^\az+ 2^\frac\az2(1-s^\az)}{2^\az}   =: \frac{f(s)}{2^\az} ,\qquad \mbox{for}\ s_0<s<1.
\end{equation*}
Notice that
\begin{equation*}
    f'(s)=\az s^{\az-1}[(s^{-1}+1)^{\az-1}-2^\frac\az2].
\end{equation*}
Let $s_0$ be the  unique zero point of $f'$ on $(0,1)$ as $\az\in(1,2)$, and equal to $1/2$ as  $\az\in(0,1]$. Then $f$ is strictly decreasing on $(s_0,1)$, whence on this interval we have $f(s)/2^\az> f(1)/2^\az=1$ as desired. The proposition is thereby proved. We remark that the two counterexamples above are new even in Heisenberg groups.
\end{proof}

Given $f\in L^1(\rr^n)$, its Fourier transform  is defined by
$$ \widehat{f}(\xi ) :=\int_{\mathbb{R}^n} f(z) e^{-2 \pi i z \cdot \xi} d z$$
as usual. If $f$ is radial, i.e.,  $f(x)=f_0(|x|)$ for some  $f_0$ defined on $[0,\infty)$, then we have (cf. e.g.  \cite[\S B.5]{Gra14a} )
\begin{equation}  \label{rad} 
   \widehat{f}(\xi)=\frac{2 \pi}{|\xi|^{\frac{n-2}{2}}} \int_0^{\infty} f_0(r) J_{\frac{n}{2}-1}(2 \pi r|\xi|)\, r^{\frac{n}{2}} d r.
\end{equation}

The  principal tools we need to deduce Theorems  \ref{thm1} and  \ref{thm2}   are the following Lemmas  \ref{lem24}-\ref{lem25}, whose proofs are offered in Sections  \ref{plem12} and  \ref{pblm3}.  
\begin{lemma}\label{lem24}
Let $\az>0$. We have
\begin{equation} \label{ff16}
\begin{aligned}
\left|\widehat{\chi_{B_1^{\az,M}}}(w, 0)\right| \lesssim \begin{cases}\left(1+|w|)^{-q-\alpha},\right. &\mathrm{as}\ q<\frac{4 m}{\alpha}-2 \alpha+1 , \\ (1+|w|)^{-\frac{q+1}{2}-\frac{2 m}{\alpha}}, &\mathrm{as}\ q \geqslant \frac{4 m}{\alpha}-2 \alpha+1 .\end{cases}
\end{aligned}
\end{equation}
Furthermore, when $\az\in 2\nn^*$ this bound can  be improved to $(1+|w|)^{-\frac{q+1}{2}-\frac{2 m}{\alpha}}$.
\end{lemma}
 
\begin{lemma}\label{lem23}
Let $\az>0$. We have
\begin{equation*}
    \left|\widehat{\chi_{B_1^{\az,M}}}(0, s)\right| \lsim \begin{cases}(1+|s|)^{-m-\frac{\alpha}{2}}, & \,\,\mathrm{as}\,\, q>\frac{m+\alpha-1}{2} \alpha, \\ (1+|s|)^{-\frac{m+1}{2}-\frac{q}{\alpha}}, & \,\,\mathrm{as}\,\, q \le \frac{m+\alpha-1}{2} \alpha.
    \end{cases}
\end{equation*}
Furthermore, when $\az\in 4\nn^*$ this bound can  be improved to $(1+|s|)^{-\frac{m+1}{2}-\frac{q}{\alpha}}$.
\end{lemma} 

\begin{lemma} \label{lem25}
Let $\az>0, c_1, c_2>0$
 and $|w|\ge c_1,|s|\ge c_2$. Then for $\az\ge2-m$, we have
\begin{align} \label{ff18}
\left|\widehat{\chi_{B_1^{\az,M}}}(w, s)\right| \lesssim |w|^{-\frac{q}{2}} |s|^{-\frac{m}{2}} |(w, s)|^{-\sigma} .
\end{align}
Furthermore, if $\frac{q}{2}+1 \ge \az \ge 3-m$, then  \eqref{ff18} can be improved to 
\begin{align} \label{ff19}
\left|\widehat{\chi_{B_1^{\az, M}}}(w, s)\right| \lesssim |w|^{-\frac{q-1}{2}} |s|^{-\frac{m}{2}} |(w, s)|^{-\sigma-\frac12} .
\end{align}
The implicit constants in  \eqref{ff18}-\eqref{ff19}   depend only on $\az,M,q,m$ and $c_1,c_2$.
\end{lemma}

\section{Proof of main results}  \label{mainP} 
Note that 
\begin{align} \label{redu} 
    \#\left(\Gz_L \cap B_R^{\az,M}\right)=\#\left(\mathbb{Z}^{q+m} \cap B_R^{\az,\widetilde{M}}\right), \qquad \frac{\vol\left(B_R^{\az,M}\right)}{|\det(L)|} = {\vol\left(B_R^{\az,\widetilde{M}}\right)},
\end{align}
where $\widetilde{M}:=ML=\diag\{M_1L_1,M_2L_2\}$. Hence we only need to cope with the special case where $L=\ii_{q+m}$ and  $M$ is invertible, from which the general case will follow immediately. 

\subsection{Proof of Theorem \ref{thm1}}   \label{ss31} 
At present, our target is to analyze
$$\#\left(\mathbb{Z}^{q+m} \cap B_R^{\az,M}\right) = \sum_{j \in \mathbb{Z}^{q+m}}\chi_{B_R^{\az,M}}(j).$$
For that, recall the Poisson summation formula (see e.g.  \cite[Theorem 3.2.8]{Gra14a}), for $f\in C_c^\infty(\rr^{q+m})$, 
\begin{equation}  \label{poif}
   \sum_{j \in \mathbb{Z}^{q+m}} f(j)= \sum_{k \in \mathbb{Z}^{q+m}} \widehat{f}(k).
\end{equation}
Since when $\az\ge1$,  $\cN_{\az,M}$ is subadditive by Proposition  \ref{prp1}, then one can examine that
\begin{equation} \label{rho} 
    \chi_{B_{R-\epsilon}^{\az,M}} * \rho_\epsilon(x,t) \le \chi_{B_R^{\az,M}}(x,t) \le \chi_{B_{R+\epsilon}^{\az,M}} * \rho_\epsilon(x,t), \quad\fall  (x,t)\in\rr^{q+m}.
\end{equation}
Here, $\ez \in(0,1)$ is to be determined later; $*$ denotes the usual Euclidean convolution; $\rho\in C_c^\infty(\rr^{q+m})$  is non-negative and supported in ${B_{1}^{\az,M}}$ with integral value 1; and $\rho_\ez :=\ez ^{-Q}\rho(\bdz_{\ez ^{-1}}(\cdot))$.  To show  \eqref{rho}, first observe that the convolution terms in  \eqref{rho} both take   values only in $[0,1]$. Then it is enough to exploit the following facts:
\begin{equation*}
    (x,t)\notin B_{R}^{\az,M} \Rightarrow (x,t)\notin\left(B_{R-\epsilon}^{\az,M} + B_{\epsilon}^{\az,M} \right), \quad (x,t)\in B_{R}^{\az,M} \Rightarrow (x,t)\in\left(B_{R+\epsilon}^{\az,M} - B_{\epsilon}^{\az,M} \right).
\end{equation*}

We next use $\chi_{B_{R\pm\epsilon}^{\az,M}} * \rho_\epsilon$ to replace $\chi_{B_{R}^{\az,M}}$, and plug them into the Poisson summation formula, and finally sum up the 
the resulting errors. Precisely, by  \eqref{poif}-\eqref{rho},  on one hand,
\begin{equation*}
    \#\left(\mathbb{Z}^{q+m} \cap B_R^{\az,M}\right) \le \sum_{k \in \mathbb{Z}^{q+m}}\left(\chi_{B_{R+\epsilon}} * \rho_\epsilon\right)^\wedge (k) =
    \sum_{k \neq 0} \widehat{\chi_{B^{\az,M}_{R+\epsilon}}}(k) \widehat{\rho}_\epsilon(k)+ \vol\left( B^{\az,M}_{R+\epsilon}\right);
\end{equation*}
and on the other, 
\begin{equation*}
    \#\left(\mathbb{Z}^{q+m} \cap B_R^{\az,M}\right) \ge \sum_{k \in \mathbb{Z}^{q+m}}\left(\chi_{B_{R-\epsilon}} * \rho_\epsilon\right)^\wedge (k) =
    \sum_{k \neq 0} \widehat{\chi_{B^{\az,M}_{R-\epsilon}}}(k) \widehat{\rho}_\epsilon(k) + \vol\left( B^{\az,M}_{R-\epsilon}\right).
\end{equation*}
From  \eqref{ballA} we see that, for all $R\ge10$,
\begin{equation} \label{red1} 
\left|\#\left(\mathbb{Z}^{q+m} \cap B_R^{\az,M}\right) -  \vol\left( B^{\az,M}_{R}\right) \right|   \lsim  \sum_{\pm}\sum_{k \neq 0} \left| \widehat{\chi_{B^{\az,M}_{R\pm\epsilon}}}(k)\right| |\widehat{\rho}_\epsilon(k)| +R^{Q-1}\ez.
\end{equation}
Because $\widehat{\rho}$ is of Schwartz class  and (through a routine verification)
\begin{equation*}
\widehat{\chi_{B^{\az,M}_{R}}}(k)=R^Q \,\widehat{\chi_{B^{\az,M}_{1}}}(Rk',R^2k''),\qquad   \widehat{\rho_\ez}(k)= \widehat{\rho}(\ez k',\ez^2 k''),
\end{equation*}
then by  \eqref{red1}, we have after a normalization that
\begin{equation} \label{red3} 
    \mathcal{P}_{q,m}^{\az,M}(R)\lsim  \sum_{\pm} \bS(R\pm\ez,\ez)  + R^{-1}\ez,
\end{equation}
with
\begin{align*}
    \bS(R,\ez) &:= \sum_{k \neq 0} \left| \widehat{\chi_{B^{\az,M}_{1}}}(Rk',R^2k'')  \right| \left(1+\ez\left|k^{\prime}\right|+\ez^2\left|k^{\prime \prime}\right|\right)^{-N}\\
    &=\sum_{k'\neq0,\, k''=0}\cdots + \sum_{k'=0,\, k''\neq0}\cdots + \sum_{k'\neq0,\, k''\neq0}\cdots \\
    &=: \sum_{j=1}^3 \bS_j(R,\ez),
\end{align*}
where $N>0$ is a large constant to be chosen (indeed $N=q+m+1$ will suffice). Thus we are  reduced to the estimation of each $\bS_j(R,\ez)$ for large $R$ and small $\ez$. In fact, we will prove that, whenever $R\ge10$,
\begin{gather}
    \bS_1(R,\ez) \lsim R^{-\beta_1} \cE(q,\beta_1), \quad \mbox{for}\ \az>0, \label{es1} \\[1mm]
    \bS_2(R,\ez) \lsim R^{-2\beta_2}\, \cE(2m,2\beta_2), \quad \mbox{for}\ \az>0, \label{es2} \\[1mm]
   \bS_3(R,\ez) \lsim \begin{cases}
   R^{-\frac{Q}{2}-2 \sigma}\left(\ez^{-1}\right)^{\frac{q}{2}}\cE(m,2\sz), &\mbox{for $\az\ge 2-m$},\\
   R^{-\frac{Q+1}{2}-2 \sigma}\left(\ez^{-1}\right)^{\frac{q+1}{2}}\cE(m-1,2\sz), &\mbox{for}\ \frac{q}{2}+1 \ge \az \ge 3-m, \\
   \end{cases}  \label{es3} 
\end{gather}
where $\sz$ is given by  \eqref{sig},  and
\begin{equation*} \label{bbt2} 
    \beta_1:=\frac{q}{2}+\frac{1}{2} \min \left\{q+2 \alpha, \frac{4 m}{\alpha}+1\right\},\quad \beta_2:=\frac{m}{2}+\frac{1}{2} \min \left\{m+\alpha, \frac{2 q}{\alpha}+1\right\},
\end{equation*}
and $\cE(\cdot)$ is defined by,  for $\ez\in(0,1)$, 
\begin{equation*}
     \cE(\eta_1,\eta_2):=
         \begin{cases}
         (\ez^{-1})^{\eta_1-\eta_2}, &\mbox{as}\  \eta_1>\eta_2, \\
         \log (\ez^{-1}), &\mbox{as}\  \eta_1=\eta_2, \\
         1, &\mbox{as}\  \eta_1<\eta_2, \\
         \end{cases}
         \qquad \fall\eta_1,\eta_2\in\rr.
\end{equation*}

To show  \eqref{es1},  by Lemma  \ref{lem24} we obtain
\begin{equation*}
 \begin{aligned}
   \bS_1(R,\ez)    & \lesssim R^{-\beta_1} \sum_{k^{\prime} \neq 0}\left|k^{\prime}\right|^{-\beta_1}\left(1+\ez \left|k^{\prime}\right|\right)^{-N}  \\
   &\lesssim R^{-\beta_1} \int_{\{z\in\rr^q:\,|z|\ge1\}} |z|^{-\beta_1}(1+\ez|z|)^{-N} d z \\
    & \lsim R^{-\beta_1} \int_{1}^{\infty} r^{-\beta_1+q-1}(1+\ez  r)^{-N} d r \\
    &\le R^{-\beta_1}\left(\int_{1}^{\ez^{-1}} r^{- \beta_1+q-1} d r+\ez^{-N} \int_{\ez^{-1}}^{\infty} r^{- \beta_1+q-1-N} d r\right), 
    \end{aligned}
\end{equation*}
which gives the desired estimation  if $N>q-\beta_1$. To show  \eqref{es2}, we see  from Lemma  \ref{lem23}  that
\begin{equation*}
    \begin{aligned}
   \bS_2(R,\ez) &\lesssim \sum_{k^{\prime\prime} \neq 0} R^{-2 \beta_2}\left|k^{\prime \prime}\right|^{-\beta_2}\left(1+\ez ^2\left|k^{\prime \prime}\right|\right)^{-N}\\
   & \lesssim R^{-2 \beta_2} \int_{\{z\in\rr^m:\,|z|\ge1\}} |z|^{-\beta_2}\left(1+\ez^2|z|\right)^{-N} d z \\
    & \lsim R^{-2\beta_2} \int_{1}^{\infty} r^{-\beta_2+m-1}(1+\ez ^2 r)^{-N} d r \\
    &\le R^{-2 \beta_2}\left(\int_{1}^{\ez ^{-2}} r^{-\beta_2 +m-1} d r+\ez ^{-2 N} \int_{\ez ^{-2}}^{\infty} r^{-\beta_2+m-1-N} d r\right),
    \end{aligned}
\end{equation*}
which implies our claim provided that $N>\max\{q-\beta_1,m-\bz_2\}$. To show  \eqref{es3},  via  \eqref{ff18} of Lemma  \ref{lem25} it yields that
\begin{equation*}
    \begin{aligned}
     \bS_3(R,\ez)
     &\lesssim \sum_{k^{\prime} \neq 0,\, k^{\prime \prime} \neq 0} R^{-\frac{q}{2}}\left|k^{\prime}\right|^{-\frac{q}{2}} R^{-m}\left|k^{\prime \prime}\right|^{-\frac{m}{2}}\left(R\left|k^{\prime}\right|+R^2\left|k^{\prime \prime}\right|\right)^{-\sigma}\left(1+\ez \left|k^{\prime}\right|+\ez ^{2}\left|k^{\prime \prime}\right|\right)^{-N} \\
    &\lesssim R^{-\frac{Q}{2}} \int_{|z'|\ge1,|z''|\ge1}  \left|z^{\prime}\right|^{-\frac{q}{2}}\left|z^{\prime \prime}\right|^{-\frac{m}{2}}\left(R\left|z^{\prime}\right|+R^2\left|z^{\prime \prime}\right|\right)^{-\sigma}\left(1+\ez \left|z^{\prime}\right|+\ez ^2\left|z^{\prime \prime}\right|\right)^{-N} d z^{\prime} d z^{\prime \prime} \\
    & \lsim R^{-\frac{Q}{2}} \int_{1}^{\infty} \int_{1}^{\infty} r_1^{\frac{q}{2}-1} r_2^{\frac{m}{2}-1}\left(R r_1+R^2 r_2\right)^{-\sigma}\left(1+\ez  r_1+\ez ^2 r_2\right)^{-N} d r_1 d r_2 \\
    & \le R^{-\frac{Q}{2}}\left[\int_{1}^{\ez ^{-1}} \int_{1}^{\ez ^{-2}} r_1^{\frac{q}{2}-1} r_2^{\frac{m}{2}-1}\left(R^2 r_2\right)^{-\sigma} d r_1 d r_2\right. \\
    &\qquad +\int_{1}^{\ez ^{-1}} \int_{\ez ^{-2}}^{\infty} r_1^{\frac{q}{2}-1} r_2^{\frac{m}{2}-1}\left(R^2 r_2\right)^{-\sigma}\left(\ez ^2 r_2\right)^{-N} d r_1 d r_2 \\
    &\qquad +\int_{\ez ^{-1}}^{\infty} \int_{1}^{\ez ^{-2}} r_1^{\frac{q}{2}-1} r_2^{\frac{m}{2}-1}\left(R^2 r_2\right)^{-\sigma}\left(\ez  r_1\right)^{-N} d r_1 d r_2 \\
    &\qquad \left.+\int_{\ez ^{-1}}^{\infty} \int_{\ez ^{-2}}^{\infty} r_1^{\frac{q}{2}-1} r_2^{\frac{m}{2}-1}\left(R^2 r_2\right)^{-\sigma}\left(\ez r_1\right)^{-\frac{N}{2}}\left(\ez ^2 r_2\right)^{-\frac{N}{2}} d r_1 d r_2\right],
    \end{aligned}
\end{equation*}
which readily leads to the first estimate in  \eqref{es3} as long as $N>\max\{q,m\}$; while the second one can also be justified likewise by using  \eqref{ff19} instead (with $N>\max\{q+1,m\}$).

The remaining proof now can be finished by determining suitable $\ez$ to control the resulting errors, based on various conditions. For  (i) of Theorem  \ref{thm1}, we use  \eqref{red3}-\eqref{es2}  and the second estimate in  \eqref{es3} to get
\begin{equation*}  \label{TOUPDA} 
    \mathcal{P}_{q,m}^{\az,M}(R)\lsim R^{-\beta_1} \cE(q,\beta_1) + R^{-2\beta_2}\, \cE(2m,2\beta_2) + R^{-\frac{Q+1}{2}-2 \sigma}\left(\ez^{-1}\right)^{\frac{q+1}{2}}\cE(m-1,2\sz)  + R^{-1}\ez.
\end{equation*}
To balance the four quantities we let $\ez=R^{-1}$, and this will imply  \eqref{GOA1} after some tedious verifications. Let us inspect the situation that $\az\in[2,\infty)$ as an example. Recall  $q\ge2$. At this moment, it holds $\sz=1/2$, and the first three resulting terms (denoted by $\bB_j,\,j=1,2,3$) are bounded by
\begin{equation*}
    \bB_1\lsim \begin{cases}
    R^{-\frac{4m}{\az}-1}, & \mbox{as}\ q>\bz_1,\\
    R^{-\frac{4m}{\az}-1}\log R, & \mbox{as}\ q=\bz_1,\\
    R^{-\bz_1}, & \mbox{as}\ q<\bz_1,\\
    \end{cases}
    \quad
    \bB_2\lsim R^{-\gz},
    \quad
   \bB_3\lsim 
   \begin{cases}
   R^{-2-2 \sz} \log R,&\mbox{as}\ m \geq 2, \\
   R^{-1-2 \sigma}, &\mbox{as}\ m=1,
   \end{cases}
\end{equation*}
respectively, where $\gz=\gz(q,m,\az)>2$ is a constant. From this,  \eqref{GOA1} follows immediately.  The situation where $\az\in[1,2)$ can also be clarified analogously. 

For the rest cases, we conclude from   \eqref{red3}-\eqref{es2}  and the first inequality in  \eqref{es3} that
\begin{equation*}
    \mathcal{P}_{q,m}^{\az,M}(R)\lsim R^{-\beta_1} \cE(q,\beta_1) + R^{-2\beta_2}\, \cE(2m,2\beta_2) + R^{-\frac{Q}{2}-2 \sigma}\left(\ez^{-1}\right)^{\frac{q}{2}}\cE(m,2\sz) + R^{-1}\ez.
\end{equation*}
Similarly, the desired conclusions can be drawn by choosing $\ez$ as follows and plugging them into the above estimation:
\begin{itemize}
\vspace{-.2cm}
\item for (ii) we take   $\ez=R^{-1}\log^{\frac{2}{Q}} R$ if $q<\frac{4m}{\az}+1$ and $m=1$; and take $\ez=R^{-1}$ otherwise; 
\vspace{-.2cm}
\item for (iii) we take $\ez= R^{-\frac{q \az-\az+4m}{q \az+\az-4m}}$ if $\az>4m$; and take $\ez= R^{-1}\log^{\frac2Q \one\{m=1\}} R$ otherwise; 
\vspace{-.2cm}
\item for (iv) we take  $\ez=R^{-1}\log^{\frac{2}{q+2}} R$; 
\vspace{-.2cm}
\item for (v) we take $\ez=R^{-\frac{3 q+4}{3 q+10}} $.
\vspace{-.2cm}
\end{itemize}

\subsection{Proof of Theorem \ref{thm2}}  \label{ss32}
Let $\az\in(0,1]$. Although  by Proposition  \ref{prp1}, $\mathcal{N}_{\alpha,M}$ is not subadditive for $\az\in(0,1)$, which plays an important part in the proof of Theorem  \ref{thm1}, we can  consider $\mathcal{N}_{\alpha,M}^\az$ as a substitute, whose subadditivity is indeed a simple outcome of  \eqref{ffa99} and the primary inequality 
$$(a+b)^\az \le a^\az +b^\az, \qquad \ \fall a,b\ge0,\ \fall \az\in(0,1].$$
Consequently, we may still obtain a counterpart of  \eqref{rho} which reads, for $R\ge10$ and $\ez\in(0,1)$,
\begin{equation*} \label{rho2} 
    \chi_{B_{(R^\az-\epsilon^\az)^{1/\az}}^{\az,M}} * \rho_\epsilon(x,t) \le \chi_{B_R^{\az,M}}(x,t) \le \chi_{B_{(R^\az+\epsilon^\az)^{1/\az}}^{\az,M}} * \rho_\epsilon(x,t), \quad\fall  (x,t)\in\rr^{q+m}.
\end{equation*}
Then by appealing to Poisson's summation formula as in the proof Theorem   \ref{thm1}, with the fact  
\begin{equation*}
    \vol\left( B^{\az,M}_{(R^\az\pm \epsilon^\az)^{1/\az}}\right)=\vol\left( B^{\az,M}_{1}\right)R^Q + O(R^{Q-\az}\ez^{\az}),
\end{equation*}
we arrive at that, for $R\ge10$,
\begin{equation*} \label{red4} 
    \mathcal{P}_{q,m}^{\az,M}(R)\lsim  \sum_{\pm} \bS((R^\az\pm\epsilon^\az)^{1/\az},\ez)  + R^{-\az}\ez^{\az}.
\end{equation*}
This, together with \eqref{es1}-\eqref{es3}, implies that, for $m=2$ (using the first estimation of  \eqref{es3}),
\begin{equation*}
    \mathcal{P}_{q,m}^{\az,M}(R)\lsim R^{-\beta_1} \cE(q,\beta_1) + R^{-2\beta_2}\, \cE(2m,2\beta_2) + R^{-\frac{Q}{2}-1}\left(\ez^{-1}\right)^{\frac{Q}{2}-1} + R^{-\az}\ez^{\az},
\end{equation*}
and for $m\ge3$ (using the second one),
\begin{equation*}
    \mathcal{P}_{q,m}^{\az,M}(R)\lsim R^{-\beta_1} \cE(q,\beta_1) + R^{-2\beta_2}\, \cE(2m,2\beta_2) + R^{-\frac{Q+1}{2}-1}\left(\ez^{-1}\right)^{\frac{Q+1}{2}-2} + R^{-\az}\ez^{\az}.
\end{equation*}
Finally, from the choices that $\ez=R^{-\frac{Q+2-2 \alpha}{Q-2+2 \alpha}}$ for $m=2$, and $\ez=R^{-\frac{Q+3-2 \alpha}{Q-3+2 \alpha}}$ for $m\ge3$, Theorem  \ref{thm2} will follow after some elementary but prolix evaluations.

\subsection{Proof of Theorem \ref{thm3}}  \label{ss33}
Theorem  \ref{thm3} is a direct consequence of the classical lattice counting  problem in ellipsoids on $\rr^q$ combining a slicing argument. This procedure is almost identical to that of  \cite[Theorem 1.1 (3)]{GNT15}, so we shall omit the details. The main difference is that instead of discrepancy estimates for standard Euclidean balls, we employ the ellipsoid version. Recall that if $\cM$ is an $n$-th order positive definite matrix and $\cB_n:=\{x\in\rr^n: x\trp \cM x\le1\}$, then we have, for $R\ge10,$
\begin{equation*}
     \left| \#\left(\mathbb{Z}^{n} \cap R\,\cB_n\right)- {\vol\left(\cB_n\right)} R^n \right|\lsim 
     \begin{cases}
     R^{\frac23}, &\mbox{as}\  n=2, \\
     R^{\frac{231}{158}}, &\mbox{as}\  n=3, \\
     R^{2}\log R, &\mbox{as}\  n=4, \\
     R^{n-2}, &\mbox{as}\  n\ge5.
     \end{cases}
\end{equation*}
For $n=2$ the result can be found in e.g.  \cite[p. 9]{IKKN06}; for $n=3$ see \cite[Theorem 1.1]{Guo12}; for $n\ge4$ the estimates are contained in \cite[Theorem 1.5(i)] {Got04} ($n\ge5$) and its proof ($n=4$); Likewise, the upper bounds in $n=2,\,3,\,4$ can be improved if we impose finer conditions on $\widetilde{M}$ and appeal to stronger conclusions accordingly  as mentioned in Remark  \ref{rem1}.  

\subsection{Proof of Remark  \ref{rem0} (2)}  \label{prRE} 
The proof is classical and based on an argument by contradiction. Fix  $c>0$ such that $\bdz_c\widetilde{M}$ is an integral matrix. We are required to show that $ \mathcal{P}_{q,1}^{2,M}(R)\neq o(R^{-2})$.   Were this false,  in view of  \eqref{redu} we would have that
\begin{equation*}
  \left|  \#\left(\mathbb{Z}^{q+1} \cap B_{R}^{2,\bdz_c\widetilde{M}}\right)- {\vol\left(B_{R}^{2,\bdz_c\widetilde{M}}\right)} \right| =    {\vol\left(B_{c^{-1} R}^{2,\widetilde{M}}\right)}\, \mathcal{P}_{q,1}^{2,M}(c^{-1}R) = o(R^{q}), \quad \mbox{as}\ R\to \infty.
\end{equation*}
Notice that presently, 
\begin{equation*}
    \#\left(\mathbb{Z}^{q+1} \cap B_{\sqrt{N}}^{2,\bdz_c\widetilde{M}}\right) =  \#\left(\mathbb{Z}^{q+1} \cap B_{\sqrt{N+1/2}}^{2,\bdz_c\widetilde{M}}\right), \quad\fall N\in\nn^*.
\end{equation*}
Therefore,
$${\vol\left(B_{\sqrt{N}}^{2,\bdz_c\widetilde{M}}\right)} = {\vol\left(B_{\sqrt{N+1/2}}^{2,\bdz_c\widetilde{M}}\right)} + o(N^{\frac{q}2}), \quad \mbox{as}\ N\to \infty, $$
which leads to a contradiction by  \eqref{ballA}. This completes the proof.   

\subsection{Proof of Theorem  \ref{lpcpS}} \label{ss35}
The socalled ``ball-to-shell" argument refers to that one can straightforwardly derive the estimates for lattice counting in a thin spherical shell from that in a pair of large balls which surround it. More precisely, denote by $R^{-Q}E(R)$ the controlling bounds for $\mathcal{P}_{q,m}^{\az,M}(R)$ obtained in Theorems  \ref{thm1},  \ref{thm2} and  \ref{thm3}. Then  one can find  $C_0=C_0(q,m,\az,M,L)>0$ such that, for all $(x,t)\in\Gz_L$ and $\dz\in(0,2),$
\begin{align*}
    \#\left(\Gz_L \cap \left(B_{R+\dz}^{\az, M}(x,t)\backslash B_{R-\dz}^{\az, M}(x,t) \right)\right) &= \#\left(\Gz_L \cap B_{R+\dz}^{\az, M}(x,t)\right) - \#\left(\Gz_L \cap B_{R-\dz}^{\az, M}(x,t)\right) \\
    &= \#\left(\Gz_L \cap B_{R+\dz}^{\az, M}\right) - \#\left(\Gz_L \cap B_{R-\dz}^{\az, M}\right) \\
    &\le \frac{\vol\left(B_1^{\az,M}\right)}{|\det(L)|} ((R+\dz)^Q-(R-\dz)^Q ) \\
    &\qquad+ C_0 (E(R+\dz) +E(R-\dz))\\
    &\lsim R^{Q-1}\dz +E(R),
\end{align*}
which gives  \eqref{lps1} as required. For  \eqref{aveLp}, one directly uses  \eqref{lps1} (which holds for $\dz\in(0,2) $) and the following observation:
\begin{equation*}
 \begin{aligned}
    \mbox{LHS of  \eqref{aveLp}} &\le T^{-Q} \sum_{(x,t)\in\Gz_L(T)} \#\left(\Gz_L(T) \cap \left(B_{R+2\dz}^{\az, M}(x,t)\backslash B_{R-2\dz}^{\az, M}(x,t)\right) \right) \\
    &\lsim T^{-Q}\, T^{Q} \max_{(x,t)\in\Gz_L(T)} \#\left(\Gz_L(T) \cap \left(B_{R+2\dz}^{\az, M}(x,t)\backslash B_{R-2\dz}^{\az, M}(x,t)\right) \right) \\
    &\le  \sup_{(x,t)\in\Gz_L} \#\left(\Gz_L \cap \left(B_{R+2\dz}^{\az, M}(x,t)\backslash B_{R-2\dz}^{\az, M}(x,t)\right) \right). \\
   \end{aligned}
\end{equation*}
The proof of Theorem  \ref{lpcpS} is therefore finished. 

\section{Proof of Lemmas \ref{lem24}-\ref{lem23}}  \label{plem12} 
Here and in what follows, we write $B_1^{\az}:=B_1^{\az,\,\ii_{q+m}}$. A direct calculation verifies that 
\begin{equation} \label{std} 
    \widehat{\chi_{B_1^{\az,M}}}(w, s) = |\det(M)|^{-1}\widehat{\chi_{B_1^{\az}}}\left((M_1^{-1})\trp w, (M_2^{-1})\trp s\right).
\end{equation}
Notice that if $M_0$ is an invertible matrix of order $n\in\nn^*$, then
$$|M_0\,z|\sim_{M_0}|z|,\quad \forall z\in\rr^n.$$
Hence by \eqref{std}, to prove Lemmas \ref{lem24}-\ref{lem25}, we only need to consider the case that $M=\ii_{q+m}$.  Using  \eqref{rad}  one evaluates $\widehat{\chi_{B_1^{\alpha}}}(w, s)$ explicitly as follows:
\begin{align}
 \widehat{\chi_{B_1^\alpha}}(w, s)  &=\int_{B_1^\alpha} e^{-2 \pi i(w \cdot x+s \cdot t)} d x d t = \int_{|x| \leq 1} e^{-2 \pi i w \cdot x}\left(\int_{|t| \leq\left(1-|x|^\alpha\right)^{\frac{2}{\alpha}}} e^{-2 \pi i s \cdot t} d t\right)dx \nonumber\\
     &= \int_{|x| \leq 1} e^{-2 \pi i w \cdot x} (1-|x|^\az)^{\frac{m}{\az}}|s|^{-\frac{m}2} J_{\frac{m}2}\left(2\pi (1-|x|^\az)^{\frac2{\az}} |s|\right) dx \nonumber\\
     &=2 \pi|w|^{-\frac{q}{2}+1}|s|^{-\frac{m}{2}}  \int_0^1 J_{\frac{q}{2}-1}(2 \pi|w| r) J_{\frac{m}{2}}\left(2 \pi|s|\left(1-r^\alpha\right)^{\frac{2}{\alpha}}\right) r^{\frac{q}{2}}\left(1-r^\alpha\right)^{\frac{m}{\alpha}} d r    \label{ff01}\\
     & =   2 \pi|w|^{-\frac{q}{2}}|s|^{-\frac{m}{2}+1}    \int_0^1 J_{\frac{m}{2}-1}\left(2 \pi|s| r\right) J_{ \frac{q}{2}} \left(2 \pi|w|\left(1-r^{\frac{\alpha}{2}}\right)^{\frac{1}{\alpha}}\right) r^{\frac{m}{2}}\left(1-r^{\frac{\alpha}{2}}\right)^{\frac{q}{2 \alpha}} d r.   \label{ff01'} 
\end{align}
Here in the last step we have used similar argument as in the proof of  \eqref{ff01}, but integrating the $x$-variable first, and  $J_\nu$ is the Bessel function of order $\nu$, defined by (first for $\Re(\nu)>-\frac12$)
 \begin{equation*}
     J_\nu(r)=\frac{\left(\frac{r}{2}\right)^\nu}{\Gamma\left(\nu+\frac{1}{2}\right) \Gamma\left(\frac{1}{2}\right)} \int_{-1}^{+1} e^{i r s}\left(1-s^2\right)^\nu \frac{d s}{\sqrt{1-s^2}}, \quad r>0.
 \end{equation*}
Applying  the power series, $J_\nu$ has an entire
extension in $\nu$ on the complex plane for all $r>0$; see e.g.  \cite[\S B]{Gra14a} and  \cite[]{Wat95}  for more details. As to our applications, we  always assume that $\nu\in\rr$.
 
Let us recall some useful properties of Bessel functions as follows:
\begin{itemize}
\item {Recursion formulas} (\cite[\S 8.471]{GR15}):
\begin{gather}
      \frac{d}{d r}\left(r^\nu J_\nu(r)\right)= r^\nu J_{\nu-1}(r), \quad \nu\in\rr,\,r>0,  \label{ff12} \\
      J_{-n}(r)=(-1)^n J_n(r), \quad n\in\zz,\,r>0.  \label{ff13}
\end{gather}

\item Small argument estimates (\cite[\S B]{Gra14a}):
\begin{equation} \label{ff05}
    |J_\nu(r)|\le C_\nu\, r^{\nu}, \quad \forall\, r\le1;\ \forall\,\nu>-\frac32.
\end{equation}

\item Large argument estimates (\cite[\S 8.451]{GR15}): 
\begin{equation} \label{ff06}
    |J_\nu(r)|\le C_\nu'\, r^{-\frac12}, \quad \forall\, r\ge1;\  \forall\,\nu\in\rr.
\end{equation}

\item Asymptotics at infinity  (\cite[\S 8.451]{GR15}): 
\begin{equation} \label{ff11}
J_\nu(r) \sim r^{-\frac{1}{2}} \sum_{\pm} e^{ \pm i r} \sum_{j=0}^{\infty} a_{\nu, j}^{\pm}\, r^{-j}, \quad \mathrm{as}\,\, r\to\infty,
\end{equation}
where $a^\pm_{\nu, j}$'s are complex numbers that can be determined exactly, and the asymptotic sum holds in the sense that, for any $N\in\nn$,
\begin{equation*}
    J_\nu(r) - r^{-\frac{1}{2}} \sum_{\pm} e^{ \pm i r} \sum_{j=0}^{N} a_{\nu, j}^{\pm}\, r^{-j} = O(r^{-N-\frac32}), \quad \mathrm{as}\,\, r\to\infty.
\end{equation*}
\end{itemize}

\textbf{Proof of Lemma  \ref{lem24}.} 
 Note that by \cite[\S B.6]{Gra14a} 
 \begin{equation} \label{Jv0} 
     \lim_{r\to0+} J_\nu(r)r^{-\nu}=\frac{1}{2^\nu\,  \Gamma(\nu+1)}, \qquad\fall \nu>-\frac12.
 \end{equation}
Then letting $s\to 0$ in  \eqref{ff01} gives  
\begin{equation*}
    \begin{aligned}
    \widehat{\chi_{B_1^\alpha}}(w, 0)  = \frac{2\pi^{\frac{m}{2}+1}}{\Gamma\left(\frac{m}{2}+1\right)} |w|^{-\frac{q}{2}+1} \int_0^1 J_{\frac{q}{2}-1}(2 \pi r|w|) r^{\frac{q}{2}}\left(1-r^\alpha\right)^{\frac{2 m}{\alpha}} d r .
    \end{aligned}
\end{equation*}
Since the desired estimates  are trivial for $ |w|\lsim1$, we are left with large $|w|$. Let  $ \lz=2\pi|w|\ge10$. We write
\begin{equation*}
 \mathcal{I}(\lz):= \frac{\Gamma\left(\frac{m}{2}+1\right)}{2^\frac{q}2\pi^{\frac{q+m}{2}}} \lz^{\frac{q}{2}-1} \widehat{\chi_{B_1^\alpha}}(w, 0)= \int_0^1 J_{\frac{q}{2}-1}(\lambda  r) \, r^{\frac{q}{2}}\left(1-r^\alpha\right)^{\frac{2 m}{\alpha}} d r.
\end{equation*}
To control it, the points $r=0$ and $r=1$ are treated  separately. Fixing a non-negative function $\varphi\in C^\infty(\rr)$ such that 
\begin{equation} \label{vhi} 
    \varphi(r)= \begin{cases}0, & r \leq \frac{1}{4}, \\ 1, & r \geq \frac{1}{2},\end{cases}
\end{equation}
then the integral can be split  as follows:
\begin{equation*}
    \mathcal{I}=\int_0^1 \varphi\cdots +\int_0^1(1-\varphi) \cdots =: \mathcal{I}_1+\mathcal{I}_2.
\end{equation*}
In the sequel, to get Lemma \ref{lem24} we will show that
\begin{equation} \label{Ig24} 
    |\cI_1(\lz)|\lsim\lambda^{-\frac{2 m}{\alpha}-\frac{3}{2}}, \qquad |\cI_2(\lz)| \lsim \lz^{-\frac{q}{2}-\az-1},
\end{equation}
and for any $\az\in2\nn$ that
\begin{equation} \label{IG24'} 
    |\cI_2(\lz)| \lsim_k \lz^{-k},\quad \fall k\in\nn.
\end{equation}

For $\mathcal{I}_1$, the new amplitude function behaves well near $0$, but there is a singularity of order $2m/\az$ at point $1$. For that, we let 
$$\varphi_0(r):=\left(\frac{1-r^\alpha}{1-r}\right)^{\frac{2 m}{\alpha}}, \quad r \in[0,1),\qquad \varphi_0(1):=\alpha^{\frac{2 m}{\alpha}}.$$
It is easily seen that $\varphi\varphi_0\in C^\infty([0,1]).$
Then we  can write
\begin{equation*}
    \cI_1(\lambda)=\int_0^1 J_ {\frac{q}{2}-1}(\lambda  r)(1-r)^{\frac{2 m}{\alpha}} \varphi(r) \varphi_0(r) r^{\frac{q}{2}} d r ,
\end{equation*}
which, after plugging \eqref{ff11} into it,  can be bounded by
\begin{equation*}
    \begin{aligned}
    \cI_1^{(N)}(\lambda) & :=\sum_{j=0}^{N-1} C_j \lambda^{-j-\frac{1}{2}}\sum_\pm\left|\int_0^1 e^{ \pm  i \lambda r}(1-r)^{\frac{2 m}{\alpha}} \tilde{\varphi}_j(r) d r\right|+C_N \lambda^{-N-\frac{1}{2}} \int_0^1(1-r)^{\frac{2 m}{\alpha}} \tilde{\varphi}_N(r) d r \\
    & =: \sum_{j=0}^{N-1} C_j \lambda^{-j-\frac{1}{2}}\ \cI_{1, j}(\lambda) + \cI_{1, N}(\lambda) ,
    \end{aligned}
\end{equation*}
where $N\in\nn^*$
 will be chosen soon, $C_j>0$ is suitable  constant and 
 $$\tilde{\varphi}_j(r):=\varphi(r) \varphi_0(r) r^{\frac{q}{2}-\frac{1}{2}-j}, \quad j=0,1,\ldots, N .$$
For $ \cI_{1, N}$, it is obvious that $\cI_{1, N}(\lambda) \lsim \lambda^{-N-\frac{1}{2}} $. For
$\cI_{1, j} (0\le j \le N-1)$, since $\tilde{\varphi}_j\in C^\infty([0,1]) $ and $ \tilde{\varphi}_j^{(k)}(0)=0, \forall k \in \mathbb{N} $, by
the standard  oscillatory integral theory with singularities in the integrands (see e.g.  \cite[pp. 47-48]{Erd56}),  we have $\cI_{1, j}(\lambda)\lsim\lambda^{-\frac{2 m}{\alpha}-1}$. As a result, one can choose an integer $N$ larger than $2m/\az+1$, and obtain
$$\left|\cI_1(\lambda)\right| \le \cI_1^{(N)}(\lambda)  \lesssim \sum_{j=0}^{N-1} \lambda^{-j-\frac{1}{2}} \cdot \lambda^{-\frac{2 m}{\alpha}-1}+\lambda^{-N-\frac{1}{2}} \lesssim \lambda^{-\frac{2 m}{\alpha}-\frac{3}{2}}.
$$

For  $\mathcal{I}_2$, the asymptotic expansions of $J_{\frac{q}{2}-1}$ are not available near $0$. However, we can use the recursion formula  \eqref{ff12}  of Bessel functions to perform integration by parts, which will lead to a better decay. Specifically, by  \eqref{ff12}  we can write
$$\begin{aligned}
\cI_2(\lambda) =\left.\lambda ^{-\frac{q}{2}} \frac{(\lambda  r)^{\frac{q}{2}} J_{\frac{q}{2}}(\lambda  r)}{ \lambda}(1-\varphi)\left(1-r^\alpha\right)^{\frac{2 m}{\alpha}} \right|_0 ^1 
-\frac{ \lambda^{-\frac{q}{2}}}{\lambda } \int_0^1(\lambda  r)^{\frac{q}{2}} J_{\frac{q}{2}}(\lambda  r)\left(-\varphi_1-\varphi_2\right) d r,
\end{aligned}$$
where $\varphi_1=\varphi^{\prime}(r)\left(1-r^\alpha\right)^{\frac{2m}{\az}} $ and $ \varphi_2=2 m(1-\varphi(r)) r^{\alpha-1}\left(1-r^\alpha\right)^{\frac{2 m}{\alpha}-1} .$ Note that the first item  in RHS equals 0 by  \eqref{ff05} and the definition of $\vhi$, whence
\begin{equation*}
    \begin{aligned}
    \cI_2(\lambda) & =\lambda^{-\frac{q}{2}-1} \sum_{j=1}^2 \int_0^1(\lambda  r)^{\frac{q}{2}} J_{\frac{q}{2}}(\lambda  r) \varphi_j(r) d r     =: \sum_{j=1}^2 \cI_{2, j}(\lambda) .
    \end{aligned}
\end{equation*}

The estimation of $\cI_{2,1}$ is easy. Since $\supp\, \vhi_1\subset[\frac14,\frac12]$, there is no singularity in the amplitude. We then  define inductively 
$$
 \varphi_{1, k+1}(r):= (r^{-1} \varphi_{1, k})^{\prime}(r), \qquad \varphi_{1,1}:=\varphi_1, \qquad k\ge1,
 $$ and  set
 $$ \cI_{2,1}^{(k)}(\lambda):=(-1)^{k-1}  \lambda^{-\frac{q}{2}+1-2 k} \int_0^1(\lambda  r)^{\frac{q}{2}+k-1} J_{\frac{q}{2}+k-1}( \lambda r)\, \varphi_{1, k}(r)\, d r .$$
Utilizing  \eqref{ff12}  and integration by parts repeatedly we see that $\cI_{2,1}=\cI_{2,1}^{(k)},\, \fall k\ge1$. One can also easily verify that 
$$\left|\varphi_{1, k}(r)\right| \leq C_{\alpha, m, k}, \qquad \left|J_ {\frac{q}{2}+k-1}(\lz r)\right| \leq C_{q, k, 1}, \qquad \forall\, k \geq 1 ,$$
 which implies that
\begin{equation} \label{I21} 
    \left|\cI_{2,1}(\lambda)\right|=\left|\cI_{2,1}^{(k)}(\lambda)\right| \lesssim \lambda^{-\frac{q}{2}+1-2 k} \int_0^1 \lambda^{\frac{q}{2}+k-1} r^{\frac{q}{2}+k-1} d r \lesssim \lambda^{-k} , \quad \forall \,k \ge 1 .
\end{equation} 

To estimate $\cI_{2,2}$, more care is necessary for  singularities of the type $ r^{\az-1}$. Likewise, we define    
$$
 \varphi_{2, k+1}(r):= (r^{-1} \varphi_{2, k})^{\prime}(r),\quad \varphi_{2,1}:=\varphi_2, \quad k\ge1$$
  and  put 
\begin{equation*}
    \mathcal{I}_{2, 2}^{(k)}(\lambda):=(-1)^{k-1} \lambda^{-\frac{q}{2}+1-2 k} \int_0^1(\lambda r)^{\frac{q}{2}+k-1} J_{\frac{q}{2}+k-1}(\lambda r) \varphi_{2, k}(r) d r .
\end{equation*}
For any $k\ge1$, through direct  calculations  we have that $\cI_{2,2}=\cI_{2,2}^{(k)}$, and that
\begin{equation} \label{ff130} 
     \left|\varphi_{2, k}(r)\right| \le C_{\alpha, m, k}^{\prime}\  r^{\alpha+1-2 k}, \qquad \forall r \in(0,1). 
\end{equation}
Consequently, 
\begin{align*}
    \left|\cI_{2, 2}(\lambda)\right|=\left|\cI_{2,2}^{(k)}(\lambda)\right|  &\lesssim \lambda^{-k} \int_0^1 r^{\alpha+\frac{q}{2}-k} \left|J_{\frac{q}{2}+k-1}( \lambda r) \right\rvert\, d r \\
    & =:\lambda^{-k}\left(\int_0^{\lambda^{-1}}\cdots+\int_{\lambda^{-1}}^1\cdots\right).
\end{align*}
The first integral in the brace can be controlled by
\begin{equation*}
    \lambda^{\frac{q}{2}+k-1} \int_0^{\lambda^{-1}} r^{\alpha+q} d r \lesssim \lambda^{-\frac{q}{2}-\alpha-1+k},
\end{equation*}
if we invoke  \eqref{ff05}; while for the second one by  \eqref{ff06}  it yields that
\begin{equation*}
    \left|\int_{\lambda^{-1}}^{1}\cdots\right| \lesssim \lambda^{-\frac{1}{2}} \int_{\lambda^{-1}}^{1} r^{\alpha+\frac{q}{2}-\frac{1}{2}-k} d r \lesssim \lambda^{-\frac{q}{2}-\alpha-1+k} .
\end{equation*}
where we have chosen an $k>\az+\frac{q+3}2$. These estimates above give that $\left|\cI_{2,2}(\lambda)\right| \lsim \lambda^{-\frac{q}{2}-\alpha-1} ,$ from which, together with  \eqref{I21}, it follows that
\begin{equation*}
    \left|\cI_2(\lambda)\right| \leq\left|\cI_{2,1}(\lambda)\right|+\left|\cI_{2,2}(\lambda)\right| \lsim \lambda^{-\frac{q}{2}-\alpha-1} 
\end{equation*}
as required in  \eqref{Ig24}. 

Finally, we are left to the case that $\az\in2\nn$. Indeed at this moment, from the definition of $\varphi_{2, k}$,  \eqref{ff130} can be improved to
$$\left|\varphi_{2, k}(r)\right| \le C_{\alpha, m, k}^{\prime\prime}, \qquad\forall k \ge 1, \ \forall r \in(0,1). $$
Then a similar argument as in the estimation of $\cI_{2,1}^{(k)} $ will lead to that
$$\left|\cI_{2,2}(\lambda)\right|=\left|\cI_{2,2}^{(k)}(\lambda)\right| \lesssim \lambda^{-k+1},\qquad \fall k\ge1,$$
which shows that $\left|\cI_2(\lambda)\right|=O\left(\lambda^{-k}\right) $. This proves  \eqref{IG24'} and hence Lemma  \ref{lem24}.

\textbf{Proof of Lemma  \ref{lem23}.} Using  \eqref{std} again, we also only need to bound $\widehat{\chi_{B_1^ \alpha}}(0,s)$. Note that by  \eqref{ff01'} and  \eqref{Jv0}  
\begin{equation*}
    \begin{aligned}
    \widehat{\chi_{B_1^ \alpha}}(0, s)   = \frac{2\pi^{\frac{q}{2}+1}}{\Gamma\left(\frac{q}{2}+1\right)} |s|^{-\frac{m}{2}+1} \int_0^1 J_{\frac{m}{2}-1}(2 \pi|s| r) r^{\frac{m}{2}}\left(1-r^{\frac{\alpha}{2}}\right)^{\frac{q}{\alpha}} d r.
    \end{aligned}
\end{equation*}
The rest proof can be done similarly as in the $(w,0)$ case. We omit the details.

\section{Proof of  Lemma \ref{lem25}}  \label{pblm3} 
The proof for $(w,s)$ case is somewhat more complicated. By  \eqref{ff01}-\eqref{ff01'}  we introduce
\begin{gather}
\cJ(\lz):=     C_{q,m}\, \lz_1^{\frac{q}2-1} \lz_2^{\frac{m}2} \widehat{\chi_{B_1^{\alpha}}}(w, s) = \int_0^1 J_{\frac{q}{2}-1}\left(\lambda_1 r\right) J_{\frac{m}{2}}\left(\lambda_2\left(1-r^\alpha\right)^{\frac{2}{\alpha}}\right) r^{\frac{q}{2}}\left(1-r^\alpha\right)^{\frac{m}{\alpha}} d r,  \label{ff101}     \\
\cK(\lz) :=   C_{q,m} \,  \lz_1^{\frac{q}2} \lz_2^{\frac{m}2-1} \widehat{\chi_{B_1^{\alpha}}}(w, s) = \int_0^1 J_{\frac{m}{2}-1}\left(\lambda_2 r\right) J_{\frac{q}{2}}\left(\lambda_1\left(1-r^{\frac{\alpha}{2}}\right)^\frac{1}{\alpha}\right) r^{\frac{m}{2}}\left(1-r^{\frac{\alpha}{2}}\right)^{\frac{q}{2 \alpha}} d r,   \label{XBde2}
\end{gather}
where $C_{q,m}:=(2\pi)^{-\frac{q+m}{2}},$ and $\lz=(\lz_1,\lz_2) :=(2\pi|w|,2\pi|s|)$. 

To treat  \eqref{ff101}, the skill developed in Section  \ref{plem12}  becomes less effective since there occur two Bessel functions in the integrand.  A natural idea is to use a resolution of the identity, which covers the interval $(0,1)$ by an intermediate part and two end parts; the latter are usually neglectable or easy to control because of the short lengths;   while for the former, one resorts to the asymptotics of Bessel functions and reduces matters to the following type of oscillatory integral with singularities at two ends formally: 
\begin{equation} \label{OI1} 
    \int_0^1 e^{i \phi_\lz^{\pm}(r)} r^{\mu_1} (1-r)^{\mu_2} \varphi_0(r) \,dr,
\end{equation}
for some $\mu_1,\mu_2\in\rr$ and smooth $\varphi_0$, where  $\phi_\lz^{\pm}$ is defined by
\begin{gather} \label{pha1} 
    \phi_\lambda^{\pm}(r):=\lambda_1 r\pm\lambda_2\left(1- r^\alpha\right)^{\frac{2}{\alpha}},  \qquad r\in(0,1).
\end{gather}
Regarding  this kind of phase, the  Lemma  (\ref{vand} below) of van der Corput is a useful tool. The problem then is how to acquire the optimal decay, or, how to decide the appropriate order of derivatives of the phase. Noting that for $\phi_\lambda^{-}$, its derivative has a lower bound  $\lz_1$ for all  $\az>0$. For $\phi_\lambda^{+}$,
when $\az\notin(1,2)$ the phase has at most a (unique non-degenerate)  stationary point.
However, when $ \az\in(1,2), $ the first two derivatives of the phase can both vanish at somewhere in $(0,1)$. This is distinct from the foregoing case and demands the third derivative. 

On the other hand, due to the anisotropy of the $\mathcal{N}_{\alpha,M}$-ball $B_1^{\az,M}$, the two parameters $\lz_1$ and $\lz_2$ that we encounter   are in fact mutually  independent. Unfortunately, this aspect seems to have been overlooked by  \cite[]{GNT15} in their proof for (2.5) of  \cite[Lemma 2.5]{GNT15}, which is essential to \cite[Theorem (1)]{GNT15}. Yet there is a gap as $\lz_2/\lz_1\to\infty$; for example, on  \cite[p. 2216]{GNT15} the integral on $[0,\dz]$ is bounded above by $|w|^{-1/2}\dz^{d+1/2}$, but the later choice  $\dz=|w|^{-1} $ therein is not enough to yield   \cite[(2.5)]{GNT15}.
In fact, we find it  necessary to distinguish these cases. And a practical division is that:  $\lz_1\gsim \lz_2$ and $\lz_1\lsim \lz_2$, for which we shall use  \eqref{ff101} and  \eqref{XBde2}  respectively. In this process, the second case leads to an oscillatory integral  of   type  \eqref{OI1}, but with the  following new phase
\begin{gather} \label{pha2} 
    \psi_\lambda^{\pm}(r):=\lambda_2 r\pm\lambda_1\left(1-r^{\frac{\alpha}{2}}\right)^{\frac{1}{\alpha}}, \qquad r\in(0,1).
\end{gather}   
The benefit of this new expression we adopt is that it possesses an analytic structures analogous to those in the $\lz_1\gsim \lz_2$ case.

In Subsection  \ref{phAnl}, the above procedure will be made elaborately. Through these efforts, we conclude that it is possible to not only extend the spectral estimates of  \cite[]{GNT15}  for the ball $B_1^{\az}$ in Heisenberg groups to a larger framework, but also enhance  the key \cite[Lemma 2.3-2.5]{GNT15}  in both the range of $\az$ and the spectral decay rate, as in Lemmas  \ref{lem24}-\ref{lem25}. 

Now we recall the following weighted version of the van der Corput Lemma, which can be found in  \cite[\S VIII.1.2]{Ste93}, and a slight improvement has been made here; for a proof, one uses $F(x):=\int_x^b e^{i\lz\phi(t)} dt$ on the top of \cite[p. 334]{Ste93} instead. 
\begin{lemma}[van der Corput] \label{vand} 
Let $k\in\nn^*,\lz>0$ and the real-valued function $\phi\in C^\infty(a, b)$  satisfy that $\left|\phi^{(k)}(r)\right| \geq 1$ for all $r \in(a, b)$. If (i) $k\ge2$, or (ii) $k=1$ with $\phi'$  monotonic, then
\begin{equation*} 
    \left|\int_a^b e^{i \lambda \phi(r)} \psi(r) d r\right| \leq c_k \lambda^{-\frac{1}{k}}\left[\min\{|\psi(a)|,|\psi(b)|\}+\int_a^b\left|\psi^{\prime}(r)\right| d r\right],
\end{equation*}
where $c_k$ depends only on $k$, say  $c_k=5 \cdot 2^{k-1}-2$.
\end{lemma}

\subsection{Analysis of the phase functions}  \label{phAnl} 
For $\az>0$ we let
$$
\mathbf{C_\az}:=\begin{cases}
2(2-\alpha)^{\frac{2}{\alpha}-1}(\alpha-1)^{1-\frac{1}{\alpha}}, &\mathrm{if}\,\, 1<  \alpha <2,\\
2, &\mathrm{otherwise}. \\
\end{cases}
$$
The function $\phi_\lambda^{+}$ defined in  \eqref{pha1}  satisfies the following properties:
\begin{lemma} \label{cond1} 
Let $\az>0$ and $\lz_1\ge\caz\lz_2>0$.
\begin{enumerate}
\item If $\az\ge2$, then $(\phi_\lambda^{+})'$ is  decreasing
 and there exists $r_*=r_*(\az,\lz) $ such that
\begin{equation*}
    (\phi_\lambda^{+})'(r)\ge\frac{\lz_1}{2}\,\, \mathrm{on}\,\, (0,r_*), \qquad   |(\phi_\lambda^{+})''(r)|\ge\frac{\lz_1}{2}\,\, \mathrm{on}\,\, (r_*,1).
\end{equation*}

\item If $0<\az\le1$, then $(\phi_\lambda^{+})'$ is  increasing
 and with the same $r_*$ as in (i),
\begin{equation*}
    (\phi_\lambda^{+})'(r)\ge\frac{\lz_1}{2}\,\, \mathrm{on}\,\, (r_*,1), \qquad   |(\phi_\lambda^{+})''(r)|\ge\frac{\lz_1}{2}\,\, \mathrm{on}\,\, (0,r_*).
\end{equation*}

\item If $1<\az<2$, then there are three points $r_0'<r_0<r_0''$ located in $(0,1)$ depending only on $\az$ and a constant $D_\az>0$ such that
\begin{gather*}  \label{ff111} 
    (\phi_\lambda^{+})'(r)\ge D_\az{\lz_1} \ \ \mathrm{on}\ \ (0,1), \qquad  \mathrm{if}\ \lz_1\ge2\caz\lz_2, \\[2mm]   \inf_{(0,r_0')\cup(r_0'',1)} (\phi_\lambda^{+})' \ge D_\az{\lz_1}, \quad  \inf_{[r_0',r_0'']}|(\phi_\lambda^{+})'''| \ge D_\az{\lz_1}, \qquad \mathrm{if}\ \lz_1<2\caz\lz_2.
\end{gather*}
Additionally, $(\phi_\lambda^{+})'$ is decreasing on $(0,r_0)$ and increasing on $(r_0,1) $, respectively.
\end{enumerate} 
\end{lemma}

For $\psi_\lambda^{+}$ given in  \eqref{pha2}, similar property remains valid if $\az\notin(1,2)$, but when $\az\in(1,2)$  the situation behaves somewhat more degenerate.
\begin{lemma} \label{cond2} 
Let $\az>0$ and $0<\lz_1<\caz\lz_2$. 
\begin{enumerate}
\item If $\az\ge2$, then $(\psi_\lambda^{+})'$ is  decreasing
 and there exists $R_*=R_*(\az,\lz)$ such that
\begin{equation*}
    (\psi_\lambda^{+})'(r)\ge\frac{\lz_2}{2}\,\, \mathrm{on}\,\, (0,R_*), \qquad   |(\psi_\lambda^{+})''(r)| \ge \frac{\lz_2}4  \,\, \mathrm{on}\,\, (R_*,1).
\end{equation*}

\item If $0<\az\le1$, then $(\psi_\lambda^{+})'$ is  increasing
 and with the same $R_*$ as in (i),
\begin{equation*}
    (\psi_\lambda^{+})'(r)\ge\frac{\lz_2}{2}\,\, \mathrm{on}\,\, (R_*,1), \qquad   |(\psi_\lambda^{+})''(r)|\ge\frac{\lz_2}{4}\,\, \mathrm{on}\,\, (0,R_*).
\end{equation*}

\item If $1<\az<2$, then there are points $0<R_1\le R_0'<R_0<R_0''\le R_2<1$, with $R_0'$, $R_0$, $R_0''$ depending only on $\az$,  and a constant $D_\az'>0$ such that
\begin{gather*}
   \inf_{[R_1,R_2]}|(\psi_\lambda^{+})'| \ge D_\az'{\lz_2}, \quad \inf_{(0,R_1)\cup(R_2,1)}|(\psi_\lambda^{+})''| \ge D_\az'{\lz_2} \qquad \mathrm{if}\ \lz_2\ge 2\caz^{-1}\lz_1,\\[2mm]    \inf_{(0,R_0')\cup(R_0'',1)} |(\psi_\lambda^{+})''| \ge D_\az'{\lz_2}, \quad \inf_{[R_0',R_0'']}|(\psi_\lambda^{+})'''| \ge D_\az'{\lz_2}, \qquad \mathrm{if}\ \lz_2<2\caz^{-1}\lz_1.
\end{gather*}
Additionally, $(\psi_\lambda^{+})'$ is increasing on $(0,R_0)$ and decreasing on $(R_0,1) $ respectively.
\end{enumerate} 
\end{lemma}

Recall $\phi_\lambda^{-},\,\psi_\lambda^{-}$ defined in \eqref{pha1}-\eqref{pha2}. We have the better properties:
\begin{lemma} \label{cond3} 
Let $\lz>0$. Then 
\begin{equation*}
    (\phi_\lambda^{-})'(r)\ge\lz_1, \quad (\psi_\lambda^{-})'(r)\ge\lambda_2, \qquad \fall r\in(0,1).
\end{equation*}
Moreover, if  $\az\in(0,1]\cup[2,\infty)$, then both $(\phi_\lambda^{-})'$ and $(\psi_\lambda^{-})'$ are monotonic on $(0,1)$; if  $\az\in(1,2)$, then $(\phi_\lambda^{-})'$ is monotonic on $(0,(\az-1)^{1/\az})$ and $((\az-1)^{1/\az},1)$, and so is $(\psi_\lambda^{-})'$ with $(\az-1)^{1/\az}$ replaced by $(2-\az)^{2/\az}$.
\end{lemma}

{\bf Proof of Lemma \ref{cond1}.} A straightforward computation gives:
\begin{gather*}
  (\phi_\lz^{+})'(r)=\lambda_1-2 \lambda_2\left(1-r^\alpha\right)^{\frac{2}{\alpha}-1} r^{\alpha-1},\\
   (\phi_\lz^{+})''(r) = -2 \lambda_2\left(1-r^\alpha\right)^{\frac{2}{\alpha}-2} r^{\alpha-2}(\alpha-1-r ^\alpha) , \\
   (\phi_\lz^{+})'''(r)  = -2\lz_2 \left(1-r^\alpha\right)^{\frac{2}{\alpha}-3} r^{\alpha-3} [(\az-1-r^\az)(\az-2+\az r^\az)-\az(1-r^\az) r^\az  ]. 
\end{gather*}
When $\az\in(0,1]\cup[2,\infty)$,  the monotonicity of $(\phi_\lz^{+})'$ is obvious. Let $r_*=r_*(\az,\lz)\in(0,1)$ be the unique solution of the equation
\begin{equation} \label{r*} 
    \frac{\lambda_1}{4 \lambda_2}=\left(1-r^\alpha\right)^{\frac{2}{\alpha}-1} r^{\alpha-1}
\end{equation}
for $\az\in(0,1)\cup(2,\infty)$. While for $\az=1$,  $r_*$ still denotes the unique solution of  \eqref{r*} if  $\lz_1<4\lz_2$ and equals $0$ otherwise; for $\az=2$, $r_*$ still denotes the unique solution of  \eqref{r*} if  $\lz_1<4\lz_2$ and equals $1$ otherwise.

For $\az\ge2$, we have 
$$\left(\phi_\lambda^{+}\right)^{\prime}(r) \geq\left(\phi_\lambda^{+}\right)^{\prime}\left(r_*\right) \ge \frac{\lambda_1}{2}, \qquad\fall r\in\left(0, r_*\right) .$$
On $(r_*,1) $, since the function $r\mapsto \left(1-r^\alpha\right)^{\frac{2}{\az}-1} r^{\alpha-1}$ is increasing and $\az-1- r^\az \ge 1-r^\az$, 
then
$$\left|\left(\phi_\lambda^{+}\right)^{\prime \prime}(r)\right| \geq\left. 2 \lambda_2\left(1-r^\alpha\right)^{\frac{2}{\az}-1} r^{\alpha-1}\right|_{r=r_*}=\frac{\lambda_1}{2} .$$
For $\az\le1$, the rest conclusions of (ii) can be proved in a similar manner, with the simple fact that $r^\az+1-\az\ge r^\az$.  

When $\az\in(1,2)$, we let $r_0=(\alpha-1)^{\frac{1}{\alpha}} \in(0,1 )$, which is the unique stationary point of $ \left(\phi_\lambda^{+}\right)^{ \prime}$ on $(0,1) $. This implies the  monotonic properties in (iii) trivially. By calculations we have  
\begin{equation*}
    \left(\phi_\lambda^{+}\right)^{\prime}(r_0) =\lz_1-\caz\lz_2\ge0, \qquad \left(\phi_\lambda^{+}\right)^{\prime \prime\prime}(r_0) = 4\lz_2 (2-\alpha)^{\frac{2}{\alpha}-2}(\alpha-1)^{\frac{2 \alpha-3}{\alpha}} \alpha >0.
\end{equation*}
Hence there is an interval $(r_0',r_0'')\subset (0,1)$ containing $r_0$
and depending only on $\az$ such that
\begin{equation}  \label{ffl1} 
    \left(\phi_\lambda^{+}\right)^{\prime\prime \prime}(r)/\lz_2\sim_\az 1, \qquad \mbox{on}\ (r_0',r_0''),
\end{equation} 
due to the continuity of $ \left(\phi_\lambda^{+}\right)^{\prime\prime \prime}.$ 
We denote 
\begin{equation*}
    C_\alpha^{\prime}:= \lz_2^{-1}\max \left\{\lambda_1-\left(\phi_\lambda^{+}\right)^{\prime}\left(r_0^{\prime}\right), \lambda_1-\left(\phi_\lambda^{+}\right)^{\prime}\left(r_0^{\prime \prime}\right)\right\}>0.
\end{equation*}
It follows from the last assertion in (iii)  that $C_\az'<\caz$. And it is simple to verify    that $C_\az'$ indeed depends only on $\az$.  There are two possible cases: If $\lambda_1-\caz \lambda_2 \geq \frac12 \lambda_1,$ then $\left(\phi_\lambda^{+}\right)^{\prime}(r) \geq\left(\phi_\lambda^{+}\right)^{\prime}\left(r_0\right) \geq \frac12 \lambda_1$ on $(0,1)$. Conversely, if $\lambda_1-\caz \lambda_2 < \frac12 \lambda_1$ (so $\lz_1\sim_\az\lz_2$), then by the  definition of $C_\az'$ and the monotonic properties of $\left(\phi_\lambda^{+}\right)^{\prime } $ we obtain   
\begin{equation*}
    \left(\phi_\lz^{+}\right)^{\prime}(r) \ge \lambda_1-C_\alpha^{\prime} \lambda_2 \ge \lambda_1-\frac{C_\alpha'}{\caz} \lambda_1\sim_\az \lambda_1 ,\quad \fall r\in\left(0, r_0^{\prime}\right) \cup\left(r_0^{\prime \prime}, 1\right).
\end{equation*}
Lastly, by  \eqref{ffl1} it yields that $
    |\left(\phi_\lambda^{+}\right)^{\prime \prime\prime}(r)| \sim_\az \lz_2 \sim_\az \lz_1 $ on $(r_0',r_0'')$.
These lead to (iii). Lemma  \ref{cond1} is therefore proved. 

{\bf Proof of Lemma \ref{cond2}.} We calculate that 
\begin{gather*}
    \left(\psi_\lambda^{+}\right)^{\prime} (r)=\lambda_2-\frac{1}{2} \lambda_1\left(1-r^{\frac{\alpha}{2}}\right)^{\frac{1}{\alpha}-1} r^{\frac{\alpha}{2}-1},\\
    \left(\psi_\lambda^{+}\right)^{\prime\prime}(r) = -\frac{1}{4} \lambda_1\left(1-r^{\frac{\alpha}{2}}\right)^{\frac{1}{\alpha}-2} r^{\frac{\alpha}{2}-2}\left(\alpha-2+r^{\frac{\alpha}{2}}\right) ,\\
    \left(\psi_\lambda^{+}\right)^{\prime\prime\prime}(r) = -\frac{1}{4}\lz_1 r^{\frac{\alpha}{2}-3}\left(1-r^{\frac{\alpha}{2}}\right)^{\frac{1}{\alpha}-3}\left[\left(\frac{\az}2-2-(\az-2)r^{\frac{\az}2}\right)\left(\alpha-2+r^{\frac{\alpha}{2}}\right)+\frac{\alpha}{2} r^{\frac{\alpha}{2}}\left(1-r^{\frac{\alpha}{2}}\right)\right].
\end{gather*}
When $\az\in(0,1]\cup[2,\infty)$, the desired conclusions can be proved similarly as in  Lemma  \ref{cond1} (i)-(ii). We only mention some  cruxes  here: Firstly, the point we use is that $R_*=R_*(\az,\lz)\in(0,1)$, which is defined by the unique solution of the equation 
\begin{equation} \label{R*} 
    \frac{\lambda_2}{\lambda_1}= \left(1-r^\frac{\az}{2}\right)^{\frac{1}{\az}-1} r^{\frac{\alpha}2-1}
\end{equation}
for $\az\in(0,1)\cup(2,\infty)$; while for $\az=1 $, $R_*$ is still defined by  \eqref{R*} if $\lz_2>\lz_1$ and equals $1$ otherwise; for $\az=2 $, $R_*$ is still defined by  \eqref{R*} if $\lz_2>\lz_1$ and equals $0$ otherwise.
Secondly, for $\az\ge2$ it holds that $r^{-1}\left(\alpha-2+r^{\frac{\alpha}{2}}\right) \ge \frac{\az}{2}, \,\fall r\in(0,1)$. Thirdly, for $\az\le1$, one has $\left(1-r^{\frac{\alpha}{2}}\right)^{-1}\left(2-\alpha-r^{\frac{\alpha}{2}}\right) \geq 1 , \,\fall r\in(0,1)$.

When $\az\in(1,2)$, put $R_0=(2-\alpha)^{\frac{2}{\alpha}} \in(0,1 )$, the unique stationary point of $ \left(\psi_\lambda^{+}\right)^{ \prime}$ on $(0,1) $. Then  the  monotonic properties in (iii) hold by our calculations of the derivatives. It is easily seen that the equation $ \left(\psi_\lambda^{+}\right)^{\prime}(r) = \frac12 \left(\psi_\lambda^{+}\right)^{\prime}(R_0)$ or 
\begin{equation} \label{ffl3} 
    \frac{\lambda_2}{  \lambda_1} + \frac1{\caz}=\left(1-r^\frac{\az}{2}\right)^{\frac{1}{\az}-1} r^{\frac{\alpha}2-1}
\end{equation}
 has exactly two solutions $R_j=R_j(\az,\lz), j=1,2$ in $(0,1)$ with $R_1<R_0<R_2$. On the other hand, we get via a direct computation that
\begin{equation*} \label{f11} 
    \left(\psi_\lambda^{+}\right)^{\prime}(R_0) = \lambda_2-\frac{1}{\caz} \lambda_1>0 , \qquad -\left(\psi_\lambda^{+}\right)^{\prime\prime\prime}(R_0) =  \frac{\lz_1}{8} {\alpha}(\alpha-1)(2-\alpha)^{2-\frac{6}{\alpha}}>0 .
\end{equation*}
Now we  meet two situations: If $\lambda_2-\caz^{-1} \lambda_1 \geq \frac12 \lambda_2$ (so $\lz_2/\lz_1\ge2/\caz$), then
\begin{equation} \label{12R1} 
    \left(\psi_\lambda^{+}\right)^{\prime}(r) \geq \min\{\left(\psi_\lambda^{+}\right)^{\prime}\left(R_1\right),\left(\psi_\lambda^{+}\right)^{\prime}\left(R_2\right) \} =\frac12  \left(\psi_\lambda^{+}\right)^{\prime}(R_0) \ge \frac14\lz_2, \quad \fall r\in(R_1,R_2).
\end{equation}
We denote by $R_0'=R_0'(\az)$ and $R_0''=R_0''(\az)$ the two  solutions of  \eqref{ffl3} with LHS replaced by $2.5/\caz$.  Then we must have that $R_1<R_0'<R_0<R_0''<R_2 $.
Since the function $r \longmapsto\left(1-r^{\frac{\alpha}{2}}\right)^{\frac{1}{\alpha}-1} r^{\frac{\alpha}{2}-1}$ is decreasing on $(0,R_1)$ and increasing on  $(R_2,1) $, respectively, which is a direct consequence of the last assertion in (iii), then by  \eqref{ffl3}  we have, for all $r\in (0,R_1)\cup(R_2,1)$
\begin{align*}
    \left|\left(\psi_\lambda^{+}\right)^{\prime \prime}(r)\right| & =\frac{\lambda_1}{4}\left(1-r^{\frac{\alpha}{2}}\right)^{\frac{1}{\alpha}-2} r^{\frac{\alpha}{2}-2}\left|r^{\frac{\az}{2}}+\alpha-2\right| \nonumber\\
    & \ge \frac{\lambda_1}{4} \min_{r\in\{R_1,R_2\}} \left(1-r^{\frac{\alpha}{2}}\right)^{\frac{1}{\alpha}-1} r^{\frac{\alpha}{2}-1} \left|r^ \frac{\alpha}{2}+\alpha-2\right| \nonumber\\
    & \ge\frac{\lambda_2}{4}  \min_{r\in\{R_0',R_0''\}}\left|r^ \frac{\alpha}{2}+\alpha-2\right|  \sim_\az \lambda_2 . 
\end{align*}
 Conversely, if  $\lambda_2-\caz^{-1} \lambda_1 < \frac12 \lambda_2$ (so $\lz_1\sim_\az\lz_2$), from the continuity of the function $-\left(\psi_\lambda^{+}\right)^{\prime\prime \prime}(\cdot)/\lz_1$ we can choose $R_0'=R_0'(\az)$ and $R_0''=R_0''(\az)$ properly, such that $R_0\in(R_0',R_0'') $ and 
 \begin{equation}  \label{ffl2} 
     -\left(\psi_\lambda^{+}\right)^{\prime\prime \prime}(r)/\lz_1\sim_\az 1, \qquad \mathrm{for\ all}\,\, r\in [R_0',R_0''].
 \end{equation}
For later use and  notational convenience,  we also define
\begin{equation} \label{ffa1} 
    R_1:=R_0', \quad R_2:=R_0''.
\end{equation}
Fixing our choices, then for $\az\in(1,2)$ we obtain
\begin{equation*}
    |\left(\psi_\lz^{+}\right)^{\prime\prime} (r)| \ge  \frac{\lambda_1}{4} \min_{r\in\{R_0',R_0''\}} \left(1-r^{\frac{\alpha}{2}}\right)^{\frac{1}{\alpha}-1} r^{\frac{\alpha}{2}-1} \left|r^ \frac{\alpha}{2}+\alpha-2\right|  \sim_\az \lambda_2, \quad \fall r\in (0,R_0')\cup(R_0'',1).
\end{equation*}
These estimates imply (iii) and hence the proof of Lemma   \ref{cond2} is completed.  

{\bf Proof of Lemma \ref{cond3}.} This lemma can be proved in a simpler way and by some slight modifications of the partial proofs of Lemmas  \ref{cond1} and  \ref{cond2}. So we shall omit the details.

\subsection{Proof of  \eqref{ff18}}  \label{prb3a} 
We will divide the discussion into the following two cases: {Case (i): $\lz_1\ge \caz\lz_2$} and {Case (ii): $\lz_1 < \caz\lz_2$}.

\subsubsection{Proof of Case (i): $\lz_1\ge \caz\lz_2$} \label{prfM} 
Note that we now have $\lz_1\sim_\az |\lz|$, which will be used implicitly in what follows. Applying integration by parts with  \eqref{ff12} we have 
\begin{equation*}
    \begin{aligned}
    \cJ(\lambda) & =\lambda_1^{-\frac{q}{2}-1} \int_0^1\left[\left(\lambda_1 r\right)^{\frac{q}{2}} J_{\frac{q}{2}}\left(\lambda_1 r\right)\right]' J_{\frac{m}{2}}\left(\lambda_2\left(1-r^\alpha\right)^{\frac{2}{\alpha}}\right)\left[\left(1-r^\alpha\right)^{\frac{2}{\alpha}}\right]^{\frac{m}{2}} d r \\
    & =2 \lambda_1^{-1} \lambda_2 \int_0^1 J_{\frac{q}{2}}\left(\lambda_1 r\right) J_{\frac{m}{2}-1}\left(\lambda_2\left(1-r^\alpha\right)^{\frac{2}{\alpha}}\right)\left(1-r^\alpha\right)^{\frac{m+2}{\alpha}-1} r^{\frac{q}{2}+\alpha-1} d r \\
    & =2 \lambda_1^{-1} \lambda_2 \int_0^1 \cdots(1-\varphi) d r+ 2 \lambda_1^{-1} \lambda_2\int_0^1 \cdots \varphi d r \\
    &=: 2 \sum_{k=1}^{2} \cJ_k(\lambda) ,
    \end{aligned}
\end{equation*}
where $\vhi$ is defined by  \eqref{vhi} as before. In view of this, we merely have to show that 
\begin{equation*}
    |\cJ_1(\lz)|\lsim |\lz|^{-1-\sz},\qquad |\cJ_2(\lz)| \lsim |\lz|^{-1-\sz}
\end{equation*} 
holds uniformly for large $|\lz|$.

{\bf Estimation of $\cJ_1$.} Before going  further, we remark that the methods in  this part actually provide a paradigm, which will be  transplanted to other situations. Core portions in the  proof of  \eqref{ff19} are also performed in such a way, with  more delicate discussions. 

Now we write 
\begin{align*}
  \cJ_1(\lz) &=  \lambda_1^{-1} \lambda_2 \int_0^1 J_{\frac{q}{2}}\left(\lambda_1 r\right) J_{\frac{m}{2}-1}\left(\lambda_2\left(1-r^\alpha\right)^{\frac{2}{\alpha}}\right)\left(1-r^\alpha\right)^{\frac{m+2}{\alpha}-1} r^{\frac{q}{2}+\alpha-1} (1-\vhi) d r \\
  &= \int_0^\dz\cdots + \int_\dz^1 \cdots \\
  &=:\cJ_{1,1}(\lz) + \cJ_{1,2}(\lz).
\end{align*}
Here $\dz=\lambda_1^{-\frac{1}{q+2 \alpha-1}} \in(\lz_1^{-1},1)$. The reason of this choice will be clear soon.

For  $\cJ_{1,1}$, from  \eqref{ff05}-\eqref{ff06}  we see that, for all $r \in\left[0, 1/2\right]$,
\begin{equation*}
    \left|J_{\frac{q}{2}}\left(\lambda_1 r\right)\right| \lesssim\left\{\begin{array}{ll}
    \left(\lambda_1 r\right)^{\frac{q}{2}}, & r \le \lambda_1^{-1} , \\
    \left(\lambda_1 r\right)^{-\frac{1}{2}}, & r \geq \lambda_1^{-1} ,
    \end{array}
    \qquad \left|J_{ \frac{m}{2}-1} \left(\lambda_2\left(1-r^\alpha\right)^{\frac{2}{\alpha}}\right)\right| \lsim \lambda_2^{-\frac{1}{2}}.\right.
\end{equation*}
Then by the fact that $\supp (1-\vhi)\cap[0,1]\subset [0,1/2]$ it yields 
\begin{equation*}
    \begin{aligned}
    \left|\cJ_{1,1}(\lambda)\right| & \lesssim \lambda_1^{-1}  \lambda_2 \left(\lambda_1^{\frac{q}{2}} \lambda_2^{-\frac{1}{2}} \int_0^{\lambda_1^{-1}} r^ {q+\alpha-1} dr + \lambda_1^{-\frac{1}{2}} \lambda_2^{-\frac{1}{2}} \int_{\lambda_1^{-1}}^\delta r^{\frac{q}{2}+\alpha-\frac{3}{2}} d r\right) \\
    & \lesssim \lambda_1^{-1} \lambda_2\left(\lambda_2^{-\frac{1}{2}} \lambda_1^{-\frac{q}{2}-\alpha}+ \lambda_1^{-\frac{1}{2}} \lambda_2^{-\frac{1}{2}} \delta^{\frac{q+2 \alpha-1}{2}}\right)  \\
    & \lesssim \lz_1^{-\frac32} \sim |\lz|^{-\frac32}. \end{aligned}
\end{equation*}
Note that we have chosen $\dz$ as above to balance the quantities in the second the line.

For  $\cJ_{1,2}$, since $\supp  \vhi\cap[0,1]\subset [1/4,1]$, by our choice we have  $\lz_1\dz\gsim1$, whence one can substitute the expansion \eqref{ff11} into its formula and obtain that
\begin{equation*}
    \begin{aligned}
    |\cJ_{1,2}(\lz)| &\lsim  \lambda_1^{-1} \lambda_2\cdot
      \lambda_1^{-\frac{1}{2}} \lambda_2^{-\frac{1}{2}}\left[\sum_{ \pm}\left|\int_\delta^1 e^{i \phi_\lambda^{\pm}(r)} r^{\frac{q}{2}+\alpha-\frac{3}{2}}\left(1-r^\alpha\right)^{\frac{m+1}{\alpha}-1}(1-\varphi) d r\right|\right. \\
    &\qquad+\lambda_1^{-1} \int_\delta^1 r^{\frac{q}{2}+\alpha-\frac{5}{2}}\left(1-r^\alpha\right)^{\frac{m+1}{\alpha}-1}(1-\varphi) d r \\
    &\qquad\left.+\lambda_2^{-1} \int_\delta^1 r^{\frac{q}{2}+\alpha-\frac{3}{2}}\left(1-r^\alpha\right)^{\frac{m-1}{\alpha}-1}(1-\varphi) d r\right].
    \end{aligned}
\end{equation*}
The third term in the bracket is seen to be controlled by $\lz_2^{-1}$; while the second can be bounded by $ \lz_1^{-1/2}$, since 
\begin{align*}
    \int_\delta^1 r^{\frac{q}{2}+\alpha-\frac{5}{2}} d r \lsim \begin{cases}1, &\mbox{as}\ \alpha>\frac{3-q}{2}, \\ \lambda_1^{\frac{1}{2}-\alpha} \log   \lambda_1, &\mbox{as}\ q=2,  \alpha \leq \frac{1}{2} .\end{cases}
\end{align*}

Turning to the first term (designated by $FI(\lz)$), we will employ  Lemma  \ref{vand}  of  van der Corput  extensively, and so as  to improve the integrability of the amplitude in some singular case, the skill of integration by parts (with the differential property \eqref{ff12} of Bessel functions) will also play an important part.

Note that $\phi_\lambda^{-}$ behaves better than $\phi_\lambda^{+}$ as in  Lemma \ref{cond3}, subsequently we will merely deduce the case with the latter phase. Once this is done, the former case can be proved similarly and more simply. This step will be proceeded repeatedly by default in such situation throughout Subsection \ref{prb3a}, without extra explanation. Similar convention will be adopted to Subsection \ref{prb3b}.

If $\az\ge \frac{3-q}{2}$, according to Lemmas  \ref{vand} and   \ref{cond1}  we have
\begin{equation} \label{FI1} 
    FI(\lz) \lesssim \lambda_1^{-\sigma}\left(1+\int_\delta^1\left|\left[r^{\frac{q}{2}+\alpha-\frac{3}{2}}\left(1-r^\alpha\right)^{\frac{m+1}{\alpha}-1}(1-\varphi)\right]^{\prime}\right| d r\right),
\end{equation}
where $\sz$ is given by  \eqref{sig}. Then $FI(\lz)\lsim \lambda_1^{-\sigma} $. Otherwise, it must hold that $\az< \frac{3-q}{2}$  or equivalently,  $q=2$ and $\az<1/2$ (since $q\ge2$ as explained in Subsection  \ref{S2G}). This can be split into two possible cases: $\dz\ge r_*$ and $\dz< r_*$, where $r_*$ is defined by  \eqref{r*}. Note that $\sz=1/2$ at this moment.

\textit{(1) Let $q=2, \az<1/2$ and $\dz\ge r_* $.} By Lemma \ref{cond1} (ii), the decay $\lambda_1^{-\sigma} $ in  \eqref{FI1} can be improved to $\lz_1^{-1}$. However,
\begin{equation*}
    \int_\delta^1\left|\left[r^{\frac{q}{2}+\alpha-\frac{3}{2}}\left(1-r^\alpha\right)^{\frac{m+1}{\alpha}-1}(1-\varphi)\right]^{\prime}\right| d r \lesssim \int_\delta^1 r^{\alpha-\frac{3}{2}} d r \sim \delta^{\alpha-\frac{1}{2}} \le \lambda_1^{\frac{1}{2}-\alpha},
\end{equation*}
which implies that $FI(\lz)\lsim \lambda_1^{-\frac{1}{2}}$. Through these estimates we arrive at that $|\cJ_{1,2}(\lz)| \lsim |\lz|^{-1-\sz}$.

\textit{(2) Let $q=2, \az<1/2$ and $\dz< r_* $.}   In this case, by  \eqref{r*} one has  $ r_* \sim\left(\lambda_1 \lambda_2^{-1}\right)^{-\frac{1}{1-\alpha}}$, from which, the condition $\dz< r_* $ readily means that
\begin{equation} \label{FI2} 
    \lambda_1 \lsim \lambda_2^ {\frac{1+2 \alpha}{3 \alpha}}.
\end{equation}
Additionally, we need a new expression for $\cJ_{1,2}$.
Using the recursion formula  \eqref{ff12}  we obtain
\begin{equation*}
    \begin{aligned}
    \cJ_{1,2}(\lambda) & = -\frac{1}{2} \lambda_1^{-1} \int_\delta^1\left[J_{\frac{m}{2}}\left(\lambda_2\left(1-r^\alpha\right)^{\frac{2}{\alpha}}\right)\left(1-r^\alpha\right)^{\frac{m}{\alpha}}\right]^{\prime} J_{\frac{q}{2}}\left(\lambda_1 r\right) r^{\frac{q}{2}}(1-\varphi) d r \\
    &= \frac{1}{2} \lambda_1^{-1} J_{\frac{m}{2}}\left(\lambda_2\left(1-\delta^\alpha\right)^{\frac{2}{\alpha}}\right)\left(1-\delta^\alpha\right)^{\frac{m}{\alpha}} J_{\frac{q}{2}}\left(\lambda_1 \delta\right) \delta^{\frac{q}{2}}(1-\varphi(\delta)) \\
    &\qquad -\frac{1}{2} \lambda_1^{-1} \int_\delta^1 J_{\frac{m}{2}}\left(\lambda_2\left(1-r^\alpha\right)^{\frac{2}{\alpha}}\right)\left(1-r^\alpha\right)^{\frac{m}{\alpha}} J_{\frac{q}{2}}\left(\lambda_1 r\right) r^{\frac{q}{2}} \varphi' d r \\
    &\qquad +\frac{1}{2} \int_\delta^1 J_{\frac{m}{2}}\left(\lambda_2\left(1-r^\alpha\right)^{\frac{2}{\alpha}}\right)\left(1-r^\alpha\right)^{\frac{m}{\alpha}} J_{\frac{q}{2}-1}\left(\lambda_1 r\right) r^{\frac{q}{2}}(1-\varphi) d r.
    \end{aligned}
\end{equation*}
The first resulting term has the upper bound $ \lambda_1^{-1} \lambda_2^{-\frac{1}{2}} \delta^{\frac{q}{2}} \cdot\left(\lambda_1 \delta\right)^{-\frac{1}{2}} \sim \lambda_1^{-\frac{3}{2}} \delta^{\frac{q-1}{2}} .$ Since $\operatorname{supp} \varphi^{\prime} \subseteq\left[\frac{1}{4}, \frac{1}{2}\right]$, then the second term can be controlled by $$ \lambda_1^{-1} \int_{\frac{1}{4}}^{\frac{1}{2}} \lambda_1^{-\frac{1}{2}} \lambda_2^{-\frac{1}{2}} d r \lesssim \lambda_1^{-\frac{3}{2}} \lambda_2^{-\frac{1}{2}} .$$
Regarding the third  term (denoted by $T_3(\lz)$), plugging  \eqref{ff11} into it (with $1$-term and $N$-term expansions for $J_{\frac{q}{2}-1}\left(\lambda_1 r\right)$ and $J_{\frac{m}{2}}(\lambda_2\left(1-r^\alpha\right)^{2/\az})$ respectively) produces that
\begin{equation*}
    \begin{aligned}
    |T_3(\lz)|&\lsim  \lambda_1^{-\frac{1}{2}} \lambda_2^{-\frac{1}{2}} \left[\sum_{\pm} \sum_{j=0}^{N-1}\left|\int_\delta^1 e^{i \phi_\lambda^{\pm}(r)} r^{\frac{q-1}{2}}\left(1-r^\alpha\right)^{\frac{m-1}{\alpha}}\left(\lambda_2\left(1-r^\alpha\right)^{\frac{2}{\alpha}}\right)^{-j}(1-\varphi) d r\right|\right. \\
    &\qquad+  \left.\int_\dz^1 r^{\frac{q-1}{2}}\left(1-r^\alpha\right)^{\frac{m-3}{\alpha}}\left(\lambda_2\left(1-r^\alpha\right)^{\frac{2}{\alpha}}\right)^{-N}(1-\varphi) d r \,\right. \\
    &\qquad \left.+\lambda_1^{-1} \int_\dz^1 r^{\frac{q-3}{2}}\left(1-r^\alpha\right)^{\frac{m-1}{\alpha}}(1-\varphi) d r\right] \\
    &=:\sum_{j=1}^3 T_{3,j}(\lz)
    \end{aligned}
\end{equation*}
where $N\ge2$ is to be chosen. Then by  \eqref{FI2} and fixing an integer $N>\frac{1+2\az}{3\az}$   we have
\begin{gather*}
     T_{3,2}(\lz)\lsim \lambda_1^{-\frac{1}{2}} \lambda_2^{-\frac{1}{2}}\cdot \lambda_2^{-N} \int_\delta^1 r^{\frac{q-1}{2}} d r \lesssim \lz_1^{-\frac12}\lambda_2^{-N}\lsim \lambda_1^{-\frac32}, \\
    T_{3,3}(\lz)\lsim \lambda_1^{-\frac{1}{2}} \lambda_2^{-\frac{1}{2}}\cdot \lambda_1^{-1} \int_\delta^1 r^{\frac{q-3}{2}} d r \lesssim \lambda_1^{-\frac32}.
\end{gather*} 
For $T_{3,1}(\lz)$, since $\dz<r_*$, the integral inside is written as 
$$\int_\dz^1\cdots= \int_{r_*}^1\cdots + \int_\dz^{r_*}\cdots,$$
whose associated quantities are denoted by $T_{3,1}'$ and $T_{3,1}''$ respectively, so that $T_{3,1}= T_{3,1}'+ T_{3,1}'' $.
For $T_{3,1}'$, one acquires by using again Lemma  \ref{cond1} (ii) that
\begin{align*}
T_{3,1}'(\lz) & \lesssim \lambda_1^{-\frac{1}{2}} \lambda_2^{-\frac{1}{2}}  \sum_{j=0}^{N-1} \lambda_1^{-1} \lambda_2^{-j} \left(1+\int_{r_*}^{1} \left\lvert\,\left[r^{\frac{q-1}{2}}\left(1-r^\alpha\right)^{\frac{m-1}{\alpha}}\left(\left(1-r^\alpha\right)^{\frac{2}{\alpha}}\right)^{-j}(1-\varphi)\right]'\right| d r\right) \nonumber\\
&\lsim \lambda_1^{-\frac{1}{2}} \lambda_2^{-\frac{1}{2}}\sum_{j=0}^{N-1} \lambda_1^{-1} \lambda_2^{-j}\left(1+\int_{r_*}^{1} r^{\frac{q-1}{2}-1} d r\right) \\
&\lesssim \lambda_1^{-\frac32} .  \label{Lff36} 
\end{align*}
For   $T_{3,1}''$,  noting that at present $ r_*^{-1} \sim\left({\lambda_1}{\lambda_2^{-1}}\right)^{\frac{1}{1-\alpha}} \gtrsim {\lambda_1}{\lambda_2^{-1}}$, then the second estimation in Lemma  \ref{cond1} (ii) can be optimized as follows: 
\begin{equation*}
    \left|\left(\phi_{\lambda}^{+}\right)''(r)\right| \geqslant\left.  \lambda_2\left(1-r^\alpha\right)^{\frac{2}{\alpha}-2} r^{ \alpha-1} \cdot r^{-1}\right|_{r=r_*} \gtrsim \lambda_1^2 \lambda_2^{-1}, \quad \fall r\in(0,r_*).
\end{equation*}
Hence from the van der Corput Lemma it follows that
\begin{equation*}
    \begin{aligned}
    T_{3,1}''(\lz) & \lesssim \lambda_1^{-\frac{1}{2}} \lambda_2^{-\frac{1}{2}} \sum_{j=0}^{N-1} \lambda_2^{-j}\left(\lambda_1^2 \lambda_2^{-1}\right)^{-\frac{1}{2}}\left[1+\int_\delta^{r_*} \left\lvert\,\left[r^{\frac{q-1}{2}}\left(1-r^\alpha\right)^{\frac{m-1-2j}{\alpha}}    (1-\varphi)\right]^{\prime} \right\lvert d r\right] \\
    & \lesssim \lambda_1^{-\frac{1}{2}} \lambda_2^{-\frac{1}{2}}\sum_{j=0}^{N-1} \lambda_2^{-j}\left(\lambda_1^2 \lambda_2^{-1}\right)^{-\frac{1}{2}}\left(1+\int_\delta^{r_*} r^{\frac{q-1}{2}-1} d r\right) \\
    & \lesssim \lambda_1^{-\frac{1}{2}} \lambda_2^{-\frac{1}{2}}\sum_{j=0}^{N-1} \lambda_2^{-j}\left(\lambda_1^2 \lambda_2^{-1}\right)^{-\frac{1}{2}} \\
    &\lesssim \lambda_1^{-\frac32}.
    \end{aligned}
\end{equation*} 
Finally, putting these estimates together, we conclude  that $|\cJ_1(\lz)| \lsim |\lz|^{-1-\sz}$ as required.

{\bf Estimation of $\cJ_2$.} The method is analogous to that  of  $\cJ_1$.
We first write 
\begin{align*}
  \cJ_2(\lz) &=  \lambda_1^{-1} \lambda_2 \int_0^1 J_{\frac{q}{2}}\left(\lambda_1 r\right) J_{\frac{m}{2}-1}\left(\lambda_2\left(1-r^\alpha\right)^{\frac{2}{\alpha}}\right)\left(1-r^\alpha\right)^{\frac{m+2}{\alpha}-1} r^{\frac{q}{2}+\alpha-1}  \vhi(r) d r \\
  &= \int_0^{r_\dz}\cdots + \int_{r_\dz}^1 \cdots \\
  &=:\cJ_{2,1}(\lz) + \cJ_{2,2}(\lz),
\end{align*}
with $ \delta=\bar{c}\, \lambda_2^{-\frac{1}{m+1}}<1$ ($\bar{c}<(2\pi c_2)^{1/(m+1)}$ is a  positive constant and $c_2$ is assumed already in Lemma  \ref{lem25}), where we have put $r_\delta=\left(1-\dz^{\frac{\alpha}{2}}\right)^{\frac{1}{\alpha}} \sim_{\az,c_2,m} 1.$
         
For $\cJ_{2,2}$, observing that $ \supp\,\vhi \cap[0,1]\subset [1/4,1]$. Then by the choice of $\dz$ and  \eqref{ff05}-\eqref{ff06} we have
\begin{equation*}
    \begin{aligned}
     \left|\cJ_{2,2}(\lambda)\right| &\lesssim \lambda_1^{-1} \lambda_2\cdot \lambda_1^{-\frac{1}{2}} \int^1_{1-\lambda_2^{-\frac{\alpha}{2}}} \left(\lambda_2(1-r)^{\frac{2}{\alpha}}\right)^{\frac{m}{2}-1}(1-r)^{\frac{m+2}{\alpha}-1} d r \\
    &\qquad  +\lambda_1^{-1} \lambda_2 \cdot \lambda_1^{-\frac{1}{2}}\left|\int_{r_\delta}^{1-\lambda_2^{-\frac{\alpha}{2}}}\left(\lambda_2(1-r)^{\frac{2}{\alpha}}\right)^{-\frac{1}{2}}(1-r)^{\frac{m+2}{\alpha}-1} d r\right|\\
    &\lsim \lz_1^{-\frac32} \sim |\lz|^{-\frac32}.
    \end{aligned}
\end{equation*}

To dominate   $\cJ_{2,1}$, we use  \eqref{ff11} and  Lemma  \ref{vand}  to obtain
\begin{equation*}
    \begin{aligned}
\left|\cJ_{2,1}(\lambda)\right| &\lsim \lambda_1^{-1} \lambda_2\cdot \lambda_1^{-\frac{1}{2}} \lambda_2^{-\frac{1}{2}}\left[\sum_{ \pm}\left|\int_0^{r_\dz} e^{i \phi_\lambda^\pm(r)} r^{\frac{q}{2}+\alpha-\frac{3}{2}}\left(1-r^\alpha\right)^{\frac{m+1}{\alpha}-1} \varphi(r) d r\right|\right. \\
    &\qquad +\lambda_1^{-1} \int_0^{r_\dz} r^{\frac{q}{2}+\alpha-\frac{5}{2}}\left(1-r^\alpha\right)^{\frac{m+1}{\alpha}-1} \varphi(r) d r \\
    &\qquad \left. + \lambda_2^{-1} \int_0^{r_\dz} r^{\frac{q}{2}+\alpha-\frac{3}{2}}\left(1-r^\alpha\right)^{\frac{m-1}{\alpha}-1} \varphi(r) d r\right].
    \end{aligned}
\end{equation*}
The second and third resulting terms in the bracket are well controlled by $ 
\lz_1^{-1}$ and $ 
\lz_2^{-1}\log\lz_2$ respectively. Regarding the first term (denoted by $FI(\lz)$ as before), using Lemma  \ref{vand} again,  
\begin{equation} \label{FI'1} 
    F I(\lz) \lesssim 
     \lambda_1^{-\sigma}\left(1+\int_0^{r_\dz}\left|\left[r^{\frac{q}{2}+\alpha-\frac{3}{2}}\left(1-r^\alpha\right)^{\frac{m+1}{\alpha}-1} \varphi(r)\right]^{\prime}\right| d r\right) .
\end{equation}
As a result, if $\az\le m+1$, then one sees that $F I(\lz) \lesssim \lambda_1^{-\sigma}$, which implies $\left|\cJ_{2,1}(\lambda)\right| \lsim \lambda_1^{-\frac32}  $. As for the rest case that $\az>m+1$ (so $\az>2$ and $\sz=1/2$), there are two possible subcases: $r_\dz\le r_*$  and $ r_\dz>r_*$, to which we next turn. 

\textit{(1) Let $\az> m+1$ and $r_\dz\le r_*$.} Note that
\begin{equation} \label{FI'5} 
    \left(1-r_*^\alpha\right)^{\frac{2}{\alpha}} \sim\left(\frac{\lambda_1}{\lambda_2}\right)^{\frac{2}{2-\alpha}}
\end{equation}
by  \eqref{r*}, then $\lambda_1 \gtrsim \lambda_2^{\frac{\alpha+2 m}{2+2 m}}$. And by Lemma  \ref{cond1} (i), the factor $\lambda_1^{-\sigma}$ in \eqref{FI'1} can be improved to $\lz_1^{-1}$, whence 
$$FI(\lz)\lsim 
\lambda_1^{-1} \delta^{\frac{m+1-\alpha}{2}} \sim  \lambda_1^{-1} \lambda_2^{\frac{\alpha-m-1}{2+2 m}} \lesssim \lambda_2^{-\frac{1}{2}},$$ 
which also leads to the desired estimate for $\cJ_{2,1}(\lz)$.

\textit{(2) Let $\az> m+1$ and $r_\dz> r_*$.} Note also at this point it holds that
\begin{equation} \label{FI'3} 
    \lambda_1 \lsim \lambda_2^{\frac{\alpha+2 m}{2+2 m}}.
\end{equation}
A new expression of $\cJ_{2,1}$ is derived from an integration by parts with  \eqref{ff12}  as follows:
\begin{equation*}
    \begin{aligned}
    \cJ_{2,1}(\lambda)= & \left. -\frac{1}{2} \lambda_1^{-1}\left[J_ {\frac{m}{2}}\left(\lambda_2\left(1-r_\delta^\alpha\right)^{\frac{2}{\alpha}}\right)\left(1-r_\delta^\alpha\right)^{\frac{m}{\alpha}} J_{\frac{q}{2}} \left(\lambda_1 r_\delta\right) r^{\frac{q}{2}} \varphi\right]   \right|_{r=r_\delta} \\
    &\qquad +\frac{1}{2} \lambda_1^{-1} \int_0^{r_\delta} J_{ \frac{m}{2}}\left(\lambda_2\left(1-r^\alpha\right)^{\frac{2}{\alpha}}\right)\left(1-r^\alpha\right)^{\frac{m}{\alpha}} J_{\frac{q}{2}}\left(\lambda_1 r\right) r^{\frac{q}{2}} \varphi' d r \\
    &\qquad +\frac{1}{2} \int_0^{r_\delta} J_{ \frac{m}{2}}\left(\lambda_2\left(1-r^\alpha\right)^{\frac{2}{\alpha}}\right)\left(1-r^\alpha\right)^{\frac{m}{\alpha}} J_{\frac{q}{2}-1}\left(\lambda_1 r\right) r^{\frac{q}{2}} \varphi d r.
    \end{aligned}
\end{equation*}
The first two terms are bounded by $\lz_1^{-\frac32}$ as easily verified. For the third term (denoted by $T_3(\lz)$), applying   \eqref{ff11} with $1$-term and $N$-term expansions for $J_{\frac{q}{2}-1}\left(\lambda_1 r\right)$ and $J_{\frac{m}{2}}(\lambda_2\left(1-r^\alpha\right)^{2/\az})$ respectively gives
\begin{equation*}
    \begin{aligned}
     \left|T_3(\lz)\right| &\lsim  \lambda_1^{-\frac{1}{2}} \lambda_2^{-\frac{1}{2}}\left[\sum_{ \pm} \sum_{j=0}^{{N-1}} \lambda_2^{-j}\left|\int_0^{r_\dz} e^{i \phi_\lambda^{ \pm}(r)} r^{\frac{q}{2}-\frac{1}{2}}\left(1-r^\alpha\right)^{\frac{m-1-2 j}{\alpha}} \varphi(r) d r\right|\right. \\
    &\qquad +\lambda_1^{-1} \int_0^{r_\dz} r^{\frac{q}{2}-\frac{3}{2}}\left(1-r^\alpha\right)^{\frac{m-1}{\alpha}} \varphi(r) d r  \\
    &\qquad \left.+  \int^{r_\dz}_0 r^{\frac{q}{2}-\frac{1}{2}}\left(1-r^\alpha\right)^{\frac{m-3}{\alpha}}\left(\lambda_2\left(1-r^\alpha\right)^{\frac{2}{\alpha}}\right)^{-N} \varphi(r) d r\right] ,
    \end{aligned}
\end{equation*}
where $N\in\nn^*$ is to be chosen. In fact, the second term in the bracket is bounded by $\lz_1^{-1}$; by fixing $N>  \frac{\alpha+2 m}{1+ m}$ and using  \eqref{FI'3} we see the last one can be majorized by $ \lambda_2^{-\frac12 N} \lesssim \lambda_1^{-1} .$ Thus, one can reduce matters to showing that
\begin{equation*}
   \left|\int_0^{r_\dz} e^{i \phi_\lambda^{ \pm}(r)} r^{\frac{q}{2}-\frac{1}{2}}\left(1-r^\alpha\right)^{\frac{m-1-2 j}{\alpha}} \varphi(r) d r\right| \lsim \lz_1^{-1} \lz_2^{\frac{j}2+\frac12}.
\end{equation*}
This can be done as follows. Firstly, note that $$\left(1-r^\alpha\right)^{\frac{2}{\alpha}}\ge\dz\gsim \lz_2^{-\frac12}, \quad \mbox{for all}\ r\in(0,r_\dz).$$
Secondly, split the integral into two parts: $(0,r_*)$ and $(r_*,r_\dz)$. Lastly, for $(0,r_*)$ part we employ  Lemmas  \ref{vand} and \ref{cond1} (i); while for $(r_*,r_\dz)$ part, we use again Lemma  \ref{vand} with the  following fact (by  \eqref{FI'5})
\begin{equation*}
    \left|\left(\phi_\lambda^{+}\right)^{\prime \prime}(r) \right| \geq\left. \left(2(\az-2) \lambda_2\left(1-r^\alpha\right)^{\frac{2}{\az}-1} r^{\alpha-1}\cdot \left(1-r^\alpha\right)^{-1} \right)\right|_{r=r_*}=\frac{\az-2}{2} \lambda_1 \cdot (1-r_*^\az)^{-1}\gsim \lz_1^2\lz_2^{-1}.
\end{equation*}
which is an improvement of the second estimation in Lemma  \ref{cond1} (i). As a result, we conclude that $\cJ_2(\lz)\lsim |\lz|^{-1-\sz}$ as required, and the proof of \textbf{Case (i)} is finished.
 
\subsubsection{Proof of Case (ii): $\lz_1 < \caz\lz_2$} \label{SSS} 
Note that $\lz_2 \sim_\az |\lz|$ and recall  \eqref{XBde2}. Following a similar line of
reasoning,  one can examine under the condition where $\az+m\ge2$, it holds that $| \cK(\lz) |\lsim |\lz|^{-1-\sz} $. We  sketch the main steps, and the details will be provided only when major differences arise.

\textit{Step 1:} Performing integration by parts with  \eqref{ff12} we have 
\begin{align}
    \cK(\lambda) & =\frac{1}{2} \lambda_2^{-1} \lambda_1 \int_0^1 J_{\frac{m}{2}}\left(\lambda_2 r\right) J_{ \frac{q}{2}-1}\left(\lambda_1\left(1-r^{\frac{\alpha}{2}}\right)^{\frac{1}{\alpha}}\right)\left(1-r^{\frac{\alpha}{2}}\right)^{\frac{q+2}{2 \alpha}-1} r^{\frac{\alpha+m}{2}-1} d r \nonumber\\
    & =\int_0^1 \cdots(1-\varphi) d r+\int_0^1 \cdots \varphi d r \nonumber\\
    &=: \cK_1(\lambda)+ \cK_2(\lambda)  \label{K1new} .
\end{align}
 Thus, it remains to show that
 \begin{equation*}  
     |\cK_1(\lambda)|\lsim \lz_2^{-1-\sz},\qquad |\cK_2(\lambda)|\lsim \lz_2^{-1-\sz}.
 \end{equation*}
 
\textit{Step 2:} To dominate  $\cK_1$, we split the integral into two parts 
\begin{equation}\label{K12new} 
    \cK_1(\lambda)=\int_0^\delta\cdots+\int_\delta^1=: \cK_{1,1}(\lz)+ \cK_{1,2}(\lz) 
\end{equation}
with $ \delta=\underline{c}\, \lambda_2^{-\frac{1}{\az+m-1}}<1$, where $\underline{c}<(2\pi c_2)^{1/(\az+m-1)}$ is a  positive constant and $c_2$ is assumed as in Lemma  \ref{lem25}. The estimation of $\cK_{1,1}$ can also be established by  \eqref{ff05}-\eqref{ff06} immediately, where the condition $\az+m\ge2$ takes effect. In fact,
\begin{align} 
    \left|\cK_{1, 1}(\lambda)\right| & \lesssim \lambda_2^{-1} \lambda_1 \cdot \lambda_1^{-\frac{1}{2}} \left(\int_0^{\lambda_2^{-1}} \lambda_2^{\frac{m}{2}} r^{\frac{m}{2}} r^{\frac{\alpha+m}{2}-1} d r+\lambda_2^{-\frac{1}{2}} \int_{\lambda_2^{-1}}^\delta r^{\frac{\alpha+m}{2}-\frac{3}{2}} d r\right) \nonumber\\
    & \lesssim \lambda_2^{-1}  \lambda_1^{\frac{1}{2}}\left(\lambda_2^{\frac{m}{2}} \lambda_2^{-\frac{\alpha+2 m}{2}}+\lambda_2^{-\frac{1}{2}} \cdot \delta^{\frac{\alpha+m-1}{2}}\right)  \label{K11E} \\
    &\lesssim \lambda_2^{-\frac{3}{2}}.\nonumber
\end{align}

In order to assess $\cK_{1,2}$,  we use  \eqref{ff11}  and obtain
\begin{equation*}
    \begin{aligned}
    \left|\cK_{1 , 2}(\lambda)\right| &\lesssim \lambda_2^{-1} \lambda_1 \cdot \lambda_1^{-\frac{1}{2}} \lambda_2^{-\frac{1}{2}}\left[\sum_\pm\left|\int_\delta^1 e^{i \psi_\lz^{ \pm}(r)}\left(1-r^{\frac{\alpha}{2}}\right)^{\frac{q+1}{2 \alpha}-1} r^{\frac{\alpha+m}{2}-\frac{3}{2}}(1-\varphi) d r\right|\right. \\
    &\qquad+\lambda_2^{-1} \int_\delta^1\left(1-r^{\frac{\alpha}{2}}\right)^{\frac{q+1}{2 \alpha}-1} r^{\frac{\alpha+m}{2}-\frac{5}{2}}(1-\varphi) d r \\
    &\qquad\left.+\lambda_1^{-1} \int_\delta^1\left(1-r^{\frac{\alpha}{2}}\right)^{\frac{q-1}{2 \alpha}-1} r^{\frac{\alpha+m}{2}-\frac{3}{2}}(1-\varphi) d r\right].
    \end{aligned}
\end{equation*}
The sum of second and third terms can be bounded by $\lz_2^{-\frac32}\lz_1^\frac12 (\lz_2^{-\frac12} + \lz_1^{-1}) \lsim \lz_2^{-\frac32}.$ For the first term (denoted by $FI(\lz)$), from the van der Corput Lemma with Lemma \ref{cond2} it follows that
\begin{equation} \label{f90} 
    FI(\lz)\lsim \lz_2^{-\frac32}\lz_1^\frac12 \lz_2^{-\sz} \left(\one\{\az+m\ge3\} + \dz^{\frac{\az+m-3}{2}}\one\{\az+m<3\} \right).
\end{equation}
Thus the case where $\az+m\ge3 $ is proved (using the fact that $\lz_2\gsim\lz_1$). The main difficulty lies in the remaining case that $2\le \az+m <3$. If $\az\in(0,1]$ then the desired conclusion can be drawn from a similar argument as in the estimation of $\cJ_{1,2}$ in Subsubsection \ref{prfM}, with the roles of $r_*$ replaced by $R_*$ defined as in  \eqref{R*}, and  of Lemma  \ref{cond1} (ii) replaced by Lemma \ref{cond2} (ii), respectively.  

Now focus on the case $\az\in(1,2)$. Note that $m=1$ and $\sz=1/3$. Recall the numbers $R_0'$ and $R_0''$ defined in the proof of  Lemma  \ref{cond2}, which depend only on $\az$. Then we redefine $\vhi$ given as in  \eqref{vhi} with $\frac14,\frac12$ replaced by $R_0', R_0''$ respectively, since the new points depend only on $\az$. It is obvious that the above bounds are still valid, whence we are only left with $\cK_{1,2}$ as well. Note also that $\supp (1-\vhi)\cap[0,1]\subset [0,R_0'']$, and that $R_j (j=1,2)$ defined by  \eqref{ffl3} and  \eqref{ffa1} are subject to $0<R_1\le R_0'<R_0<R_0''\le R_2 <1$ with $R_0',R_0''$ depending only $\az$, and, via a simple calculation,
\begin{equation} \label{f92} 
R_1^{-1} \sim_\az \left({\lambda_2}{\lambda_1^{-1}}\right)^{\frac{2}{2-\alpha}}.
\end{equation}                    If $R_1\le\dz $, then it holds that $\lambda_1 \lsim \lambda_2^{\frac{3 \alpha+2 m-4}{2 \alpha+2 m-2}} $ and so $\lz_2\lz_1^{-1}\to\infty$ as $\lz_2\to\infty$. Consequently,  only the situation  that $\lambda_2-\caz^{-1} \lambda_1 \geq \frac12 \lambda_2$ can happen when $\lz_2$ is sufficiently large, whence by \eqref{12R1} we can refine the decay from $\lz_2^{-\sz}  $ in  \eqref{f90}  to $\lz_2^{-1}  $.  This implies that 
$$FI(\lz)\lsim \lz_2^{-\frac32}\lz_1^\frac12 \lz_2^{-1}\dz^{\frac{\az+m-3}{2}}\lsim \lz_2^{-1-\sz} .$$
If $R_1>\dz $, notice that by  \eqref{f92},
\begin{equation} \label{f94} 
    \lambda_1 \gsim \lambda_2^{\frac{3 \alpha+2 m-4}{2 \alpha+2 m-2}}.
\end{equation}
We obtain from an integration by parts  that
\begin{equation*}
    \begin{aligned}
  \cK_{1,2}(\lz)  = & -\left.\lambda_2^{-1} J_{\frac{q}{2}}\left(\lambda_1\left(1-r^{\frac{\alpha}{2}}\right) ^{\frac{1}{\alpha}}\right)\left(1-r^{\frac{\alpha}{2}}\right)^{\frac{q}{2 \alpha}} J_{\frac{m}{2}}\left(\lambda_2 r\right) r^{\frac{m}{2}}(1-\varphi)\right|_\delta ^1 \\
    &\qquad-  \lambda_2^{-1} \int_\delta^1 J_{\frac{q}{2}}\left(\lambda_1\left(1-r^{\frac{\alpha}{2}}\right)^{\frac{1}{\alpha}}\right)\left(1-r^{\frac{\alpha}{2}}\right)^{\frac{q}{2 \alpha}} J_{\frac{m}{2}}\left(\lambda_2 r\right) r^{\frac{m}{2}} \varphi' d r \\
    &\qquad +\int_\delta^1 J_{\frac{q}{2}}\left(\lambda_1\left(1-r^{\frac{\alpha}{2}}\right)^{\frac{1}{\alpha}}\right)\left(1-r^{\frac{\alpha}{2}}\right)^{\frac{q}{2 \alpha}} J_{\frac{m}{2}-1} \left(\lambda_2 r\right) r^{\frac{m}{2}}(1-\varphi) d r.
    \end{aligned}
\end{equation*}
The first two resulting terms are seen to be controlled by $\lz_2^{-\frac32}$. For the third term (denoted by $TI(\lz)$), from  \eqref{ff11} it follows that, for $N\in\nn^*$ to be chosen,
\begin{equation*}
    \begin{aligned}
    |TI(\lz)| &\lsim  \lambda_1^{-\frac{1}{2}} \lambda_2^{-\frac{1}{2}} \left[\sum_{ \pm} \sum_{j=0}^{N-1}\left|\int_\delta^1 e^{i \psi_\lambda^{ \pm}(r)} r^{\frac{m-1}{2}}\left(1-r^{\frac{\alpha}{2}}\right)^{\frac{q-1}{2 \alpha}}\left(\lambda_1\left(1-r^{\frac{\alpha}{2}}\right)^{\frac{1}{\alpha}}\right)^{-j}(1-\varphi) d r\right|\right. \\
    &\qquad+  \int_\delta^1 r^{\frac{m-1}{2}}\left(1-r^{\frac{\alpha}{2}}\right)^{\frac{q-3}{2 \alpha}}\left(\lambda_1\left(1-r^{\frac{\alpha}{2}}\right)^{\frac{1}{\alpha}}\right)^{-N}(1-\varphi) d r   \\
    &\qquad\left.+\lambda_2^{-1} \int_\delta^1 r^{\frac{m-3}{2}}\left(1-r^{\frac{\alpha}{2}}\right)^ {\frac{q-1}{2 \alpha}}(1-\varphi) d r\right].
    \end{aligned}
\end{equation*}  
The third term is majorized by $\lz_2^{-\frac32}\lz_1^{-1}\log \lz_2 \lsim \lz_2^{-\frac32}$, if we use  \eqref{f94}; while second term is bounded by $\lz_1^{-N} \lsim \lz_2^{-\frac32}$ if we use again \eqref{f94} and choose a large $N$, which is fixed hereafter.

Next we handle the first term, denoted by $TI_1(\lz)$. It is necessary to further divide into two subcases: (I) $\lambda_2-\caz^{-1} \lambda_1 \geq \frac12 \lambda_2$ and (II) $\lambda_2-\caz^{-1} \lambda_1 < \frac12 \lambda_2$. Moreover, we split the nontrivial integration interval of $TI_1(\lz)$ as 
$$[\dz,R_0'']=[\dz,R_1)\cup[R_1,R_0''].$$
In subcase (I), for the second interval the final estimation is the same as in the case $R_1\le\dz $ and thereby done; for the first one, using the improved lower bound (by  \eqref{f92})
\begin{equation*}
    \left|\left(\psi_\lambda^{+}\right)^{\prime \prime}(r)\right|  \ge \left. \frac{\lambda_1}{4}\left(1-r^{\frac{\alpha}{2}}\right)^{\frac{1}{\alpha}-2} r^{\frac{\alpha}{2}-2}\left| r^{\frac{\az}{2}}+\alpha-2\right|\right|_{r=R_1}\gsim \lz_2\,(\lz_2\lz_1^{-1})^{\frac{2}{2-\az}},
\end{equation*}
and via the van der Corput Lemma, we have 
$$TI_1|_{[\dz,R_1)} (\lz) \lsim \lambda_1^{-\frac{1}{2}} \lambda_2^{-\frac{1}{2}}\ (\lz_2\,(\lz_2\lz_1^{-1})^{\frac{2}{2-\az}})^{-\frac12} \lsim \lz_2^{-1-\sz}, $$
by recalling $\sz=1/3$ and $\az<2$.
In subcase (II), noting that $\lz_1\sim_\az\lz_2$ and $R_1=R_0'$, then  by Lemmas  \ref{vand} and  \ref{cond2} (iii)   we obtain $TI_1|_{[R_1,R_0'']} (\lz) \lsim \lambda_1^{-\frac{1}{2}} \lambda_2^{-\frac{1}{2}}\ \lz_2^{-\sz} \sim \lz_2^{-1-\sz} $, as required; while $TI_1|_{[\dz,R_1]} (\lz) \lsim \lambda_1^{-\frac{1}{2}} \lambda_2^{-\frac{1}{2}}\ \lz_2^{-1/2} \lsim \lz_2^{-1-\sz} $ for the same reason. The domination of $\cK_1 $ is therefore finished.                                                                                                                                                                  

\textit{Step 3:} To dominate $\cK_2$, we take $\dz=\tilde{c}\,\lz_1^{-\frac1{q+1}}$ with $\tilde{c}<(2\pi c_1)^{1/(q+1)}$ a  constant and $c_1$  assumed as in Lemma  \ref{lem25}, and let $ r_\delta=\left(1-\delta^\alpha\right)^{\frac{2}{\alpha}}$. Then the integral can be further split  into the following forms:
\begin{equation*}
    \cK_2(\lambda)=\int_0^{r_\delta}\cdots+\int_{r_\delta}^1=: \cK_{2,1}(\lz)+ \cK_{2,2}(\lz). 
\end{equation*}
For $\cK_{2,2}$, we employ  \eqref{ff05}-\eqref{ff06} as before. For $\cK_{2,1}$, we use the van der Corput Lemma with Lemma \ref{cond2}, where the case that $\az>\frac{q+1}{2}$ needs more care. To that end, we can argue as in the estimation of $\cJ_2$ in Subsubsection \ref{prfM} but with $R_*$ and Lemma  \ref{cond2} (i) for $\az\ge2$; whereas for $\az\in(1,2)$, noting that $\az>3/2$ and  $\sz=1/3$, we can use   Lemma  \ref{cond2} (iii)
 instead and follow a similar line of reasoning  as in  \textit{Step 2}. So the details will not be repeated.

\subsection{Proof of \eqref{ff19}}  \label{prb3b} 
We divide the problem into two cases as well: $\lz_1\ge \caz\lz_2$ and $\lz_1 < \caz\lz_2$.
Note that the first case is already contained in  \eqref{ff18}, so only the latter case needs a further discussion, whose proof is actually a slight modification of the foregoing procedure. The difference is that, roughly speaking, under the condition where $\frac{q}{2}+1 \ge \az \ge 3-m$, the amplitude enjoys better regularities,  which will enable us to deduce improved estimates. 

Recall from  \eqref{XBde2} that $\widehat{\chi_{B_1^\alpha}}(w, s) = C_{q,m}^{-1}\,  \lz_1^{-\frac{q}2} \lz_2^{-\frac{m}2+1} \cK(\lz)$. We split $\cK(\lz)$ into two parts as in \eqref{K1new}  and presently  aim to demonstrate that, for sufficiently large $\lz_2 (\sim_\az |\lz|)$,
 \begin{equation}   \label{goalF} 
 |\cK_{1}(\lambda)|\lsim \lz_2^{-\frac32-\sz}\lz_1^\frac12
     \qquad \mbox{and}\qquad  |\cK_2(\lambda)|\lsim \lz_2^{-\frac32-\sz}\lz_1^\frac12,
 \end{equation}
respectively, under the conditions that $\az+m\ge3$ and $\az\le\frac{q}2+1$.

\subsubsection{Estimation of $\cK_1$}
Let $\az+m\ge3$. Retain the notation  as in  \eqref{K12new} subsequently. By  \eqref{K11E} and the assumption we obtain that $|\cK_{1,1}(\lambda)|\lsim \lz_2^{-2}\lz_1^\frac12$, which is better than the desired bound since $\sz\le1/2$. To assess  $\cK_{1,2}$, we use  \eqref{ff11} to find that
\begin{align*}
    \left|\cK_{1,2}(\lambda)\right| &\lsim \lz_2^{-1}\lz_1\cdot \lambda_2^{-\frac{1}{2}} \lambda_1^{-\frac{1}{2}} \left[\sum_{ \pm}\left|\int_\delta^1 e^{i \psi_\lambda^\pm(r)} r^ {\frac{\alpha+m}{2}-\frac{3}{2}}\left(1-r^{\frac{\alpha}{2}}\right)^{\frac{q+1}{2 \alpha}-1}(1-\varphi) d r\right|\right. \\
    &\qquad+\lambda_2^{-1} \int_\delta^1 r^{\frac{\alpha+m}{2}-\frac{5}{2}}\left(1-r^{\frac{\alpha}{2}}\right)^{\frac{q+1}{2 \alpha}-1}(1-\varphi) d r \\
    &\qquad\left.+\lambda_1^{-1} \int_\delta^1 r^{\frac{\alpha+m}{2}-\frac{3}{2}}\left(1-r^{\frac{\alpha}{2}}\right)^{\frac{q-1}{2 \alpha}-1}(1-\varphi) d r\right].
\end{align*}
Then, for $\az+m\ge3$, from Lemmas  \ref{vand} and  \ref{cond2}   it follows  that 
\begin{equation*}
    \left|\cK_{1,2}(\lambda)\right| \lsim \lz_2^{-\frac32}\lz_1^\frac12 \left(\lz_2^{-\sz}+\lz_2^{-\frac12} + \lz_1^{-1} \right).
\end{equation*}
Hence when $\lz_2\le\lz_1^2$ it holds that $ |\cK_1(\lambda)|\lsim \lz_2^{-\frac32-\sz}\lz_1^\frac12$. Consider now  $\lz_2>\lz_1^2$. Performing integration by parts with  \eqref{ff12}  we have 
\begin{align*}
    &2\lz_1^{-1}\lz_2\,\cK_{1,2}(\lambda)=\lambda_2^{-1} J_{\frac{m}{2}+1}\left(\lambda_2 r\right) \left. J_ {\frac{q}{2}-1}\left(\lambda_1\left(1-r^{\frac{\alpha}{2}}\right)^{\frac{1}{\alpha}}\right)\left(1-r^{\frac{\alpha}{2}}\right)^{\frac{q+2}{2 \alpha}-1} r^{\frac{\alpha+m}{2}-1} (1-\varphi)\right|_\delta ^1 \nonumber\\
    &\quad +\lambda_2^{-1} \int_\delta^1 J_{\frac{m}{2}+1}\left(\lambda_2 r\right) J_ {\frac{q}{2}-1}\left(\lambda_1\left(1-r^{\frac{\alpha}{2}}\right)^{\frac{1}{\alpha}}\right)\left(1-r^{\frac{\alpha}{2}}\right)^{\frac{q+2}{2 \alpha}-1} r^{\frac{\alpha+m}{2}-1} \varphi' d r    \\
    & \quad+\frac{1}{2} \lambda_2^{-1} \lambda_1 \int_\delta^1 J_{\frac{m}2+1}\left(\lambda_2 r\right) J_ {\frac{q}{2}-2}\left(\lambda_1\left(1-r^{\frac{\alpha}{2}}\right)^{\frac{1}{\alpha}}\right)\left(1-r^{\frac{\alpha}{2}}\right)^{\frac{q+4}{2 \alpha}-2} r^{\frac{2 \alpha+m}{2}-2} (1-\varphi) d r   \nonumber \\
    & \quad-\lambda_2^{-1} \int_\delta^1 J_{\frac{m}2+1}\left(\lambda_2 r\right) J_ {\frac{q}{2}-1}\left(\lambda_1\left(1-r^{\frac{\alpha}{2}}\right)^{\frac{1}{\alpha}}\right)\left(1-r^{\frac{\alpha}{2}}\right)^{\frac{q+2}{2 \alpha}-2} r^{\frac{\alpha+m}{2}-2}\left(\frac{\alpha}{2}-2+r^{\frac{\alpha}{2}}\right)(1-\varphi) d r. \nonumber
\end{align*}
Note that $\lz_2 r\gsim 1$ on $[\dz,1]$. Thus, using the conditions that $\az+m\ge3 $ and $\lz_2>\lz_1^2$, together with  \eqref{ff06} it is easy to examine that $\left|\cK_{1,2}(\lambda)\right| \lsim \lz_2^{-2}\lz_1^\frac12$. These results therefore lead to the first estimate  in \eqref{goalF}. 

\subsubsection{Estimation of $\cK_2$}
Let $\az\le\frac{q}2+1$. Using an integration by parts with  \eqref{ff12}, we can  write
\begin{align*}
    \cK_2(\lambda) & =\frac{1}{2} \lambda_2^{-1} \lambda_1 \int_0^1 J_{\frac{m}{2}}\left(\lambda_2 r\right) J_{ \frac{q}{2}-1}\left(\lambda_1\left(1-r^{\frac{\alpha}{2}}\right)^{\frac{1}{\alpha}}\right)\left(1-r^{\frac{\alpha}{2}}\right)^{\frac{q+2}{2 \alpha}-1} r^{\frac{\alpha+m}{2}-1} \vhi(r)d r \nonumber\\
    & =\int_0^{r_\dz} \cdots  +\int_{r_\dz}^1 \cdots  d r \\
    &=: \cK_{2,1}(\lambda)+ \cK_{2,2}(\lambda)   .
\end{align*}
Here $ r_\delta=\left(1-\delta^\alpha\right)^{\frac{2}{\alpha}}$ with $\dz={c}_0\,\lz_1^{-\frac{q}{q+1}}$ and ${c}_0<(2\pi c_1)^{q/(q+1)}$ fixed.
We distinguish the following comparison relations: $\lz_2\le\lz_1^q $ and $\lz_2>\lz_1^q $.

{\bf Case (I):  $\lz_2\le\lz_1^q $.} Notice that $\supp  \vhi\cap[0,1]\subset [1/4,1]$. Then through  \eqref{ff05}-\eqref{ff06}  we obtain that
\begin{align*}
    \left|\cK_{2,2}(\lambda)\right| & \lesssim \lz_2^{-1}\lz_1\cdot \lambda_2^{-\frac{1}{2}} \left( \int_{1-\lambda_1^{-\alpha}}^1 \lambda_1^{\frac{q}{2}-1}(1-r)^{\frac{q-2}{2 \alpha}  +\frac{q+2}{2 \alpha} -1}  d r +  \int_{r_\delta}^{1-\lambda_1^{-\alpha}} \lambda_1^{-\frac{1}{2}}(1-r)^{\frac{q+1}{2 \alpha}-1} d r \right) \\
    & \lesssim \lz_2^{-\frac32}\lz_1\cdot (  \lambda_1^{\frac{q}{2}-1} \lambda_1^{-q}+  \lambda_1^{-\frac{1}{2}} \delta^{\frac{q+1}{2}}),
\end{align*}
which shows that $\left|\cK_{2,2}(\lambda)\right| \lsim \lz_2^{-2}\lz_1^\frac12$. For $ \cK_{2,1}$  we use  \eqref{ff11}  and acquire that
\begin{align*}
    \left|\cK_{2,1}(\lambda)\right| &\lsim \lz_2^{-1}\lz_1\cdot {\lambda_2}^{-\frac{1}{2}} \lambda_1^{-\frac{1}{2}}\left[\sum_{ \pm} \sum_{j=0}^{N-1}\left|\int_0^{r_\delta} e^{i \psi_\lambda^\pm(r)} r^{\frac{\alpha+m}{2}-\frac{3}{2}}\left(1-r^{\frac{\alpha}{2}}\right)^{\frac{q+1}{2 \alpha}-1}\left(\lambda_1\left(1-r^{\frac{\alpha}{2}}\right)^{\frac{1}{\alpha}}\right)^{-j} \varphi d r\right|\right. \\
    &\qquad+\lambda_2^{-1} \int_0^{r_\delta} r^{\frac{\alpha+m}{2}-\frac{5}{2}}\left(1-r^{\frac{\alpha}{2}}\right)^{\frac{q+1}{2 \alpha}-1} \varphi d r   \\
    &\qquad\left.+  \int_0^{r_\delta} r^{\frac{\alpha+m}{2}-\frac{3}{2}}\left(1-r^{\frac{\alpha}{2}}\right)^{\frac{q-1}{2 \alpha}-1}\left(\lambda_1\left(1-r^{\frac{\alpha}{2}}\right)^{\frac{1}{\alpha}}\right)^{-N} \varphi d r\right],
\end{align*}
where $N\in\nn^*$ is to be determined later. Note that under our assumption
\begin{equation} \label{LB1} 
    \lambda_1\left(1-r^{\frac{\alpha}{2}}\right)^{\frac{1}{\alpha}}\ge\lambda_1^{\frac{1}{q+1}} \gsim  \lambda_2^{\frac{1}{q^2+q}}, \quad \mbox{for all}\ r\in[0,r_\dz].
\end{equation}
Then invoking Lemmas  \ref{vand} and   \ref{cond2} produces that 
\begin{align} 
    \left|\cK_{2,1}(\lambda)\right| &\lsim \lz_2^{-\frac32}\lz_1^\frac12 \left(\sum_{j=0}^{N-1} \lambda_2^{-\frac{j}{q^2+q}} \left(\chi_{(0,\frac{q+1}{2}]}(\az)+\dz^{\frac{q+1-2\az}{2}}\chi_{(\frac{q+1}{2},\infty)}(\az)\right) \lambda_2^{-\sigma} \right. \nonumber\\
    &\qquad \left. + \lz_2^{-1} +  \lambda_2^{-\frac{N}{q^2+q}} \right),\label{ffF1} 
\end{align}
which, together with the choice that $N=\lfloor q^2+q\rfloor+1$, implies the desired estimate for $ \az\le\frac{q+1}{2}.$ For the remaining case where $ \frac{q+1}{2}<\az\le\frac{q+2}{2}$, we exploit the skills developed in the previous  subsection. And there arise two situations: $\az\ge2$ and $\az<2.$

\textit{(1) Let $\az\ge2$ and $ \frac{q+1}{2}<\az\le\frac{q+2}{2}$.} First note that $\sz=1/2$ at this point. Via  \eqref{R*} one has 
\begin{equation} \label{ff45} 
    \left(1-R_*^\frac{\az}{2}\right)^{\frac{1}{\az}-1} \sim_\az {\lambda_2}{\lambda_1^{-1}}.
\end{equation}
Hence from the proof of Lemma  \ref{cond2} (i) we may conclude that, for all $r\in(R_*,1)$,
\begin{equation} \label{ff24} 
   \left| \left(\psi_\lambda^{+}\right)^{\prime\prime}(r)\right| \ge \frac{1}{4} \lambda_1\left(1-r^{\frac{\alpha}{2}}\right)^{\frac{1}{\alpha}-2} r^{\frac{\alpha}{2}-1} \ge \frac{1}{4} \lambda_2 \left(1-R_*^{\frac{\alpha}{2}}\right)^{-1} \gsim \lz_2\left({\lambda_2}{\lambda_1^{-1}}\right)^{\frac{\az}{\alpha-1}} .
\end{equation}
Next, we further divide the problem into two subcases: $\lz_2\ge\lz_1^2$ and $\lz_2<\lz_1^2$.

Focus on the first subcase $\lz_2\ge\lz_1^2$. If $R_*\ge r_\dz$, then  the $\lz_2^{-\sz}$ decay in  \eqref{ffF1} can be improved to  $\lz_2^{-1}$ by Lemma  \ref{cond2} (i). But in this subcase 
$$ \lambda_2^{-1} \cdot \delta^{\frac{q+1-2 \alpha}{2}}\sim \lz_2^{-1} \lambda_1^{\frac{q(2 \alpha-q-1)}{2(q+1)}} \lsim \lz_2^{-\frac12},$$
whence we obtain $\left|\cK_{2,1}(\lambda)\right| \lsim \lz_2^{-2}\lz_1^\frac12$ by choosing the same $N$ as before; If $R_*< r_\dz$, then we split the integral into two parts, $(0,R_*)$ and $(R_*, r_\dz)$. The first part can be handled by  the previous line of reasoning; for the second part, the decay $\lz_2^{-\sz}$ can be improved to $\lz_2^{-\frac12}(\lz_2\lz_1^{-1})^{-\frac{\az}{2\az-2}}$
if we use the van der Corput Lemma with  \eqref{ff24}. Therefore the 
required estimate follows since it also holds  that
$$  \lz_2^{-\frac12}(\lz_2\lz_1^{-1})^{-\frac{\az}{2\az-2}} \cdot \delta^{\frac{q+1-2 \alpha}{2}}  \lsim \lz_2^{-\frac12}$$
in this subcase with our assumptions.

Focus on the second subcase $\lz_2<\lz_1^2$.  Using integration by parts we get
\begin{equation*}
    \begin{aligned}
    \cK_{2,1}(\lambda) 
    & \left.= - \lambda_2^{-1}  J_{\frac{q}{2}}\left(\lambda_1\left(1-r^{\frac{\alpha}{2}}\right)^{\frac{1}{\az}}\right)\left(1-r^{\frac{\alpha}{2}}\right)^{\frac{q}{2\az}}J_{\frac{m}{2}}\left(\lambda_2 r\right) r^{\frac{m}{2}} \varphi\right|_0 ^{r_\delta} \\
    & \qquad+  \lambda_2^{-1} \int_0^{r_\delta} J_{\frac{q}2}\left(\lz_1 \left(1-r^{\frac{\alpha}{2}}\right)^{\frac{1}{\alpha}} \right) \left(1-r^{\frac{\alpha}{2}}\right)^{\frac{q}{2 \alpha}} J_{\frac{m}{2}}\left(\lambda_2 r\right) r^{\frac{m}{2}} \varphi' d r \\
    & \qquad+ \int_0^{r_\delta} J_{\frac{q}{2}}\left(\lambda_1\left(1-r^{\frac{\alpha}{2}}\right)^{\frac{1}{\alpha}}\right)\left(1-r^{\frac{\alpha}{2}}\right)^{\frac{q}{2 \alpha}} J_{\frac{m}{2}-1}\left(\lambda_2 r\right) r^{\frac{m}{2}} \varphi d r .
    \end{aligned}
\end{equation*}
The first resulting term can be bounded by $\lz_2^{-2}\lz_1^\frac12$ if we use  \eqref{ff06}; while the second can also be controlled by this amount via Lemmas  \ref{vand} and  \ref{cond2} (i) (note that  $\supp\vhi'\subset[1/4,1/2]$ bringing no singularity to the integrand). To assess the third term, we employ  \eqref{ff11} with $N$-term and $2$-term expansions for $J_{\frac{q}{2}}\left(\lambda_1\left(1-r^{\frac{\alpha}{2}}\right)^{\frac{1}{\alpha}}\right)$ and $J_{\frac{m}{2}-1}\left(\lambda_2 r\right)$ respectively)  and find that, for $N\in\nn^*$,
\begin{equation*}
    \begin{aligned}
    \left|{\cK_{2,1}}(\lambda)\right| &\lsim   \lambda_2^{-\frac{1}{2}} \lambda_1^{-\frac{1}{2}} \left[ \sum_{ \pm} \sum_{j=0}^{N-1} \sum_{k=0}^1\left|\int_0^{r_\dz} e^{i  \psi_\lambda^\pm(r)} r^{\frac{m-1}{2}}\left(1-r^{\frac{\alpha}{2}}\right)^{\frac{q-1}{2 \alpha}}\left(\lambda_1\left(1-r^{\frac{\alpha}{2}}\right)^{\frac{1}{\alpha}}\right)^{-j}\left(\lambda_2 r\right)^{-k} \varphi d r\right| \right.\\
    &\qquad +\lambda_2^{-2} \int_0^{r_\dz} r^ {\frac{m-5}{2}}\left(1-r^{\frac{\alpha}{2}}\right)^{\frac{q-1}{2 \alpha}} \varphi d r \\
    &\qquad \left.+  \int^{r_\dz}_0 r^{\frac{m-1}{2}}\left(1-r^{\frac{\alpha}{2}}\right)^{\frac{q-3}{2 \alpha}}\left(\lambda_1\left(1-r^{\frac{\alpha}{2}}\right)^{\frac{1}{\alpha}}\right)^{-N} \varphi d r\right] .
    \end{aligned}
\end{equation*} 
Similarly, fixing large $N$ and using  \eqref{LB1}  the last two terms are  majorized by $\lz_2^{-\frac52}\lz_1^{-\frac12}  .$ Turn to the first term (denoted by $FI(\lz)$). If $R_*\ge r_\dz$, then by   Lemmas  \ref{vand} and   \ref{cond2} (i)    with  \eqref{LB1}  we have  in this subcase
$$FI(\lz)\lsim \lambda_2^{-\frac{1}{2}} \lambda_1^{-\frac{1}{2}}\cdot \lz_2^{-1}\lsim \lz_2^{-2}\lz_1^\frac12.$$ 
If $R_*< r_\dz$, we split the integration interval as  $(0,R_*)\cup(R_*,r_\dz)$. The first part can be handled by a similar argument; while the second is bounded by 
\begin{equation*}
    \lambda_2^{-\frac{1}{2}} \lambda_1^{-\frac{1}{2}}   \sum_{j=0}^{N-1} \sum_{k=0}^1 \lz_1^{-\frac{j}{q+1}} \lz_2^{-k} \lz_2^{-\frac12}(\lz_2\lz_1^{-1})^{-\frac{\az}{2\az-2}} \left(1-R_*^{\frac{\alpha}{2}}\right)^{\frac{q-1}{2 \alpha}}
\end{equation*}
via Lemmas  \ref{vand} (with  \eqref{ff24}) and  \eqref{LB1}, which is majorized by 
$$\lz_2^{-1}\lz_1^{-\frac12}\left({\lambda_2}{\lambda_1^{-1}}\right)^{- \frac{\alpha+q-1}{2 \alpha-2}}\lsim \lz_2^{-2}\lz_1^\frac12$$
 from  \eqref{ff45}. 

\textit{(2) Let $\az<2$ and $ \frac{q+1}{2}<\az\le\frac{q+2}{2}$.} Note that we have $q=2$, $\az\in(3/2,2)$ and $\sz=1/3$. The proof is similar to that of the former situation; cf. also the estimation of $\cK_{1,2}$ in \textit{Step 2} of  Subsubsection  \ref{SSS}. We  omit the details.

{\bf Case (II):  $\lz_2>\lz_1^q $.}
Recall that $q\ge2$ and we always assume that $\az\le\frac{q}{2}+1$ in this subsubsection. An integration by parts with  \eqref{ff12}  gives 
\begin{align*}
    \cK_{2}(\lambda)&= \frac12 \lz_2^{-1}\lz_1 \left[\lambda_2^{-1} J_{\frac{m}{2}+1}\left(\lambda_2 r\right) \left. J_ {\frac{q}{2}-1}\left(\lambda_1\left(1-r^{\frac{\alpha}{2}}\right)^{\frac{1}{\alpha}}\right)\left(1-r^{\frac{\alpha}{2}}\right)^{\frac{q+2}{2 \alpha}-1} r^{\frac{\alpha+m}{2}-1} \varphi\right|_0 ^1 \right. \nonumber\\
    &\quad -\lambda_2^{-1} \int_0^1 J_{\frac{m}{2}+1}\left(\lambda_2 r\right) J_ {\frac{q}{2}-1}\left(\lambda_1\left(1-r^{\frac{\alpha}{2}}\right)^{\frac{1}{\alpha}}\right)\left(1-r^{\frac{\alpha}{2}}\right)^{\frac{q+2}{2 \alpha}-1} r^{\frac{\alpha+m}{2}-1} \varphi' d r    \\
    & \quad+\frac{1}{2} \lambda_2^{-1} \lambda_1 \int_0^1 J_{\frac{m}2+1}\left(\lambda_2 r\right) J_ {\frac{q}{2}-2}\left(\lambda_1\left(1-r^{\frac{\alpha}{2}}\right)^{\frac{1}{\alpha}}\right)\left(1-r^{\frac{\alpha}{2}}\right)^{\frac{q+4}{2 \alpha}-2} r^{\frac{2 \alpha+m}{2}-2} \varphi d r   \nonumber \\
    & \left.\quad-\lambda_2^{-1} \int_0^1 J_{\frac{m}2+1}\left(\lambda_2 r\right) J_ {\frac{q}{2}-1}\left(\lambda_1\left(1-r^{\frac{\alpha}{2}}\right)^{\frac{1}{\alpha}}\right)\left(1-r^{\frac{\alpha}{2}}\right)^{\frac{q+2}{2 \alpha}-2} r^{\frac{\alpha+m}{2}-2}\left(\frac{\alpha}{2}-2+r^{\frac{\alpha}{2}}\right)\varphi dr\right] \nonumber\\ 
    & =:\sum_{j=1}^4 \cK_{2}^{(j)}(\lz).
\end{align*}
Note that  $\supp \vhi\cap[0,1]\subset[1/4,1]$. By  \eqref{ff05} we have  
\begin{equation*}
    \left| \cK_{2}^{(1)}(\lz) \right| +  \left| \cK_{2}^{(2)}(\lz) \right| \lsim \lz_2^{-1}\lz_1 \left(\lz_2^{-\frac32} + \lz_2^{-\frac32}\lz_1^{-\frac12} \right) \lsim \lz_2^{-2} \lz_1^{\frac12}.
\end{equation*}
And from \eqref{ff05}-\eqref{ff06}  one sees that 
\begin{equation*}
    \begin{aligned}
    \left| \cK_{2}^{(4)}(\lz) \right| &\lsim  \lz_2^{-1}\lz_1\cdot \lambda_2^{-1} \lambda_2^{-\frac{1}{2}} \left(\int_{1-\lambda_1^{-\alpha}}^1 \lambda_1^{\frac{q}{2}-1}\left(1-r^{\frac{\alpha}{2}}\right)^{\frac{q+2}{2 \alpha}-2+\frac{q-2}{2 \alpha}} \left|\frac{\alpha}{2}-2+r^{\frac{\alpha}{2}}\right| d r \right. \\
    &\qquad  + \left. \int_0^{1-\lambda_1^{-\alpha}}\lambda_1^{-\frac{1}{2}}\left(1-r^{\frac{\alpha}{2}}\right)^{\frac{q+1}{2 \alpha}-2} d r \right).
    \end{aligned}
\end{equation*}
The second integral in the brace is actually bounded under our assumption that $\az\le\frac{q}{2}+1$. For the first one, when $\az<q$ it has the upper bound
\begin{equation*}
    \lambda_1^{\frac{q}{2}-1} \int_{1-\lambda_1^{-\alpha}}^1\left(1-r^{\frac{\alpha}{2}}\right)^{\frac{q}{\alpha}-2} d r \sim \lambda_1^{\frac{q}{2}-1} \cdot \lambda_1^{\alpha-q} \lesssim 1 ;
\end{equation*}
alternatively when $\az\ge q$, we must have $ \az=q=2$, whence  $\frac{\alpha}{2}-2+r^{\frac{\alpha}{2}}= r^{\frac{\alpha}{2}}-1$ and the first integral then can be controlled by 
\begin{equation*}
    \lambda_1^{\frac{q}{2}-1}  \int_{1-\lambda_1^{-\alpha}}^1\left(1-r^{\frac{\alpha}{2}}\right)^{-1} \cdot\left(1-r^{\frac{\alpha}{2}}\right) d r \lesssim 1 .
\end{equation*}
This implies that $ \left| \cK_{2}^{(4)}(\lz) \right| \lsim \lz_2^{-2} \lz_1^{\frac12}$. For $\cK_{2}^{(3)}$, using  \eqref{ff05}-\eqref{ff06} again we obtain that 
\begin{align*}
  \left| \cK_{2}^{(3)}(\lz) \right| &\lsim  (\lambda_2^{-1} \lambda_1)^2 \cdot \lambda_2^{-\frac{1}{2}}\left(\int_{1-\lambda_1^{-\alpha}}^1\left|J_{ \frac{q}{2}-2 } \left(\lambda_1\left(1-r^{\frac{\alpha}{2}}\right)^{\frac{1}{\alpha}}\right)\right|\left(1-r^{\frac{\alpha}{2}}\right)^{\frac{q+4}{2 \alpha}-2} d r \right.\\
  &\qquad \left.+\int_0^{1-\lambda_1^{-\alpha}} \lambda_1^{-\frac{1}{2}}(1-r^\frac{\az}2)^{\frac{q+3}{2 \alpha}-2} d r\right).
\end{align*}
The second term in the brace is majorized by $\lambda_1^{-\frac{1}{2}}$. For the first one, when  $\az<q$ it has the upper bound 
\begin{equation*}
    \int_{1-\lambda_1^{-\alpha}}^1 \lambda_1^{\frac{q-4}{2}}\left(1-r^{\frac{\alpha}{2}}\right)^{\frac{q-4}{2 \alpha}+\frac{q+4}{2 \alpha}-2} d r \sim \lambda_1^{\frac{q-4}{2}-q+\alpha} \lesssim \lambda_1^{-1},
\end{equation*}
since $\az\le\frac{q}{2}+1$; when $\az\ge q$ (so $\az=q=2$), noting that $J_{-1}= -J_{1}$ via  \eqref{ff13}, thus it can be bounded by
\begin{equation*}
    \int_{1-\lambda_1^{-2}}^1 \lambda_1(1-r)^{\frac{1}{2}} \cdot(1-r)^{-\frac{1}{2}} d r \lesssim \lambda_1^{-1} .
\end{equation*}
These estimates and the assumption $\lz_2>\lz_1^q$ then lead to that 
$$ \left| \cK_{2}^{(3)}(\lz) \right| \lsim  (\lambda_2^{-1} \lambda_1)^2 \cdot \lambda_2^{-\frac{1}{2}} \lz_1^{-\frac{1}{2}} \le \lz_2^{-2} \lz_1^{\frac12}.$$
Finally, we have verified  \eqref{goalF} and completed the proof of Lemma  \ref{lem25}.

\phantomsection\addcontentsline{toc}{section}{Acknowledgements}\section*{Acknowledgements}
This work is partially supported by the China Postdoctoral Science
Foundation (Grant No.~2026M793367), and the National Natural Science Foundation of China (Grant No. 12431006). The author would like to thank Prof. Sibei Yang for the kind support and suggestions,  and Ye Zhang for the useful remarks on nilpotent Lie groups. 

\vspace{.5cm}
\begin{small}
\noindent{\bf Data Availability}\quad Not applicable. No datasets were generated or analysed during the current study.
\end{small}

\section*{Declarations}
\begin{small}
\noindent{\bf Conflict of interest}\quad The author declares that there is no conflict of interest.
\end{small}

\phantomsection\addcontentsline{toc}{section}{References}
\bibliographystyle{abbrv}
\bibliography{LPCP2604}

\vspace{1cm}\mbox{}\\\vspace{.2cm}\noindent
Sheng-Chen Mao (\textit{Corresponding author})\\
School of Mathematics and Statistics\\
Lanzhou University \\
No. 222 Tianshui South Road\\
Lanzhou 730000, P.R. China \\
\noindent
\begin{tabular}{@{}ll@{}}
{\it E-Mails:}&{\ttfamily maoshengchen@lzu.edu.cn; maosci@163.com }
\end{tabular}

\end{document}

%% file: MyOperators.tex
\def\bbf{{\mathbb F}}
\def\bbg{{\mathbb G}}
\def\ii{{\mathbb{I}}}
\def\jj{{\mathbb J}}

\def\nn{{\mathbb N}}
\def\oo{{\mathbb O}}
\def\rr{{\mathbb R}}
\def\uu{{\mathbb U}}

\def\zz{{\mathbb Z}}

\def\bB{{\mathbf B}}

\def\bS{{\mathbf S}}

\def\cB{\mathcal{B}}

\def\cE{\mathcal{E}}

\def\cI{\mathcal{I}}
\def\cJ{{\mathcal J}}
\def\cK{{\mathcal K}}
\def\cM{\mathcal{M}}
\def\cN{\mathcal{N}}

\def\pol{\mathrm{pol}}

\def\az{\alpha}
\def\bz{\beta}
\def\dz{\delta}
\def\ez{\epsilon}

\def\lz{\lambda}

\def\Oz{\Omega}
\def\tz{\theta}
\def\vez{\varepsilon}
\def\sz{\sigma}
\def\gz{\gamma}
\def\Gz{\Gamma}

\def\T{{\mathrm{T}}}

\def\X{{\mathrm{X}}}

\def\lsim{\lesssim}
\def\gsim{\gtrsim}

\def\bS{\mathbf{S}}

\def\hh{{\mathbb{H}}}

\def\fg{\mathfrak{g}}

\def\vol{\operatorname{vol}}
\def\supp{\operatorname{supp}}
\def\diag{\operatorname{diag}}

\def\fall{\forall\,}
\def\one{\mathbf{1}}
\def\pt{\partial}
\def\trp{^\mathsf{T}}
\def\vhi{\varphi}
\def\bsl{\backslash}

\def\caz{\mathbf{C_\az}}
\def\bdz{\boldsymbol{\delta}}